\documentclass[11pt]{amsart}
\usepackage[margin=1in]{geometry}
\usepackage{latexsym}

\usepackage{libertine}
\usepackage[utf8]{inputenc}
\usepackage{amsfonts}
\usepackage{amsmath}
\usepackage{amssymb}
\usepackage{amsthm}
\usepackage{commath}
\usepackage{subfig}
\usepackage{float}
\usepackage{mathrsfs}
\usepackage{appendix}
\usepackage[colorlinks, breaklinks=true]{hyperref}
\usepackage[noabbrev,capitalise,nameinlink]{cleveref}
\hypersetup{
	linkcolor=blue,
	citecolor=Green,
	filecolor=Red,
	urlcolor=NavyBlue,
	menucolor=Red,
	runcolor=Red}
\usepackage{thmtools}
\usepackage{thm-restate}
\usepackage{cleveref}
\numberwithin{equation}{section}
\allowdisplaybreaks

\usepackage[dvipsnames]{xcolor}

\usepackage[style=alphabetic, maxbibnames=9, giveninits]{biblatex}
\renewbibmacro{in:}{}

\usepackage{todonotes}
\usepackage{amsmath}

\usepackage{enumitem}
\newlist{thmparts}{enumerate}{1}
\setlist[thmparts]{label=(\roman*), ref=\thetheorem(\roman*)}
\crefname{thmpartsi}{Theorem}{Theorems}
\Crefname{thmpartsi}{Theorem}{Theorems}

\newlist{lemparts}{enumerate}{2}
\setlist[lemparts]{label=(\roman*), ref=\thelemma(\roman*)}
\crefname{lempartsi}{Lemma}{Lemmas}
\Crefname{lempartsi}{Lemma}{Lemmas}

\newlist{corparts}{enumerate}{3}
\setlist[corparts]{label=(\roman*), ref=\thecorollary(\roman*)}
\crefname{corpartsi}{Corollary}{Corollaries}
\Crefname{corpartsi}{Corollary}{Corollaries}

\newcommand{\R}{\mathbb{R}}
\newcommand{\E}{\mathbb{E}}
\newcommand{\Var}{\textup{Var}}
\newcommand{\Cov}{\textup{Cov}}
\newcommand{\Pois}{\textup{Po}}
\newcommand{\TV}{\textup{TV}}
\newcommand{\Prob}{\mathbb{P}}

\newcommand{\IH}{\mathsf{IH}}

\newcommand{\rev}{\textup{rev}}

\newcommand{\Z}{\mathbb{Z}}
\newcommand{\N}{\mathbb{N}}

\newcommand{\lcm}{\textup{lcm}}
\newcommand{\Bin}{\textup{Bin}}
\newcommand{\Aut}{\textup{Aut}}

\DeclareFieldFormat[article]{title}{#1}

\newtheorem{theorem}{Theorem}[section]
\newtheorem{lemma}[theorem]{Lemma}
\newtheorem{proposition}[theorem]{Proposition}
\newtheorem{corollary}[theorem]{Corollary}

\theoremstyle{definition}
\newtheorem{definition}[theorem]{Definition}
\newtheorem{remarkx}[theorem]{Remark}
\newenvironment{remark}
  {\pushQED{\qed}\remarkx}
  {\popQED\endremarkx}

\title{Phase Transitions for Sparse Random Sets Under Linear Forms}

\author{Ryan Jeong}
\address{Department of Statistics, Stanford University}
\email{rsjeong@stanford.edu}

\author{Steven J. Miller}
\address{Department of Mathematics and Statistics, Williams College, Williamstown, MA 01267}
\email{sjm1@williams.edu}

\begin{document}

\begin{abstract}
    Let $A \subseteq \{0,1,\dots,N\}$ be a random set in which each element is included independently with probability $p=p(N)$. Fix an integer $h \geq 2$ and a linear form
    \begin{align*}
        L(x_1,\dots,x_h) := u_1x_1 + \cdots + u_hx_h.
    \end{align*}
    We study the random image set
    \begin{align*}
        L(A) = \left\{ L(a_1,\dots,a_h) : a_i \in A \right\},
    \end{align*}
    inside the feasible interval of values of $L$ on $\{0,1,\dots,N\}^h$, as well as the associated representation counts. Our results exhibit two distinct threshold scales. First, there is a \emph{global} transition at $p(N) \asymp N^{-(h-1)/h}$ governing the size of $L(A)$: below this scale collisions are rare and $L(A)$ is sparse, while above it $L(A)$ contains nearly all feasible values. We give sharp asymptotics for the size of $L(A)$ in all regimes, including inside the critical window. Second, there is a \emph{local} transition at $p(N)\asymp N^{-(h-2)/(h-1)}$ governing multiplicities: for typical values in the bulk, the number of essentially distinct representations is asymptotically Poisson below this scale, and Poisson behavior fails above it. For $h \geq 3$ these scales are separated, yielding a regime in which $L(A)$ is already globally close to full while local multiplicities remain approximately Poisson. Our framework subsumes the classical sumset and difference-set models, as well as generalized sumsets of the form $sA-dA$, as special cases. Notably, after correcting its formulation, our global theorem settles the 2009 threshold conjecture of Hegarty-Miller \cite[Conjecture 4.2]{hegarty2009almost} on the behavior of these random images.
\end{abstract}

\maketitle

\section{Introduction} 

\subsection{Motivation} \label{subsec:motivation}

Many threshold phenomena in probabilistic combinatorics follow a common paradigm: one starts with a random discrete input, governed by a density (or intensity) parameter, applies a simple deterministic transformation, and asks how the typical output changes as the underlying parameter varies. Even very simple maps can produce sharp thresholds, and different features of the same random object may appear at different scales. The Erd\H{o}s--R\'enyi random graph $G(n,p)$ is a standard example: as $p$ increases, global properties such as the emergence of a giant component and connectivity occur around specific thresholds, while local statistics such as fixed subgraph counts have their own threshold windows and limit laws (e.g., see \cite{janson2011random}). A different classical example comes from random $k$-SAT, where a random formula suddenly becomes typically unsatisfiable as the clause density increases (e.g., see  \cite{friedgut1999sharp}); here a global feasibility property is tied to the accumulation and overlap of many local constraints (as one might study, for instance, in the clause-variable factor graph).

Viewed this way, it is often helpful to distinguish between macroscopic and microscopic questions. At the macroscopic level one asks about global structure or coverage: does the random object percolate, connect, span, or cover? At the microscopic level one studies local counts and the sources of dependence: the appearance of small substructures, collisions, overlaps, or local obstructions that accumulate as the density increases. Clarifying how these macroscopic and microscopic behaviors interact, and whether they are governed by the same scale of the density parameter, is a recurring issue across threshold problems.

The present paper studies these questions in an additive combinatorics setting. Specifically, we let $A \subseteq \{0,1,\dots,N\}$ be a $p$-random set and we fix an integer linear form $L: \Z^h \to \Z$. This framework subsumes classical sumset/difference set models and generalized sumset models as special cases. We investigate the image $L(A)$ and the associated representation counts (i.e., for a feasible value $k$, the number of injective $h$-tuples $(a_1, \dots, a_h) \in A^h$, modulo natural symmetries, with $L(a_1, \dots, a_h) = k$). Our emphasis is on global coverage of the feasible interval (equivalently, the size of $L(A)^c$) and on local multiplicity statistics for typical values in the bulk. Our main results identify two threshold scales:
\begin{enumerate}
    \item[(i)] a global transition at $p(N)\asymp N^{-(h-1)/h}$ for the size of $L(A)$ and its complement (inside the feasible interval), together with sharp asymptotics in all regimes, including inside the critical window;

    \item[(ii)] a second threshold at $p(N)\asymp N^{-(h-2)/(h-1)}$ governing local multiplicity, below which bulk representation counts admit a Poisson approximation and above which Poisson behavior fails.
\end{enumerate}
For $h \geq 3$ these scales are separated, leaving an intermediate regime in which $L(A)$ is already macroscopically close to saturation while typical bulk representation counts remain well-approximated by a Poisson law (even though the expectation may diverge).

Notably, our global results imply the threshold picture conjectured in 2009 by Hegarty-Miller \cite[Conjecture 4.2]{hegarty2009almost}. A key new ingredient in our analysis is a local theory of representation multiplicities: we establish Poisson behavior in the bulk together with a quantitative description of its breakdown. This local perspective does not appear in \cite{hegarty2009almost} and is what enables our high-dimensional generalization.

Finally, we note that this line of work was partly motivated by the study of \textit{more sums than differences} (MSTD) sets in additive number theory, i.e., finite sets $A \subset \Z$ satisfying $|A+A|>|A-A|$. Conway first observed in the 1960s that such sets exist via the explicit example $\{0, 2, 3, 4, 7, 11, 12, 14\}$. Interest in MSTD sets was revived by Nathanson's survey \cite{nathanson2006problems}, which suggested that, under an appropriate notion of random sampling, MSTD sets should be rare. Subsequent work generalized these notions to arbitrary linear forms \cite{nathanson2007binary}, showed that MSTD sets occur with positive probability in the dense model \cite{martin2006many} and that this phenomenon extends to high-dimensional settings \cite{iyer2012generalized}, whereas \cite[Theorem 1.1]{hegarty2009almost} establishes that MSTD sets occur with vanishing probability in the sparse model. From this perspective, our results extend the sparse model picture of \cite{hegarty2009almost} to high-dimensional settings, supplying the missing component in this narrative.

\subsection{Notation and Conventions} \label{subsec:notation}

In this subsection, we introduce the notation, conventions, and basic quantities used throughout the paper.

\subsubsection{Linear forms, images, and complements} \label{subsubsec:lin_forms_notation}

For a natural number $N$ (we will include $0$ in the natural numbers), we denote $I_N = \{0, \dots, N\}$, and if $N$ is positive, we denote $[N] := \{1, \dots, N\}$. For a positive integer $h$, a \textit{linear form in $h$ variables} is a function $L: \Z^h \to \Z$ of the form
\begin{align*}
    & L(x_1, \dots, x_h) = u_1x_1 + \cdots + u_hx_h; & u_i \in \Z_{\neq 0} \text{ for all } i \in [h].
\end{align*}
Throughout this paper, we understand intervals to be their intersections with $\Z$. Say we are given a linear form $L$ in $h$ variables with coefficients $u_1, \dots, u_h \in \Z_{\neq 0}$. We define
\begin{gather*}
    h_{\text{pos}} = h_{\text{pos}}^{(L)} := \left|\{i \in [h]: u_i > 0\}\right|; \qquad \qquad \qquad
    s = s^{(L)} := \sum_{i: u_i > 0} u_i; \\
    d = d^{(L)} := \sum_{j: u_j < 0} |u_j|; \qquad \qquad \qquad
    m = m^{(L)} := s + d.
\end{gather*}
Without loss of generality, we always assume that 
\begin{align*}
    u_1 \geq \cdots \geq u_{h_{\text{pos}}} > 0 > u_{h_{\text{pos}}+1} \geq \cdots \geq u_h.
\end{align*}
If we are also given a subset $A \subseteq I_N$, we let
\begin{align} \label{eq:L(A)}
    L(A) := \left\{u_1a_1 + \cdots + u_ha_h : a_i \in A \right\}.
\end{align}
From these definitions, it follows that
\begin{align} \label{eq:L(A)_range}
    L(A) \subseteq \left[ -dN, sN \right].
\end{align}
As such, we will define the \textit{complement of $L(A)$} to be
\begin{align} \label{eq:L(A)_complement}
    L(A)^c := \left[ -dN, sN \right] \setminus L(A).
\end{align}
Thus, $|L(A)^c|$ is $|L(A)|$ subtracted from the maximum possible value of $|L(A)|$. Finally, the setting where
\begin{align} \label{eq:balanced_cond}
    \left[u_1, \dots, u_h \right] = \left[-u_1, \dots, -u_h \right]
\end{align}
as multisets will occasionally need to be treated separately. We call linear forms whose coefficients satisfy \eqref{eq:balanced_cond} \textit{balanced}. With the conventions and notation that we have established, it follows for a balanced linear form that $h$ is even, $h_+ = h/2$, $s = d = m/2$, and $u_i = -u_{h-i+1}$ for all $i \in [h]$.

\subsubsection{Injective solutions modulo redundancies} \label{subsubsec:solutions}

We let $\mathfrak{S}_h$ denote the symmetric group on $h$ elements, and we let $\sigma_\rev \in \mathfrak{S}_h$ denote the reversal permutation given by $\sigma_\rev(i) = h-i+1$ for all $i \in [h]$. The symmetric group $\mathfrak S_h$ acts on $\mathbb Z^h$ from the right by permuting coordinates via 
\begin{align*}
    (u_1,\dots,u_h) \cdot \sigma =(u_{\sigma(1)},\dots,u_{\sigma(h)}).
\end{align*}
With respect to this action, we define
\begin{align*}
    \Aut(L) := \text{Stab}_{\mathfrak S_h}\left((u_1,\dots,u_h)\right) = \{\sigma\in\mathfrak S_h : (u_{\sigma(1)},\dots,u_{\sigma(h)}) = (u_1,\dots,u_h) \} \leq \mathfrak S_h.
\end{align*}
In particular, $\Aut(L)$ consists precisely of permutations that only permute indices among equal coefficients. For $k \in [0, mN]$, we define the \textit{redundancy subgroup} corresponding to the value $-dN + k$ (capturing all redundancies amongst $h$-tuples related to permuting the order in which we pass elements through $L$) by\footnote{If $L$ is a balanced linear form, then $\Aut(L)$ fails to capture all such redundancies since
\begin{align*}
    L(a_1, \dots, a_h) = 0 \iff L\left(a_{\sigma_\rev(1)}, \dots, a_{\sigma_\rev(h)}\right) = 0,
\end{align*}
demonstrating the sense in which balanced linear forms serve as an edge case in this article.}
\begin{align*}
    \mathscr{R}_k = \mathscr{R}_k^{(L,N)} := \begin{cases}
        \left\langle \Aut(L), \sigma_{\mathrm{rev}} \right\rangle & \text{if $L$ is balanced and $k=mN/2$}; \\
        \Aut(L) & \text{otherwise}.
    \end{cases}
\end{align*}
Since it will suffice to focus on local representations with $h$ distinct inputs in our analysis, we write
\begin{align*}
    I_N^{\underline h} := \{(a_1, \dots, a_h)\in I_N^h: a_1, \dots, a_h \text{ are pairwise distinct}\}.
\end{align*}
We now define the set of \textit{$L$-expressions}, i.e., representations modulo these symmetries for the candidate value $-dN+k$, as the orbit space
\begin{align} \label{eq:orbits}
    \mathscr{E}_k = \mathscr{E}_k^{(L,N)} := \left\{ (a_1, \dots, a_h)\in I_N^{\underline h}: L(a_1, \dots, a_h) = -dN + k \right\} \Big/ \mathscr R_k.
\end{align}
Equivalently, we identify two injective solutions if one is obtained from the other via a permutation in $\mathscr{R}_k$. Each $\Lambda \in \mathscr{E}_k$ is an $\mathscr{R}_k$-orbit of injective solutions, and we define its \emph{$L$-evaluation} by
\begin{align*}
    L(\Lambda) := -dN+k,
\end{align*}
and its \emph{ground set} by
\begin{align*}
    S(\Lambda) := \{a_1,\dots,a_h\} \subseteq I_N.
\end{align*}
This is well-defined since $\mathscr{R}_k$ acts by permuting coordinates. Under the injectivity constraint we always have $|S(\Lambda)| = h$. The action of $\mathscr{R}_k$ on the solution set in \eqref{eq:orbits} is free, so every orbit has size $|\mathscr{R}_k|$.

\subsubsection{Random model and probabilistic notation}

This article studies what happens when we apply linear forms to \textit{$p(N)$-random subsets} $A \subseteq I_N$. Specifically, we independently include each element from $I_N$ in $A$ with probability $p(N): \N \to (0,1)$, which we take to be a function satisfying the condition
\begin{align} \label{eq:p_assumptions}
    p(N) \ll 1 \ll Np(N).
\end{align}
Formally, we set $\Omega_N := \{0, 1\}^{N+1}$, so subsets $A \subseteq I_N$ and elements $(a_0, \dots, a_N) \in \Omega_N$ correspond in the natural way (that is, for all $i \in I_N$, $a_i = 1$ corresponds to $i \in A$). We thus refer to $A$ and $(a_0, \dots, a_N)$ interchangeably. We let $\Prob_N$ denote the product of $N+1$ instances of the Bernoulli measure with parameter $p$, so that we are (principally) working in the product space $\big(\Omega, \mathcal F, \Prob \big) = \big(\Omega_N, 2^{\Omega_N}, \Prob_N \big)$.

For probability measures $\mu_1$ and $\mu_2$ defined on the same space $(\Omega, \mathcal F)$, we use $d_\TV(\mu_1, \mu_2)$ to denote the total variation distance between $\mu_1$ and $\mu_2$, which we recall is defined by
\begin{align*}
    d_\TV(\mu_1, \mu_2) := \sup_{E \in \mathcal F} \left|\mu_1(E) - \mu_2(E) \right|.
\end{align*}
We now introduce some random variables and processes which will be of central importance throughout the remainder of the article. For an $L$-expression $\Lambda \in \mathscr{E}_k$, we let $X_\Lambda$ denote the Bernoulli random variable corresponding to the event $S(\Lambda) \subseteq A$. So in particular, there exists $\Lambda \in \mathscr{E}_k$ for which $X_\Lambda = 1$ if and only if $-dN + k \in L(A)$. We furthermore have that 
\begin{align*}
    p_\Lambda := \E[X_\Lambda] = p^{|S(\Lambda)|}.
\end{align*}
For a real number $\lambda > 0$, we let $\Pois(\lambda)$ denote the Poisson law with parameter $\lambda$. For a measure $\pi$ defined on $\mathcal B_{[0,1]}$, the Borel $\sigma$-field of $[0,1]$, we let $\Pois(\mathbf{\pi})$ denote the law of a Poisson process with intensity $\mathbf{\pi}$.

\subsubsection{The Irwin-Hall density and the critical profile}

The asymptotic formulae that we derive in \cref{thm:Z_linear_forms_ii}, which characterize the critical window for our global threshold picture, involve the density of the Irwin-Hall distribution of order $h$ (e.g., see \cite{feller1991introduction, laplace1814theorie}). This is the distribution of the sum of $h$ independent random variables uniformly distributed on $[0,1]$. We denote the density of the Irwin-Hall distribution of order $h$ at a real number $x$ by $\IH_h(x)$. Letting $x_+ := \max\{0, x\}$ denote the positive part of the real number $x$, this density is given for $h \geq 2$ by
\begin{align} \label{eq:irwin_hall_density}
    \IH_h(x) := \begin{cases}
        \frac{1}{(h-1)!}\sum_{j=0}^h (-1)^j \binom{h}{j}(x-j)_+^{h-1} & x \leq h, \\
        0 & x > h.
    \end{cases}
\end{align}
For $h=1$, we define $\IH_1(x) := \mathbf 1_{[0,1]}(x)$. We make use of the Irwin-Hall density function primarily via the ``critical profile function"
\begin{align} \label{eq:critical_regime_exponential_term}
    \Phi_L(x) := \frac{1}{|\Aut(L)| \cdot \prod_{i=1}^h |u_i|} \sum_{t_1=0}^{|u_1|-1} \cdots \sum_{t_h=0}^{|u_h|-1} \IH_h\left( x - \sum_{i=1}^h t_i \right),
\end{align}
which dictates the size of $L(A)$ in the critical window.

\subsubsection{Global assumptions and notation}

We now record those conventions that we assume, unless stated otherwise, throughout the rest of the paper. Throughout the analysis, we assume that we have fixed an arbitrary integer $h \geq 2$ and a linear form $L: \Z^h \to \Z$ with coefficients $u_1, \dots, u_h \in \Z_{\neq 0}$ satisfying $\gcd(u_1, \dots, u_h) = 1$. We employ standard asymptotic notation throughout the article: for the sake of completeness, we describe this notation in \cref{sec:asymptotic_notation}. We also assume that $N$ is a large positive integer which tends to infinity, and that all asymptotic notation is with respect to $N$. We omit floor and ceiling symbols when doing so does not affect asymptotics. In \cref{sec:preliminaries,sec:computations,sec:fast_decay,sec:critical_decay,sec:slow_decay,sec:poisson_convergence}, we will assume that \eqref{eq:p_assumptions} holds, and frequently abbreviate $p(N)$ to $p$. Throughout, we use $C > 0$ to denote a positive constant whose value may change from line to line, but depends only on fixed parameters (e.g., $L$ and $h$) unless explicitly indicated otherwise.

\subsection{Main Results} \label{subsec:main_results}

Our primary interest is to understand what happens to $|L(A)|$ and $|L(A)^c|$ as we vary the rate at which the inclusion probability $p(N)$ decays as $N$ grows. Our techniques yield two broad classes of results.

\subsubsection{Global Phase Transition}
At the global scale, we observe the following phase transition concerning the sizes of $L(A)$ and its complement.

\begin{theorem} \label{thm:Z_linear_forms}
    Let $p : \N \to (0,1)$ be a function satisfying \eqref{eq:p_assumptions}. Fix an integer $h \geq 2$ and a linear form $L: \Z^h \to \Z$ with coefficients $u_1, \dots, u_h \in \Z_{\neq 0}$ such that $\gcd(u_1, \dots, u_h) = 1$. Let $A \subseteq I_N$ be a random subset where each element of $I_N$ is independently included in $A$ with probability $p(N)$. The following three situations arise.
    \begin{thmparts}
        \item \label{thm:Z_linear_forms_i}
        
        If $p(N) \ll N^{-\frac{h-1}{h}}$, then
        \begin{align} \label{eq:fast_decay}
            |L(A)| \sim \frac{\left( N \cdot p(N) \right)^h}{|\Aut(L)|}.
        \end{align}
        
        \item \label{thm:Z_linear_forms_ii}
        
        If $p(N) = cN^{-\frac{h-1}{h}}$ for some constant $c > 0$, then
        \begin{align} \label{eq:critical_decay}
            & |L(A)| \sim \left(\sum_{i=1}^h |u_i| - 2\int_0^{m/2} e^{-c^h \Phi_L(x)} dx\right)N;
            & |L(A)^c| \sim \left(2\int_0^{m/2} e^{-c^h \Phi_L(x)} dx \right) N.
        \end{align}
        
        \item \label{thm:Z_linear_forms_iii}
        
        If $p(N) \gg N^{-\frac{h-1}{h}}$, then (where $\Gamma(\cdot)$ denotes the Gamma function)
        \begin{align} \label{eq:slow_decay}
            |L(A)^c| \sim \frac{2 \cdot \Gamma\left( \frac{1}{h-1} \right) \sqrt[h-1]{(h-1)! \cdot |\Aut(L)| \cdot \prod_{i=1}^{h} |u_i|}}{(h-1) \cdot p(N)^{\frac{h}{h-1}}}.
        \end{align}
    \end{thmparts}
\end{theorem}

\cref{thm:Z_linear_forms} identifies $p(N) = N^{-\frac{h-1}{h}}$ as a threshold function, separating a subcritical regime in which $L(A)$ behaves like a near-injective image from a supercritical regime in which $L(A)$ is macroscopically saturated. Indeed, the number of essentially distinct $L$-expressions generated by $A$ is $\sim |A|^h/|\Aut(L)| \sim (Np)^h/|\Aut(L)|$, while the ambient range of $L(A)$ has cardinality $\asymp mN$ by \eqref{eq:L(A)_range}. Thus, if we are below the threshold, it follows that $L(A)$ is asymptotically almost surely ``basically Sidon," i.e., that almost all pairs of such $L$-expressions generated in $L(A)$ correspond to distinct values. If we are above the threshold, it follows from $p(N)^{-\frac{h}{h-1}} \ll N$ that asymptotically almost surely, almost all possible values are generated in $L(A)$. On the threshold $p(N)\sim c N^{-(h-1)/h}$, the expressions in \cref{thm:Z_linear_forms_ii} describe a genuine crossover: as $c$ increases, the limiting proportion of covered values increases from near $0$ to near $1$, interpolating between the sparse and saturated macroscopic behaviors.

As discussed earlier, \cref{thm:Z_linear_forms} completely resolves \cite[Conjecture 4.2]{hegarty2009almost}. Specifically, replacing their conjectured threshold function of $N^{-1/h}$ with $N^{-(h-1)/h}$, we derive their conjectured statement in the regime where $p(N)$ is below the threshold. By taking
\begin{align*}
    R(x_0, \dots, x_h) & = \frac{x_0^h}{|\Aut(L)| \cdot \prod_{i=1}^h |x_i|}, \\
    g_{u_1, \dots, u_h}(y) & = \int_0^{\sum_{i=1}^h |u_i|} \exp \left(- y \sum_{t_1=0}^{|u_1|-1} \cdots \sum_{t_h=0}^{|u_h|-1} \IH_h\left(x - \sum_{i=1}^h t_i\right)\right) dx,
\end{align*}
we observe that \cref{thm:Z_linear_forms_ii} establishes \cite[Equation (4.4)]{hegarty2009almost}, so we also derive their conjectured statement in the regime where $p(N)$ is on the threshold. In the regime where $p(N)$ is above the threshold, we disprove their conjecture and replace it with the correct statement.

A further takeaway of \cref{thm:Z_linear_forms} is that $|\Aut(L)|$ controls the typical size of $L(A)$ and its complement. Specifically, if $|\Aut(L)| < |\Aut(\Tilde{L})|$, then $|L(A)|$ is asymptotically almost surely larger than $|\widetilde L(A)|$. This viewpoint recovers, as a special case, the usual heuristic comparisons between generalized sumsets (and in particular the rare MSTD sets phenomenon) by specializing to linear forms whose coefficients $u_1, \dots, u_h$ all have absolute value $1$. Specifically, fixing nonnegative integers $s,d$ and setting $h = s + d \geq 2$, we define the \textit{generalized sumset} of a set $A \subseteq \N$ via
\begin{align*}
    A_{s,d} := sA - dA
    := \left\{ a_1+\cdots+a_s-a_{s+1}-\cdots-a_{s+d} : a_i \in A \right\}.
\end{align*}
Equivalently, $A_{s,d} = L_{s,d}(A)$ for the linear form
\begin{align*}
    L_{s,d}(x_1, \dots, x_h) := \sum_{i=1}^s x_i - \sum_{i=s+1}^h x_i,
\end{align*}
where, in the notation of \cref{subsubsec:lin_forms_notation}, $s^{(L_{s,d})} = s$ and $d^{(L_{s,d})} = d$. Since $A \subseteq I_N$, we have $A_{s,d}\subseteq[-dN,sN]$, and we define the corresponding complement by
\begin{align*}
    A_{s,d}^c := [-dN,sN]\setminus A_{s,d}.
\end{align*}

\begin{theorem} \label{thm:generalized_sumsets}
    Let $p : \N \to (0,1)$ be any function satisfying \eqref{eq:p_assumptions}. Fix nonnegative integers $s, d, h$ such that $h \geq 2$ and $s + d = h$. Let $A \subseteq I_N$ be a random subset where each element of $I_N$ is independently included in $A$ with probability $p(N)$. Then the following three situations arise.
    \begin{enumerate}
        \item[(i)] If $p(N) \ll N^{-\frac{h-1}{h}}$, then
        \begin{align*}
            |A_{s,d}| \sim \frac{(N\cdot p(N))^h}{s!d!}.
        \end{align*}
        
        \item[(ii)] If $p(N) = cN^{-\frac{h-1}{h}}$ for some constant $c > 0$, then
        \begin{align*}
            & |A_{s,d}| \sim \left(h - 2\int_0^{h/2} \exp \left(- \frac{c^h}{s!d!} \IH_h(x) \right) dx\right)N;
            & |A_{s,d}^c| \sim \left(2\int_0^{h/2} \exp \left(- \frac{c^h}{s!d!} \IH_h(x) \right) dx\right)N.
        \end{align*}
        
        \item[(iii)] If $p(N) \gg N^{-\frac{h-1}{h}}$, then
        \begin{align*}
            |A_{s,d}^c| \sim \frac{2 \cdot \Gamma\left( \frac{1}{h-1} \right) \sqrt[h-1]{(h-1)! s! d!}}{(h-1) \cdot p(N)^{\frac{h}{h-1}}}.
        \end{align*}
    \end{enumerate}
\end{theorem} 
In particular, \cref{thm:generalized_sumsets} generalizes \cite[Theorem 1.1]{hegarty2009almost} from sum and difference sets with two summands to generalized sumsets with $h$ summands, under any number $0 \leq s \leq h$ of sums. 

\subsubsection{Local Poisson Convergence}

A key input to our global analysis is sharp control of local representation statistics. We fix a candidate value in the bulk of the feasible range of $L$ and consider the random local count of essentially distinct injective expressions (as captured by $L$-expressions) realizing that value. In the truly sparse regime the individual expressions behave like rare, approximately independent trials, suggesting a Poisson law for this local count. What is perhaps less immediate is that this Poissonian description persists far beyond the global coverage threshold from \cref{thm:Z_linear_forms}. Indeed, there is an intermediate window
\begin{align*}
    N^{-(h-1)/h} \ll p(N) \ll N^{-(h-2)/(h-1)}
\end{align*}
in which collisions between distinct $L$-expressions are already frequent enough to affect the macroscopic size $|L(A)|$, yet the pairwise dependence induced by overlapping ground sets is still weak enough that the local count for typical bulk values remains well-approximated by a Poisson law with the same (now diverging) mean. \cref{thm:poisson_convergence} makes this separation sharp, identifying the precise scale at which Poisson behavior holds in the bulk and when it breaks down.

\begin{theorem} \label{thm:poisson_convergence}
    Assume the same setup as in \cref{thm:Z_linear_forms}. For $k \in [0, mN]$, we let $W_k$ denote the number of $L$-expressions $\Lambda \in \mathscr{E}_k$ such that $S(\Lambda) \subseteq A$. Let $\mu_k = \E[W_k]$. Let $C > 0$ be a constant. The following two situations arise.
    \begin{thmparts}
        \item \label{thm:poisson_convergence_i} 
        
        If $p(N) \ll N^{-\frac{h-2}{h-1}}$, then $d_\TV\left(\mathcal L(W_k), \Pois(\mu_k) \right) \ll 1$ uniformly over $k \in \left[ 0, mN \right]$.

        \item \label{thm:poisson_convergence_ii}

        If $p(N) \gtrsim N^{-\frac{h-2}{h-1}}$, then $d_\TV\left(\mathcal L(W_k), \Pois(\mu_k) \right) = \Omega(1)$ uniformly over $k \in \left[CN, (m-C)N\right]$.
    \end{thmparts}
\end{theorem}
For $h \geq 3$,\footnote{When $h=2$, it holds that $N^{-(h-2)/(h-1)} = 1$, so under our standing sparsity assumption $p(N) \ll 1$ only the Poisson regime in \cref{thm:poisson_convergence_i} is relevant. The breakdown regime for $h=2$ corresponds to $p$ being bounded away from $0$, where the individual representation events are not rare and Poisson approximation in total variation fails.} \cref{thm:poisson_convergence} identifies $p(N)=N^{-(h-2)/(h-1)}$ as the sharp threshold for Poisson approximation of the bulk local representation counts $W_k$. Since $N^{-(h-2)/(h-1)} \gg N^{-(h-1)/h}$, this threshold strictly dominates the global threshold in \cref{thm:Z_linear_forms}: Poissonian local behavior breaks down only at densities where $L(A)$ is already macroscopically close to saturation. \cref{fig:phase_diagram} visually compares the two threshold scales governing global coverage and local Poisson behavior in the bulk, shown schematically by the curves $p(N) = N^{-(h-1)/h}$ and $p(N) = N^{-(h-2)/(h-1)}$.

\begin{figure}[ht]
    \centering
    \includegraphics[width=\linewidth]{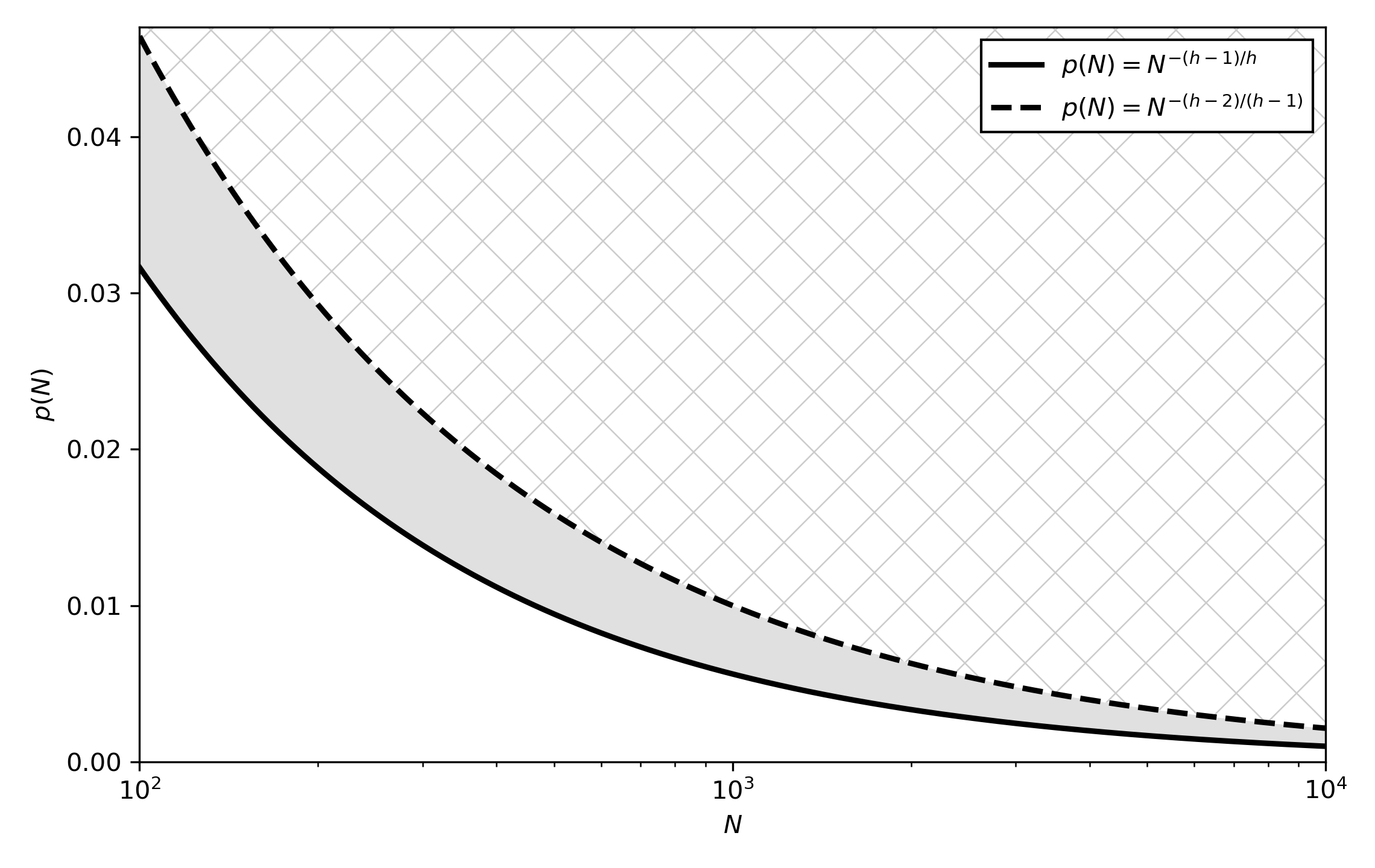}
    \caption{Schematic phase diagram (shown for $h = 4$) illustrating \cref{thm:Z_linear_forms,thm:poisson_convergence}. The solid curve $p(N) = N^{-(h-1)/h}$ represents the global threshold for the coverage of $L(A)$ from \cref{thm:Z_linear_forms}, while the dashed curve $p(N) = N^{-(h-2)/(h-1)}$ represents the local threshold for Poisson approximation of bulk representation counts from \cref{thm:poisson_convergence}. Below the solid curve, $L(A)$ is sparse within its feasible interval and is ``essentially Sidon." The shaded band indicates the intermediate regime in which the random image $L(A)$ occupies nearly all of its feasible range, but bulk representation counts remain well-approximated by a Poisson law. The hatched region indicates densities above the local threshold where dependence from overlaps is no longer negligible and Poisson approximation fails. The horizontal axis is logarithmic in $N$ and the vertical axis is linear in $p$.}
    \label{fig:phase_diagram}
\end{figure}

\subsubsection{Auxiliary Remarks}

We make a number of important comments concerning our main results.

\begin{remark}
    The condition $\gcd(u_1, \dots, u_h) = 1$ assumed in the statements of \cref{thm:Z_linear_forms,thm:poisson_convergence} can be relaxed. Indeed, let $L(x_1, \dots, x_h) = u_1x_1 + \cdots + u_hx_h$ be a linear form on $h$ variables for which $\gcd(u_1, \dots, u_h) > 1$. We define the linear form $\Tilde{L}$ on $h$ variables via
    \begin{align*}
        & \Tilde{L}(x_1, \dots, x_h) = v_1x_1 + \cdots + v_hx_h;
        & v_i = u_i/\gcd(u_1, \dots, u_h).
    \end{align*}
    It is straightforward to show that
    \begin{align} \label{eq:expression_collection_not_coprime}
        \mathscr{E}_k^{(L,N)} = \begin{cases}
            \emptyset & \gcd(u_1, \dots, u_h) \nmid k, \\
            \mathscr{E}_{k/\gcd(u_1, \dots, u_h)}^{(\Tilde{L},N)} & \gcd(u_1, \dots, u_h) \mid k,
        \end{cases}
    \end{align}
    and that
    \begin{align*}
        & |L(A)| = |\Tilde{L}(A)|, 
        & |L(A)^c| = |\Tilde{L}(A)^c| + N \cdot \frac{\gcd(u_1, \dots, u_h) - 1}{\gcd(u_1, \dots, u_h)}\sum_{i=1}^h |u_i|.
    \end{align*}
    In this sense, we will have solved the problem for all linear forms once we have solved it under the setting in which $\gcd(u_1, \dots, u_h) = 1$. We have thus assumed this condition in our main results.
\end{remark}

\begin{remark} \label{rmk:local_landscape_non_robust}
    Unlike the global landscape given by \cref{thm:Z_linear_forms}, which is proven via working with the orbits defined in \cref{subsubsec:solutions} despite the fact that the definition of $L(A)$ does not involve an injectivity constraint, the local landscape of \cref{thm:poisson_convergence} is not robust to relaxing the injectivity condition in the definition of the orbits $\mathscr{E}_k$. This injectivity condition is necessary to yield a representation hypergraph with the dependence structure needed for the bounds in our proof to hold effective uniformly over $k$. For an explicit counterexample, we take $h = 4$ and work with the linear form
    \begin{align*}
        L(x_1, x_2, x_3, x_4) = x_1 + x_2 - x_3 - x_4.
    \end{align*}
    We take $k = 2N = mN/2$ and we assume that 
    \begin{align} \label{eq:counterexample_assumptions}
        N^3p^4 \ll 1.
    \end{align}
    We are working in the subcritical regime of \cref{thm:poisson_convergence} since $N^2p^3 \ll N^3p^4$. The number of injective $4$-tuples $(a_1, a_2, a_3, a_4)$ that are solutions to the equation
    \begin{align*}
        L(a_1, a_2, a_3, a_4) = a_1 + a_2 - a_3 - a_4 = -dN + k = 0
    \end{align*}
    is, by \eqref{eq:counterexample_assumptions}, equal to zero with high probability. On the other hand, the number of noninjective effective solutions is equal to
    \begin{align*}
        \binom{|A|}{2} + |A| = \frac{|A|^2 + |A|}{2}
    \end{align*}
    with high probability. Indeed, modulo coefficient-preserving symmetries, each unordered pair $\{a,b\} \subseteq A$ with $a \neq b$ yields the solution class represented by $(a,b,a,b)$, while each $a \in A$ yields the solution $(a,a,a,a)$. Since the number of effective solutions is typically a number of the form $(z^2+z)/2$ for $z \in \N$, the number of effective solutions is not close in total variation to the Poisson distribution with the corresponding expectation, which itself tends to infinity as $N \to \infty$.

    With minor adaptations to our arguments, it is possible to recover the statement of \cref{thm:poisson_convergence} for linear forms exhibiting additional structure with the injectivity condition relaxed in our definition of local representation. For instance, it can be shown in this relaxed setting that \cref{thm:poisson_convergence} holds for all linear forms $L$ for which there does not exist a nonempty subset of coefficients which sums to zero, as the asymptotic bounds of \cref{lem:L_expression_tuples_lower_bounds,lem:L_expressions_tuples_upper_bounds} still hold. We do not pursue this here to focus on those aspects of the theory that are in some sense universal.
\end{remark}

\subsection{Proof Outline}

We now sketch the main ideas behind the proofs of our main results, emphasizing the difficulties introduced by the $h \geq 3$ setting of the problem. At a high level, the proof of \cref{thm:Z_linear_forms} follows a standard template: we compute the expectations of $|L(A)|$ and $|L(A)^c|$ in each regime, and then we show that these random variables concentrate strongly about their expectations. The main complication is that for $h \geq 3$ there are many more ways for representations of the same value to overlap, introducing additional structure and dependence which makes computations via a sieve (as was done for the $h=2$ setting by \cite{hegarty2009almost}) unwieldy.

To control the dependence, for each candidate value $-dN+k$ in $L(A)$, we instead work with the corresponding orbit model defined in \cref{subsubsec:solutions}. The injectivity restriction therein is both the right one for the clean uniform asymptotic statements that we obtain in \cref{subsec:asymptotic_enumeration} (without it, tied inputs can create additional degeneracies in exceptional settings, such as in the example considered in \cref{rmk:local_landscape_non_robust}) and ultimately benign to obtain asymptotics at the macroscopic scale. In this orbit framework, dependence between distinct $L$-expressions is governed entirely by overlaps of their ground sets, giving way to a tractable dependency graph that we can analyze using the asymptotics we derive for the orbit model and the Poisson paradigm.

With this framework in place, the three global regimes are handled as follows. In the subcritical range $p \ll N^{-(h-1)/h}$, overlaps are rare, so $|L(A)|$ is essentially the number of injective $h$-tuples (modulo redundancies) that land in $A$. In the critical window $p=cN^{-(h-1)/h}$, we apply the Stein-Chen method based on the aforementioned dependency graph to estimate the probability of a missing element uniformly in the bulk. Summing over $k$ produces the integral expressions in \cref{thm:Z_linear_forms_ii}. Lastly, in the supercritical regime $p \gg N^{-(h-1)/h}$, we use Janson-type exponential bounds to show that the bulk of the feasible interval is covered with overwhelming probability, so missing values come predominantly from the fringes of this feasible interval. Performing a similar computation over the fringe contributions then leads to the Gamma-function asymptotic in \cref{thm:Z_linear_forms_iii}.

To complete the proof of \cref{thm:Z_linear_forms}, we upgrade these expectation asymptotics into asymptotic almost sure statements. In the subcritical regime, a standard first moment argument suffices, but once $p \gtrsim N^{-(h-1)/h}$ such crude arguments no longer control the fluctuations. Thus, in the critical and supercritical regimes, we invoke the Kim-Vu martingale machinery for multivariate polynomials on the product space $\Omega$. Specifically, revealing the coordinates sequentially, we bound the one-step conditional influence $\Delta_n(A)$, i.e., the conditional expected change in $|L(A)^c|$ when the $n$\textsuperscript{th} coordinate of $\Omega$ is toggled. Since changing this $n$\textsuperscript{th} coordinate only affects values whose representations use $n$, and typical values in the bulk admit many alternative representations avoiding any fixed element, the bulk contribution to $\Delta_n(A)$ is typically negligible. Consequently, controlling $\Delta_n(A)$ again reduces to bounding fringe contributions.

Finally, the local threshold \cref{thm:poisson_convergence} uses essentially the same machinery. We again apply bounds based on the Stein-Chen method to the local counts to prove \cref{thm:poisson_convergence_i}. Once $p(N)\gtrsim N^{-(h-2)/(h-1)}$, overlaps among distinct $L$-expressions are no longer negligible. More specifically, since the indicators corresponding to distinct representations are positively related, these overlaps produce a uniform inflation of the local count variance for bulk values of $k$. This gives rise to a uniform $\Omega(1)$ obstruction to Poisson behavior throughout the bulk, proving \cref{thm:poisson_convergence_ii}.

\subsection{Organization}

The rest of the paper is organized as follows. In \cref{sec:preliminaries}, we establish much of the machinery that we will invoke in the proofs of our main results. In \cref{sec:computations}, we perform some standard computations that we will make use of in later sections. In \cref{sec:fast_decay,sec:critical_decay,sec:slow_decay}, we prove \cref{thm:Z_linear_forms}. In \cref{sec:poisson_convergence}, we prove \cref{thm:poisson_convergence}. We conclude the work in \cref{sec:future_directions} with many suggested directions for future research.

\section{Preliminaries} \label{sec:preliminaries}

In this section, we gather several results that we invoke in the proofs of our main results.

\subsection{Asymptotic Enumeration} \label{subsec:asymptotic_enumeration}

In this subsection, we present several enumerative asymptotic results for $L$-expressions which we invoke later in our arguments. Since the proofs are purely combinatorial and somewhat technical, we defer them to \cref{sec:preliminaries_proofs}. In particular, \cref{lem:number_of_subsets_asymptotics} shows that for bulk values of $k$, $\Phi_L(k/N)$ may be viewed as a local density for the number of $L$-expressions (i.e., of solutions in $I_N^{\underline{h}}$ modulo the redundancy subgroup $\mathscr{R}_k$ to $L(a_1,\dots,a_h) = -dN+k$).

\begin{lemma} \label{lem:number_of_subsets_asymptotics}
    Fix $K: \N \to \N$ satisfying $1 \ll K(N) \ll N$. 
    Uniformly over $k \in [0, K(N)]$, it holds that
    \begin{align} \label{eq:number_of_subsets_fringe_asymptotics}
        |\mathscr{E}_k| = |\mathscr{E}_{mN-k}| \ll N^{h-1}.
    \end{align}
    Uniformly over $k \in [K(N), mN/2)$, it holds that
    \begin{align} \label{eq:number_of_subsets_asymptotics}
        |\mathscr{E}_k| = |\mathscr{E}_{mN-k}| \sim \Phi_L(k/N) \cdot N^{h-1}.
    \end{align}
\end{lemma}

\begin{lemma} \label{lem:L_expression_tuples_lower_bounds}
    There exists a constant $k_0 = k_0^{(L)} > 0$ for which the following two statements hold uniformly over $k \in [k_0, mN/2]$.
    \begin{lemparts}
        \item \label{lem:L_expression_tuples_lower_bounds_i}
        
        Unless $h = 2$, $L(x_1, x_2) = x_1 - x_2$, and $k = N$, it holds that $|\mathscr{E}_k|= \Omega(k^{h-1})$.
        
        \item \label{lem:L_expression_tuples_lower_bounds_ii}
        
        For $h \geq 3$, the number of $2$-tuples $\left(\Lambda_1, \Lambda_2 \right) \in \mathscr{E}_k^2$ such that
        \begin{align} \label{eq:L_L_expression_tuples_lower_bounds_pair_cond}
            & S(\Lambda_1) \cap S(\Lambda_2) \neq \emptyset,
            & \left| S(\Lambda_1) \cup S(\Lambda_2) \right| = 2h-1
        \end{align}
        is $\Omega(k^{2h-3})$.
    \end{lemparts}
\end{lemma}

\begin{lemma} \label{lem:L_expressions_tuples_upper_bounds}
    Fix positive integers $t, \ell$ such that $t \in [4]$ and $h \leq \ell \leq \max\{h,th-1\}$. Uniformly over $k \in (0, mN/2]$, the number of $t$-tuples $\left(\Lambda_1, \dots, \Lambda_t \right) \in \mathscr{E}_k^t$ such that $\Lambda_1, \dots, \Lambda_t$ are distinct and satisfy
    \begin{align} \label{eq:L_expressions_tuples_upper_bounds}
        & S(\Lambda_i) \cap \left( \bigcup_{j\neq i} S(\Lambda_j) \right) \neq \emptyset \text{ for all } i \in [t] \text{ if } t \geq 2;
        & \left|\bigcup_{i=1}^t S(\Lambda_i)\right| = \ell
    \end{align}
    is $O(k^{\ell-\lceil (\ell+1)/h \rceil})$ if $t \geq 2$ and $O(k^{h-1})$ if $t = 1$ and $\ell = h$.
\end{lemma}

\subsection{Poisson Approximation}

In \cref{sec:critical_decay,sec:slow_decay} we will approximate local representation counts by Poisson laws using the Stein--Chen method on an explicit dependency graph over $L$-expressions. This is very natural here because each candidate value $-dN+k$ can be realized by many potential $L$-expressions, each occurring with probability $p^{|S(\Lambda)|}$, and the dependence is confined to overlaps of ground sets. We also use these tools in \cref{sec:poisson_convergence} to prove the matching lower bound underlying the second threshold by showing that Poisson approximation fails for each $k$ in the bulk once overlaps induce non-negligible local dependence. We record the needed definitions and bounds below; see \cite{arratia1990poisson, barbour1992poisson, aldous2013probability} and \cite[Chapter 8]{alon2016probabilistic} for more general treatments of this area.

\begin{definition}[\cite{barbour1992poisson}, Equation (1.1)]
    Let $I$ be a finite index set. For each $\alpha \in I$, let $X_\alpha$ be a Bernoulli random variable. We say that the random variables $\left\{X_\alpha: \alpha \in I\right\}$ are \textit{positively related} if, for every $\alpha \in I$, there exist random variables $\left\{Y_{\beta \alpha}: \beta \in I \setminus \{\alpha\}\right\}$ defined on the same probability space such that
    \begin{align*}
        & \mathcal L\left(Y_{\beta \alpha}: \beta \in I\right) = \mathcal L\left(X_\beta: \beta \in I \ | \ X_\alpha = 1 \right),
        & Y_{\beta \alpha} \geq X_\beta \text{ for all } \beta \in I \setminus \{\alpha\}.
    \end{align*}
\end{definition}

\begin{theorem}[\cite{arratia1989two}, Theorem 1] \label{thm:stein_chen}
    Let $I$ be a finite index set. For each $\alpha \in I$, let $X_\alpha$ be a Bernoulli random variable with parameter $p_\alpha$, and let $B_\alpha \subseteq I$ denote the \textit{dependency set} of $\alpha$ for every $\alpha \in I$, i.e., $\beta \in B_\alpha$ if and only if $X_\alpha$ and $X_\beta$ are dependent. Define
    \begin{align*}
        & b_1 := \sum_{\alpha \in I} \sum_{\beta \in B_\alpha} p_\alpha p_\beta, \\
        & b_2 := \sum_{\alpha \in I} \sum_{\alpha \neq \beta \in B_\alpha} \Pr[X_\alpha X_\beta = 1].
    \end{align*}
    Let $W := \sum_{\alpha \in I} X_\alpha$, and let $\mu := \E[W] = \sum_{\alpha \in I} p_\alpha$. Then
    \begin{align*}
        \left|\Pr[W = 0] - e^{-\mu} \right| < \min\{1, \mu^{-1}\}(b_1+b_2).
    \end{align*}
\end{theorem}

We next record the Janson inequalities, which we use in \cref{sec:slow_decay} to obtain exponentially decaying bounds on the probability that a bulk value $-dN + k$ fails to lie in the image set $L(A)$. We adapt the statement specifically for use in the present work.
\begin{theorem}[{\cite[Theorems 8.1.1 and 8.1.2]{alon2016probabilistic}}] \label{thm:janson_ineqs}
    Fix $k \in [0, mN]$. Let $A$ be a $p$-random subset of the universal set $I_N$. Let $\{ S(\Lambda) \}_{\Lambda \in \mathscr{E}_k}$ be subsets of $I_N$. For each $\Lambda \in \mathscr{E}_k$, we let $E_\Lambda$ denote the event that $S(\Lambda) \subseteq A$. For $\Lambda, \Lambda' \in \mathscr{E}_k$, we write $\Lambda \sim \Lambda'$ if $\Lambda \neq \Lambda'$ and $S(\Lambda) \cap S(\Lambda') \neq \emptyset$. We define
    \begin{align*}
        & \mu_k := \sum_{\Lambda \in \mathscr{E}_k} \Prob\left[ E_\Lambda \right];
        & \Delta_k^* := \mathop{\sum_{\left(\Lambda, \Lambda'\right) \in \mathscr{E}_k^2}}_{ \Lambda \sim \Lambda'} \Prob\left[ E_\Lambda \land E_{\Lambda'} \right].
    \end{align*}
    Then it holds that
    \begin{align*}
        \Prob\left[ \bigwedge_{\Lambda \in \mathscr{E}_k} \overline{E_\Lambda} \right] \leq e^{-\mu_k + \Delta_k^*/2} \land e^{-I_{\{\Delta_k^* \geq \mu_k \}}\mu_k^2/ (2\Delta_k^*)}.
    \end{align*}
\end{theorem}
We assume the setup of \cref{thm:janson_ineqs} and record an important identity here. For each $\Lambda \in \mathscr{E}_k$, we let $X_\Lambda$ denote the indicator random variable associated with the event $E_\Lambda$. Working with the quantities introduced in \cref{thm:stein_chen}, we let $B_\Lambda$ denote the dependency set of $X_\Lambda$ (as defined in \cref{thm:stein_chen}) and we let $b_2(k)$ correspond to the quantity $b_2$ for this setting. It then follows that 
\begin{align} \label{eq:delta_expression}
    \Delta_k^* = \frac{1}{2}\sum_{\Lambda \in \mathscr{E}_k} \sum_{\Lambda' \in B_\Lambda \setminus \Lambda} \Prob\left[ E_\Lambda \land E_{\Lambda'} \right] = \frac{1}{2}\sum_{\Lambda \in \mathscr{E}_k} \sum_{\Lambda' \in B_\Lambda \setminus \Lambda} \Prob\left[ X_\Lambda X_{\Lambda'} = 1 \right] = b_2(k)/2.
\end{align}
As such, we have that
\begin{align} 
    & \Prob\left[ \bigwedge_{\Lambda \in \mathscr{E}_k} \overline{E_\Lambda} \right] \stackrel{(\text{\cref{thm:janson_ineqs})}}{\leq} e^{-\mu_k + \Delta_k^*/2} \land e^{-I_{\{\Delta_k^* \geq \mu_k \}}\mu_k^2/ (2\Delta_k^*)} \nonumber \\
    & \qquad \leq I_{\{\Delta_k^* \geq \mu_k \}}e^{-\mu_k^2/ (2\Delta_k^*)} + I_{\{\Delta_k^* < \mu_k \}}e^{-\mu_k + \Delta_k^*/2} \nonumber \\
    & \qquad \leq I_{\{\Delta_k^* \geq \mu_k \}}e^{-\mu_k^2/ (2\Delta_k^*)} + I_{\{\Delta_k^* < \mu_k \}}e^{-\mu_k/2} \leq e^{-\mu_k^2/ (2\Delta_k^*)} + e^{-\mu_k/2} \stackrel{\eqref{eq:delta_expression}}{=} e^{-\mu_k^2/ b_2(k)} + e^{-\mu_k/2}. \label{eq:janson_main_bound}
\end{align}
We now collect the Poisson approximation theorems that we will invoke to control the law of our counting variables in total variation.
\begin{theorem}[\cite{arratia1989two}, Theorems 1 and 2] \label{thm:poisson_process_convergence}
    Assume the setup of \cref{thm:stein_chen}. Let $\Xi$ denote the corresponding dependent Bernoulli process with intensity $\pi$. Then we have that
    \begin{align*}
        &  d_\TV\left( \mathcal L(W), \Pois(\mu)\right) \leq \min\{1, \mu^{-1}\} \left(b_1 + b_2\right); \\
        & d_\TV\left(\mathcal L(\Xi), \Pois(\pi)\right) \leq b_1 + b_2.
    \end{align*}
\end{theorem}

\begin{theorem}[\cite{barbour1992poisson}, Theorem 3.E] \label{thm:positively_related_lower_bound}
    Let $\left\{X_\alpha: \alpha \in I\right\}$ be positively related random variables, with $p_\alpha := \E[X_\alpha]$. Let $W := \sum_{\alpha \in I} X_\alpha$, and let $\mu := \E[W]$. Set
    \begin{align*}
        & \epsilon := \frac{\Var(W)}{\mu} - 1,& \gamma := \frac{\E\left[(W-\E[W])^4 \right]}{\mu} - 1,
    \end{align*}
    and also\footnote{We note that the expression for $\psi$ we provide is not exactly the same as that which was stated in \cite[Theorem 3.E]{barbour1992poisson}, but is instead the upper bound on $\psi$ that is given in \cite[Equation (2.12)]{barbour1992poisson}. This raises no issue, as this may only weaken the lower bound on the total variation distance $d_\TV\left(\mathcal L(W), \Pois(\mu) \right)$ that is guaranteed by the theorem.}
    \begin{align*}
        \psi := \left( 1 + \frac{3}{2}\max_{\alpha \in I} p_\alpha \right) \left(\frac{\gamma}{\mu \epsilon}\right)_+ + 3\epsilon + \left( \frac{15}{2} + \frac{7}{\mu} \right)\left( 1 + \epsilon \right)^2\frac{\max_{\alpha \in I}p_\alpha}{\epsilon}.
    \end{align*}
    If $\epsilon > 0$, then
    \begin{align*}
        d_\TV\left(\mathcal L(W), \Pois(\mu) \right) \geq \frac{\epsilon}{11+3\psi}.
    \end{align*}
\end{theorem}

\subsection{Martingale Concentration} \label{subsec:kim_vu_concentration}

We will establish each of the three cases in \cref{thm:Z_linear_forms} by first computing the expectations of all relevant random variables, then showing that these random variables are strongly concentrated about their expectations. In particular, when proving concentration bounds in \cref{sec:critical_decay,sec:slow_decay}, we will rely on the martingale machinery developed in \cite{kim2000concentration, vu2000new, vu2002concentration}.  Following the presentation of \cite{vu2002concentration}, we let $\mathcal F_n:=\sigma(a_0,\dots,a_{n})$ define the natural filtration on $\Omega_N$ induced by the coordinates. Now, for each $A \in \Omega$ and $n \in I_N$, we define the random variables
\begin{align*}
    C_n(0,A) & := \E\left[ |L(A \setminus \{n\} )^c| - |L(A)^c| \ \Big| \ \mathcal F_{n-1}\right], \\
    C_n(1,A) & := \E\left[ |L(A)^c| - |L(A \cup \{n\})^c| \ \Big| \ \mathcal F_{n-1} \right], \\
    V_n(A) & := \int_0^1 C_n^2(x,A) \ d\mu_n(x) = (1-p)C_n^2(0,A) + pC_n^2(1,A).
\end{align*}
From here, we define the random variables
\begin{align*}
    & C(A) := \mathop{\max_{n \in I_N}}_{x \in \{0,1\}} C_n(x,A),
    & V(A) := \sum_{n=0}^N V_n(A).
\end{align*}
With this setup in place, we extract the following result.
\begin{theorem}[\cite{vu2002concentration}, Lemma 3.1] \label{thm:vu_lemma}
    For any $\lambda, \mathbf{V}, \mathbf{C} > 0$ such that $\lambda \leq 4\mathbf{V}/\mathbf{C}^2$, we have that
    \begin{align*}
        \Prob\left( \left| |L(A)^c| - \E\left[ |L(A)^c| \right] \right| \geq \sqrt{\lambda \mathbf{V}} \right) \leq 2e^{-\lambda/4} & + \Prob(C(A) \geq \mathbf{C}) + \Prob(V(A) \geq \mathbf{V}).
    \end{align*}
\end{theorem}
We now perform a number of simplifications which will lead to expressions for $C(A)$ and $V(A)$ which are more amenable to analysis. We introduce, for each $A \in \Omega$ and $n \in I_N$, the random variable
\begin{equation} \label{eq:delta_function}
    \begin{aligned}
        \Delta_n(A):= \E\left[ \big| L\left(A \setminus \{n\}\right)^c \big| - \big| L\left(A \cup \{n\}\right)^c \big| \ \Big| \ \mathcal F_{n-1} \right].
    \end{aligned}
\end{equation}
By ``taking conditional expectations over $a_n$," which is independent of $\mathcal F_{n-1}$, we may express
\begin{align*}
    \E\left[ |L(A)^c| \ \Big| \ \mathcal F_{n-1} \right] = p \E\left[ |L(A \cup \{n\})^c| \ \Big| \ \mathcal F_{n-1} \right] + (1-p) \E\left[ |L(A \setminus \{n\})^c| \ \Big| \ \mathcal F_{n-1} \right].
\end{align*}
From here, we readily observe that
\begin{align*}
    & C_n(0,A) = p\Delta_n(A);
    & C_n(1,A) = (1-p)\Delta_n(A).
\end{align*}
It now readily follows that
\begin{align} \label{eq:C(A)_reformulation}
    C(A) = (1-p) \max_{n \in I_N} \Delta_n(A) \sim \max_{n \in I_N} \Delta_n(A).
\end{align}
Similarly, we have that
\begin{align} \label{eq:V(A)_reformulation}
    V(A) = p(1-p)\sum_{n=0}^N \left( \Delta_n(A) \right)^2 \sim p\sum_{n=0}^N \left( \Delta_n(A) \right)^2.
\end{align}
We will make use of the expressions \eqref{eq:C(A)_reformulation} and \eqref{eq:V(A)_reformulation} in \cref{sec:critical_decay,sec:slow_decay}.

\section{Moment Asymptotics} \label{sec:computations}

In this section, we isolate those computations which we will frequently refer to throughout the proofs of our main theorems. For $k \in [0, mN]$, we let $W_k$ denote the number of $L$-expressions $\Lambda \in \mathscr{E}_k$ such that $S(\Lambda) \subseteq A$. We recall that for each $\Lambda \in \mathscr{E}_k$, we let $X_\Lambda$ denote the Bernoulli random variable corresponding to the event that $S(\Lambda) \subseteq A$, so that 
\begin{align*}
    W_k = \sum_{\Lambda \in \mathscr{E}_k} X_\Lambda.
\end{align*}
It readily follows (e.g., via \eqref{eq:number_of_subsets_left_right}) that $W_k \stackrel{d}{=} W_{mN-k}$. Furthermore, under this correspondence, it holds at the process level that we have the equality in distribution
\begin{align} \label{eq:left_right_process_equivalence}
    \left( X_\Lambda : \Lambda \in \mathscr{E}_k \right) \stackrel{d}{=} \left( X_\Lambda : \Lambda \in \mathscr{E}_{mN-k}\right).
\end{align}
Notably, \eqref{eq:left_right_process_equivalence} allows us to extend computations over values $k \in [0, mN/2]$ to values $k \in [mN/2, mN]$. We now fix some $k \in [0, mN/2]$. We take $\mathscr{E}_k$ as our finite index set of interest. Up to constants, we bound the expectation, variance, and fourth central moment of $W_k$. Our bounds hold in the following regimes.
\begin{enumerate}
    \item All asymptotic upper bounds and asymptotic equivalences hold uniformly over $k \in (0, mN/2]$. 

    \item All asymptotic lower bounds hold uniformly over $k \in [k_0, mN/2]$ so long as we are not working in the exceptional setting $h = 2$, $L(x_1, x_2) = x_1 - x_2$, and $k = N$. In this exceptional setting, all asymptotic lower bounds hold uniformly over $k \in [k_0, mN/2)$.
\end{enumerate}
Asymptotic statements involving a condition on $k$ are to be understood as holding uniformly over such $k$ when working over a regime where the condition applies. We have that
\begin{align}
    \mu_k := \E[W_k] & = \sum_{\Lambda \in \mathscr{E}_k} p^{|S(\Lambda)|} = |\mathscr{E}_k| \cdot p^h \stackrel{\text{(\cref{lem:L_expression_tuples_lower_bounds,lem:L_expressions_tuples_upper_bounds})}}{\asymp} k^{h-1}p^h; \label{eq:basic_computation_mean} \\
    \Var(W_k) & = \mathop{\sum_{\left(\Lambda_1,\Lambda_2\right) \in \mathscr{E}_k^2}}_{\Lambda_1 \neq \Lambda_2} \Cov\left(X_{\Lambda_1},X_{\Lambda_2}\right) + \sum_{\Lambda \in \mathscr{E}_k} \Var(X_\Lambda) \label{eq:basic_computation_1} \\
    & \sim \mathop{\sum_{\left(\Lambda_1,\Lambda_2\right) \in \mathscr{E}_k^2, \ \Lambda_1 \neq \Lambda_2}}_{S(\Lambda_1) \cap S(\Lambda_2) \neq \emptyset} p^{|S(\Lambda_1) \cup S(\Lambda_2) |} - p^{|S(\Lambda_1)|+|S(\Lambda_2)|} + \E[W_k] \nonumber \\
    & \sim \sum_{\ell=h}^{2h-1} \mathop{\mathop{\sum_{\left(\Lambda_1,\Lambda_2\right) \in \mathscr{E}_k^2, \ \Lambda_1 \neq \Lambda_2}}_{S(\Lambda_1) \cap S(\Lambda_2) \neq \emptyset}}_{|S(\Lambda_1) \cup S(\Lambda_2)|=\ell} p^\ell + \E[W_k] \\
    & \begin{cases}
        \stackrel{\text{(\cref{lem:L_expression_tuples_lower_bounds})},\eqref{eq:basic_computation_mean}}{\gtrsim} k^{2h-3}p^{2h-1} + k^{h-1}p^h \\
        \stackrel{\text{(\cref{lem:L_expressions_tuples_upper_bounds})},\eqref{eq:basic_computation_mean}}{\lesssim} \sum_{\ell=h}^{2h-1} k^{\ell-2}p^\ell + k^{h-1}p^h 
    \end{cases} \nonumber \\
    & \begin{cases}
        \geq k^{2h-3}p^{2h-1}; \\
        \stackrel{(k \gtrsim 1/p)}{\lesssim} k^{2h-3}p^{2h-1} + k^{h-1}p^h \stackrel{\big( k \gtrsim (1/p)^{\frac{h-1}{h-2}} \big)}{\lesssim} k^{2h-3}p^{2h-1}.
    \end{cases} \label{eq:basic_computation_2}
\end{align}
We also bound fourth central moments via
\begin{align}
    & \E\left[(W_k-\E[W_k])^4\right] = \E\left[\left( \sum_{\Lambda \in \mathscr{E}_k} (X_\Lambda - p_\Lambda) \right)^4\right]  \leq \sum_{\left(\Lambda_1, \dots, \Lambda_4\right) \in \mathscr{E}_k^4} \left| \E\left[ (X_{\Lambda_1}-p_{\Lambda_1}) \cdots (X_{\Lambda_4}-p_{\Lambda_4}) \right] \right| \nonumber \\
    & \quad = \sum_{\ell=1}^{4h} \mathop{\sum_{\left(\Lambda_1, \dots, \Lambda_4\right) \in \mathscr{E}_k^4}}_{\left|\bigcup_{i=1}^4 S(\Lambda_i)\right| = \ell} \left| \E\left[ (X_{\Lambda_1}-p_{\Lambda_1}) \cdots (X_{\Lambda_4}-p_{\Lambda_4}) \right] \right| \nonumber \\
    & \lesssim \sum_{\ell=h}^{4h-2} \mathop{\mathop{\mathop{\sum_{\left(\Lambda_1, \dots, \Lambda_4\right) \in \mathscr{E}_k^4 \text{ distinct;}}}_{\left|\bigcup_{i=1}^4 S(\Lambda_i)\right| = \ell;}}_{S(\Lambda_i) \cap \left( \bigcup_{j\neq i} S(\Lambda_j) \right) \neq \emptyset}}_{\text{for all } i \in [4]} \left| \E\left[ (X_{\Lambda_1}-p_{\Lambda_1}) \cdots (X_{\Lambda_4}-p_{\Lambda_4}) \right] \right| \label{eq:basic_computation_3} \\
    & \qquad + \sum_{\ell=h}^{3h-1} \mathop{\mathop{\mathop{\sum_{\left(\Lambda_1, \Lambda_2, \Lambda_3\right) \in \mathscr{E}_k^3 \text{ distinct;}}}_{\left|\bigcup_{i=1}^3 S(\Lambda_i)\right| = \ell;}}_{S(\Lambda_i) \cap \left( \bigcup_{j\neq i} S(\Lambda_j) \right) \neq \emptyset}}_{\text{for all } i \in [3]} \left| \E\left[ (X_{\Lambda_1}-p_{\Lambda_1})^2 (X_{\Lambda_2}-p_{\Lambda_2}) (X_{\Lambda_3}-p_{\Lambda_3}) \right] \right| \label{eq:basic_computation_5} \\
    & \qquad + \sum_{\ell=h}^{2h-1} \mathop{\mathop{\mathop{\sum_{\left(\Lambda_1, \Lambda_2\right) \in \mathscr{E}_k^2 \text{ distinct;}}}_{\left|S(\Lambda_1) \cup S(\Lambda_2) \right| = \ell;}}_{S(\Lambda_1) \cap S(\Lambda_2) \neq \emptyset}} \left| \E\left[ (X_{\Lambda_1}-p_{\Lambda_1})^2 (X_{\Lambda_2}-p_{\Lambda_2})^2 \right] \right| + \left| \E\left[ (X_{\Lambda_1}-p_{\Lambda_1})^3 (X_{\Lambda_2}-p_{\Lambda_2})\right] \right| \label{eq:basic_computation_6} \\
    & \qquad + \sum_{\Lambda \in \mathscr{E}_k} \E\left[ (X_\Lambda-p_\Lambda)^4 \right] \nonumber \\
    & \stackrel{\text{(\cref{lem:L_expressions_tuples_upper_bounds})}}{\lesssim} \sum_{\ell=h}^{4h-2} k^{\ell-\lceil (\ell+1)/h \rceil} p^\ell + k^{h-1}p^h \lesssim \sum_{\ell=3h}^{4h-2} k^{\ell-4}p^\ell + \sum_{\ell=2h}^{3h-1} k^{\ell-3}p^\ell + \sum_{\ell=h}^{2h-1} k^{\ell-2}p^\ell + k^{h-1}p^h \nonumber \\
    & \stackrel{(k \gtrsim 1/p)}{\lesssim} k^{4h-6}p^{4h-2} + k^{3h-4}p^{3h-1} + k^{2h-3}p^{2h-1} + k^{h-1}p^h \stackrel{\big( k \gtrsim (1/p)^{\frac{h-1}{h-2}} \big)}{\lesssim} k^{4h-6}p^{4h-2}. \label{eq:basic_computation_4}
\end{align}
In particular, we observe that all summands in \eqref{eq:basic_computation_1} and \eqref{eq:basic_computation_3}-\eqref{eq:basic_computation_6} corresponding to tuples which fail to satisfy the former condition of \eqref{eq:L_expressions_tuples_upper_bounds} vanish, so we truncate the upper limits of the sums accordingly. 

\smallskip

For use in \cref{sec:poisson_convergence}, we also record the following computation for $h \geq 3$, which is observed by proceeding slightly more frugally in the variance computations starting from \eqref{eq:basic_computation_1}.
\begin{align*}
    & \Var(W_k) - \left( 1-p \right)\mu_k \stackrel{\eqref{eq:basic_computation_1}}{=} \mathop{\sum_{\left(\Lambda_1,\Lambda_2\right) \in \mathscr{E}_k^2}}_{\Lambda_1 \neq \Lambda_2} \Cov\left(X_{\Lambda_1},X_{\Lambda_2}\right) + \left[ \sum_{\Lambda \in \mathscr{E}_k} \Var(X_\Lambda) - \left( 1-p \right)\mu_k \right] \\
    & \quad = \mathop{\sum_{\left(\Lambda_1,\Lambda_2\right) \in \mathscr{E}_k^2}}_{\Lambda_1 \neq \Lambda_2} \Cov\left(X_{\Lambda_1},X_{\Lambda_2}\right) + \left[ \sum_{\Lambda \in \mathscr{E}_k} \left( 1 - p^{|S(\Lambda)|} \right) p^{|S(\Lambda)|} - \left( 1-p \right)\sum_{\Lambda \in \mathscr{E}_k} p^{|S(\Lambda)|} \right] \\
    & \quad \geq \mathop{\sum_{\left(\Lambda_1,\Lambda_2\right) \in \mathscr{E}_k^2}}_{\Lambda_1 \neq \Lambda_2} \Cov\left(X_{\Lambda_1},X_{\Lambda_2}\right) = \mathop{\sum_{\left(\Lambda_1,\Lambda_2\right) \in \mathscr{E}_k^2, \ \Lambda_1 \neq \Lambda_2}}_{S(\Lambda_1) \cap S(\Lambda_2) \neq \emptyset} p^{|S(\Lambda_1) \cup S(\Lambda_2) |} - p^{|S(\Lambda_1)|+|S(\Lambda_2)|} \\
    & \quad = \mathop{\sum_{\left(\Lambda_1,\Lambda_2\right) \in \mathscr{E}_k^2, \ \Lambda_1 \neq \Lambda_2}}_{S(\Lambda_1) \cap S(\Lambda_2) \neq \emptyset} p^{|S(\Lambda_1) \cup S(\Lambda_2)|} \left( 1 - p^{|S(\Lambda_1) \cap S(\Lambda_2)|} \right) \\
    & \quad \geq \mathop{\sum_{\left(\Lambda_1,\Lambda_2\right) \in \mathscr{E}_k^2, \ \Lambda_1 \neq \Lambda_2}}_{S(\Lambda_1) \cap S(\Lambda_2) \neq \emptyset} p^{|S(\Lambda_1) \cup S(\Lambda_2)|} \left( 1 - p \right) \gtrsim \mathop{\sum_{\left(\Lambda_1,\Lambda_2\right) \in \mathscr{E}_k^2, \ \Lambda_1 \neq \Lambda_2}}_{S(\Lambda_1) \cap S(\Lambda_2) \neq \emptyset} p^{|S(\Lambda_1) \cup S(\Lambda_2) |} \\
    & \quad = \sum_{\ell=h}^{2h-1} \mathop{\mathop{\sum_{\left(\Lambda_1,\Lambda_2\right) \in \mathscr{E}_k^2, \ \Lambda_1 \neq \Lambda_2}}_{S(\Lambda_1) \cap S(\Lambda_2) \neq \emptyset}}_{|S(\Lambda_1) \cup S(\Lambda_2)|=\ell} p^\ell \stackrel{\text{(\cref{lem:L_expression_tuples_lower_bounds})}}{\gtrsim} k^{2h-3}p^{2h-1}.
\end{align*}
Rearranging the above now yields
\begin{align} \label{eq:frugal_variance_bounds}
    \Var\left( W_k \right) - \mu_k \gtrsim k^{2h-3}p^{2h-1} - p \mu_k.
\end{align}
Finally, we bound the quantities that were introduced in \cref{thm:stein_chen}. Let $B_\Lambda \subseteq \mathscr{E}_k$ denote the dependency set of $\Lambda$ for every $\Lambda \in \mathscr{E}_k$, so that $\Lambda' \in B_\Lambda$ if and only if $S(\Lambda) \cap S(\Lambda') \neq \emptyset$. We may bound $b_1(k)$ via
\begin{align}
    b_1(k) & = \sum_{\Lambda \in \mathscr{E}_k} \sum_{\Lambda' \in B_\Lambda} p_\Lambda p_{\Lambda'} = \sum_{t=1}^h \mathop{\sum_{\left(\Lambda, \Lambda'\right) \in \mathscr{E}_k^2}}_{|S(\Lambda) \cap S(\Lambda')| = t} p_\Lambda p_{\Lambda'} = \sum_{\Lambda \in \mathscr{E}_k} p_\Lambda^2 + \sum_{t=1}^h \mathop{\sum_{\left(\Lambda, \Lambda'\right) \in \mathscr{E}_k^2, \ \Lambda \neq \Lambda'}}_{|S(\Lambda) \cap S(\Lambda')| = t} \Prob\left[X_\Lambda X_{\Lambda'} = 1\right] \nonumber \\
    & = |\mathscr{E}_k| \cdot p^{2h} + \sum_{t=1}^h \sum_{\ell=h}^{2h-t} \mathop{\mathop{\sum_{\left(\Lambda, \Lambda'\right) \in \mathscr{E}_k^2, \ \Lambda \neq \Lambda'}}_{S(\Lambda) \cap S(\Lambda') \neq \emptyset}}_{|S(\Lambda) \cup S(\Lambda')|=\ell} \Prob\left[X_\Lambda X_{\Lambda'} = 1\right] \stackrel{\text{(\cref{lem:L_expressions_tuples_upper_bounds})}}{\lesssim} k^{h-1}p^{2h} + \sum_{t=1}^h \sum_{\ell=h}^{2h-t} k^{\ell-\lceil (\ell+1)/h \rceil} p^\ell \nonumber \\
    & \lesssim k^{h-1}p^{2h} + \sum_{\ell=h}^{2h-1} k^{\ell-2}p^\ell \lesssim \sum_{\ell=h}^{2h-1} k^{\ell-2}p^\ell \stackrel{(k \gtrsim 1/p)}{\lesssim} k^{2h-3}p^{2h-1}. \label{eq:b_1(k)_bound} 
\end{align}
We may similarly bound $b_2(k)$ via
\begin{align}
    b_2(k) & = \sum_{\Lambda \in \mathscr{E}_k} \sum_{\Lambda' \in B_\Lambda \setminus \Lambda} \Prob\left[X_\Lambda X_{\Lambda'} = 1\right] = \sum_{t=1}^h \mathop{\sum_{\left(\Lambda, \Lambda'\right) \in \mathscr{E}_k^2, \ \Lambda \neq \Lambda'}}_{|S(\Lambda) \cap S(\Lambda')| = t} \Prob\left[X_\Lambda X_{\Lambda'} = 1\right] \nonumber \\
    & \stackrel{\eqref{eq:b_1(k)_bound}}{\lesssim} \sum_{\ell=h}^{2h-1} k^{\ell-2}p^\ell \stackrel{(k \gtrsim 1/p)}{\lesssim} k^{2h-3}p^{2h-1}. \label{eq:b_2(k)_bound}
\end{align}
Combining these two computations yields
\begin{align} \label{eq:AGG_1_not_dN}
    b_1(k) + b_2(k) \stackrel{(k \gtrsim 1/p)}{\lesssim} k^{2h-3}p^{2h-1}.
\end{align}

\section{Subcritical Decay} \label{sec:fast_decay}

We commence the proof of \cref{thm:Z_linear_forms}. Our representation machinery is built via the orbits $\mathscr{E}_k$. Accordingly, throughout \cref{sec:fast_decay,sec:critical_decay,sec:slow_decay}, we first work with the injective image
\begin{align*}
    L_{\neq}(A) := \left\{u_1a_1 + \cdots + u_ha_h : a_i \in A,\ a_i \text{ pairwise distinct} \right\}
\end{align*}
and we define its complement via
\begin{align*}
    L_{\neq}(A)^c = \left[ -dN, sN \right] \setminus L_{\neq}(A).
\end{align*}
In this section, we prove \cref{thm:Z_linear_forms_i}. Throughout this section, we assume that
\begin{align} \label{eq:fast_decay_assumption}
     Np^{\frac{h}{h-1}} \ll 1 \iff N^{h-1}p^h \ll 1.
\end{align}
Towards this end, we observe that
\begin{align*}
    |L_{\neq}(A)| \leq |L(A)| \leq |L_{\neq}(A)| + \binom{h}{2} |A|^{h-1},
\end{align*}
and since $|A| \sim Np$ by a standard Chernoff bound, so that $|A|^{h-1} \sim (Np)^{h-1}$, it suffices to show that
\begin{align} \label{eq:subcritical_nondistinct}
    |L_{\neq}(A)| \sim \frac{\left( N \cdot p(N) \right)^h}{|\Aut(L)|}.
\end{align}
We prove an upper bound and a lower bound for $|L_{\neq}(A)|$, each of which is asymptotically equivalent to $(Np)^h/|\Aut(L)|$. We begin with the upper bound. It holds that
\begin{align} \label{eq:fast_decay_ubd}
    |L_{\neq}(A)| \leq \sum_{k=0}^{mN} W_k = \frac{|A|(|A|-1) \cdots (|A|-h+1)}{|\Aut(L)|} \sim \frac{(Np)^h}{|\Aut(L)|},
\end{align}
where the asymptotic equivalence follows since $A$ is a $p$-random subset of $I_N$ whose expected size tends to infinity. On the other hand, we obtain a lower bound for $|L_{\neq}(A)|$ via
\begin{align*}
    |L_{\neq}(A)| \geq \sum_{k=1}^{mN} W_k - \sum_{k=1}^{mN} \binom{W_k}{2}.
\end{align*}
We count this latter term via, for each $k \in (0, mN]$ and $t \in [h]$, the number of ordered pairs $\left(\Lambda, \Lambda' \right) \in \mathscr{E}_k^2$ of such $L$-expressions satisfying
\begin{align} \label{eq:pair_condition}
    S(\Lambda) \cup S(\Lambda') \subseteq A; \qquad \qquad \Lambda \neq \Lambda'; \qquad \qquad |S(\Lambda) \cap S(\Lambda')| = t.
\end{align}
The expectation of the value that we subtract is
\begin{align}
    & \sum_{k=1}^{mN} \sum_{t=1}^h \E\big[ \text{num. ordered pairs } \left(\Lambda, \Lambda'\right) \in \mathscr{E}_k^2 \text{ satisfying } \eqref{eq:pair_condition} \big] \\
    & \qquad \lesssim \sum_{k=1}^{mN/2} \sum_{t=1}^h \mathop{\sum_{\left(\Lambda, \Lambda'\right) \in \mathscr{E}_k^2, \ \Lambda \neq \Lambda'}}_{|S(\Lambda) \cap S(\Lambda')| = t} \Prob[X_\Lambda X_{\Lambda'} = 1]  \stackrel{\eqref{eq:b_1(k)_bound}}{\lesssim} \sum_{k=1}^{mN/2} \sum_{t = 1}^h \sum_{\ell=h}^{2h-t} k^{\ell-\lceil (\ell+1)/h \rceil}p^\ell \nonumber \\
    & \qquad \lesssim \sum_{k=1}^{mN/2} \sum_{t = 1}^h \sum_{\ell=h}^{2h-t} N^{\ell-\lceil (\ell+1)/h \rceil}p^\ell \lesssim \sum_{k=1}^{mN/2} N^{2h-3}p^{2h-1} \nonumber \\
    & \qquad \lesssim N^{2h-2}p^{2h-1} \ll N^{2h-1}p^{2h} = (Np)^h \cdot N^{h-1}p^h \stackrel{\eqref{eq:fast_decay_assumption}}{\ll} (Np)^h. \label{eq:fast_decay_lower_bound_1}
\end{align}
It follows from \eqref{eq:fast_decay_lower_bound_1} and Markov's inequality that asymptotically almost surely, the number we subtract to obtain the lower bound is negligible compared to $(Np)^h$. We conclude that this lower bound for $|L_{\neq}(A)|$ is asymptotically almost surely equivalent to $(Np)^h/|\Aut(L)|$, establishing \cref{thm:Z_linear_forms_i}.

\section{Critical Decay} \label{sec:critical_decay}

We now prove \cref{thm:Z_linear_forms_ii}. Throughout this section, we assume that
\begin{align} \label{eq:critical_decay_assumption}
    N^{\frac{h-1}{h}}p = c \iff N^{h-1}p^h = c^h,
\end{align}
where $c > 0$ is a constant. Our basic strategy in \cref{sec:critical_decay,sec:slow_decay} is as follows. We will show that the number of ways that an element in $\left[-dN, sN\right]$ is generated in $L_{\neq}(A)$ converges in distribution to a Poisson, from which we can find asymptotic expressions for both the expected number of elements and the expected number of missing elements in $L_{\neq}(A)$. These are the relevant expressions from \cref{thm:Z_linear_forms}. We then transfer these expectation asymptotics to the original image $L(A)$ and its complement. Finally, we conclude that the relevant random variables are concentrated about their expectations, allowing us to promote these results on expectations to the latter two parts of \cref{thm:Z_linear_forms}.

\subsection{Setup}

Before beginning this endeavor, we introduce some expressions that we use throughout. Several arguments in \cref{sec:critical_decay,sec:slow_decay,sec:poisson_convergence} require us to treat values of $k$ near the fringes of the feasible interval $[0,mN]$ separately from values in the bulk, with the location of the separation between fringe and bulk depending on the assumptions that we work under. To keep the notation consistent, we fix a small collection of deterministic cutoff scales $K_i(N)$, which may be taken to be increasing in the index. The exact forms of these cutoff functions are unimportant; any choice satisfying the stated asymptotic inequalities is admissible. We initiate this sequence by defining $K_1: \N \to \R_{\geq 0}$ to be a function satisfying the conditions
\begin{align} \label{eq:K1_asymptotics}
    1 \ll 1/p \ll K_1(N) \ll (1/p)^{\frac{h}{h-1}} \stackrel{\eqref{eq:critical_decay_assumption}}{\lesssim} N.
\end{align}

Additionally, we let $-L: \Z^h \to \Z$ denote the linear form with coefficients $-u_h \geq \cdots \geq -u_1$. It readily follows that $L_{\neq}(A) = -(-L)_{\neq}(A)$, and that for any $k$,
\begin{align} \label{eq:left_right_symmetry}
    sN-k \notin L_{\neq}(A) \iff -sN+k \notin (-L)_{\neq}(A).
\end{align} 
In many forthcoming arguments, we will prove statements over values $-dN+k$ for $k \in [0, mN/2]$, then take advantage of the symmetry apparent in \eqref{eq:left_right_symmetry} to prove the corresponding statement over values $-dN+k$ for $k \in [mN/2, mN]$. 

Finally, in \cref{sec:critical_decay,sec:slow_decay}, we work with the function $\delta: \N \to \R_{\geq 0}$. The function $\delta(N)$ should be thought of as the worst-case multiplicative margin of error for the asymptotic equivalence given by \cref{lem:number_of_subsets_asymptotics}, with $K_1(N)$ playing the role of $K(N)$ therein:
\begin{align} \label{eq:delta(N)}
    \delta(N) := \max_{k \in [K_1(N), mN/2)} \left| \frac{\mu_k}{\Phi_L(k/N) N^{h-1}p^h} - 1 \right|.
\end{align}
Indeed, it holds uniformly over $k \in [K_1(N), mN/2)$ that
\begin{align} \label{eq:critical_gamma_1}
    \mu_k \stackrel{\eqref{eq:basic_computation_mean}}{=} |\mathscr{E}_k| \cdot p^h \stackrel{\text{(\cref{lem:number_of_subsets_asymptotics})}}{\sim} \Phi_L(k/N) \cdot N^{h-1}p^h.
\end{align}
Notably, the asymptotic equivalence in \eqref{eq:critical_gamma_1} follows from the facts that all summands corresponding to $\ell < h$ are of a lower order (e.g., see \cref{lem:L_expressions_tuples_upper_bounds}) and that $K_1(N) \gg 1/p$. We deduce that 
\begin{align*}
    \delta(N) \ll 1.
\end{align*}

\subsection{Expectation}

In this subsection, we compute $\E[|L_{\neq}(A)|]$ and $\E[|L_{\neq}(A)^c|]$. We begin with the following observation, which we invoke in the main computation.

\begin{lemma} \label{lem:critical_poisson_conv}
    Uniformly over all $k \in [K_1(N), mN/2)$,
    \begin{align} \label{eq:critical_approximation_error}
       \left|\Prob\left[-dN+k \notin L_{\neq}(A)\right] - e^{-c^h\Phi_L(k/N)} \right| \ll e^{-c^h\Phi_L(k/N)}.
    \end{align}
\end{lemma}

\begin{proof}
By definition, $-dN+k \notin L_{\neq}(A)$ and $W_k = 0$ are the same event. It follows from \cref{thm:stein_chen} that uniformly over $k \in [K_1(N), mN/2)$,
\begin{align}
    & e^{c^h\Phi_L(k/N)}\left|\Prob\left[-dN+k \notin L_{\neq}(A) \right]- e^{-c^h\Phi_L(k/N)} \right| \nonumber \\
    & \qquad = e^{c^h\Phi_L(k/N)} \left|(\Prob\left[W_k = 0 \right] - e^{-\mu_k}) + (e^{-\mu_k} - e^{-c^h\Phi_L(k/N)}) \right| \nonumber \\
    & \qquad \leq \exp\left(\max_{k \in [K_1(N), mN/2]} c^h\Phi_L(k/N) \right) \left( b_1(k)+b_2(k) \right) + e^{c^h\Phi_L(k/N)}\left|e^{-\mu_k} - e^{-c^h\Phi_L(k/N)}\right| \nonumber \\
    & \qquad \stackrel{\eqref{eq:AGG_1_not_dN}}{\lesssim} k^{2h-3}p^{2h-1} + \left|e^{c^h\Phi_L(k/N)\big(1 - \frac{\mu_k}{c^h\Phi_L(k/N)}\big)} - 1\right| \lesssim N^{2h-3}p^{2h-1} + o(1) \label{eq:critical_poisson_conv} \\
    & \qquad \ll N^{2h-2}p^{2h} + N^{h-1}p^h + o(1) \stackrel{\eqref{eq:critical_decay_assumption}}{\lesssim} 1. \nonumber
\end{align}
In order, the claims in \eqref{eq:critical_poisson_conv} specifically follow from 
\begin{enumerate}
    \item the fact that $\IH_h(x)$ attains a maximum on $x \in [0, m/2]$ and that
    \begin{align*}
        k \geq K_1(N) \stackrel{\eqref{eq:K1_asymptotics}}{\gtrsim} 1/p,
    \end{align*}
    so that our invocation of \eqref{eq:AGG_1_not_dN} is justified;

    \item uniformly over $k \in [K_1(N), mN/2)$, it holds that 
    \begin{align*}
        \left|c^h\Phi_L(k/N)\left(1 - \frac{\mu_k}{c^h\Phi_L(k/N)}\right)\right| & \stackrel{\eqref{eq:critical_decay_assumption}}{=} \left|c^h\Phi_L(k/N)\left(1 - \frac{\mu_k}{\Phi_L(k/N) N^{h-1}p^h}\right)\right| \\
        & \stackrel{\eqref{eq:delta(N)}}{\leq} c^h\Phi_L(k/N) \delta(N) \ll 1.
    \end{align*}
\end{enumerate}
This proves the lemma.
\end{proof}

Equipped with \cref{lem:critical_poisson_conv}, we are ready to compute the expectations of interest.

\begin{proposition} \label{prop:critical_expectations}
    We have that
    \begin{align*}
        & \E\left[ |L(A)| \right] \sim \left(\sum_{i=1}^h |u_i| - 2\int_0^{m/2} e^{-c^h \Phi_L(x)} dx\right)N;
        & \E\left[ |L(A)^c| \right] \sim \left(2\int_0^{m/2} e^{-c^h \Phi_L(x)} dx \right) N.
    \end{align*}
\end{proposition}

\begin{proof}
    We first prove these expectation asymptotics for $|L_{\neq}(A)|$ and $|L_{\neq}(A)^c|$. It is clear that it suffices to prove one of the two corresponding statements. We prove the latter. We write
    \begin{align} \label{eq:critical_expectation_left_right}
        \E\left[ |L_{\neq}(A)^c| \right] = \E\left[ \left| L_{\neq}(A)^c \cap [-dN, -dN + mN/2] \right| \right] + \E\left[ \left| L_{\neq}(A)^c \cap (sN - mN/2, sN] \right| \right].
    \end{align}
    Using \eqref{eq:K1_asymptotics}, the first summand of \eqref{eq:critical_expectation_left_right} can be written as
    \begin{align} \label{eq:critical_fringes_expectation_1}
        o(N) + \E\left[ \left| L_{\neq}(A)^c \cap [-dN + K_1(N), -dN+mN/2] \right| \right].
    \end{align}
    We express the latter summand of \eqref{eq:critical_fringes_expectation_1} via
    \begin{align}
        & \sum_{k = K_1(N)}^{mN/2} \Prob\left[-dN+k \notin L_{\neq}(A) \right] \nonumber \\
        & \qquad = O(1) + \sum_{k = K_1(N)}^{mN/2} \Prob\left[-dN+k \notin L_{\neq}(A) \right] - e^{-c^h\Phi_L(k/N)} + e^{-c^h\Phi_L(k/N)} \nonumber \\
        & \qquad \stackrel{\text{(\cref{lem:critical_poisson_conv})}}{\sim} O(1) + \sum_{k=K_1(N)}^{mN/2} e^{-c^h\Phi_L(k/N)} \sim \int_{K_1(N)}^{mN/2} e^{-c^h\Phi_L(x/N)} dx \sim N\int_0^{m/2} e^{-c^h\Phi_L(x)} dx. \label{eq:critical_fringes_expectation_3}
    \end{align}
    Specifically, the final asymptotic equivalence of \eqref{eq:critical_fringes_expectation_3} follows since $K_1(N)/N \ll 1$ by \eqref{eq:K1_asymptotics}. It follows from \eqref{eq:critical_fringes_expectation_1} that \eqref{eq:critical_fringes_expectation_3} is asymptotically equivalent to the first summand of \eqref{eq:critical_expectation_left_right}. We now handle the second summand of \eqref{eq:critical_expectation_left_right}. Proceeding as in \eqref{eq:critical_fringes_expectation_1}, we may express this summand via
    \begin{equation} \label{eq:critical_fringes_expectation_5}
        \begin{aligned}
            & o(N) + \sum_{k=K_1(N)}^{mN/2} \Prob\left[ sN-k \notin L_{\neq}(A) \right] \stackrel{\eqref{eq:left_right_symmetry}}{=} o(N) + \sum_{k=K_1(N)}^{mN/2} \Prob\left[ -sN+k \notin (-L)_{\neq}(A) \right] \\
            & \qquad \sim N \int_0^{m/2} e^{-c^h\Phi_L(x)} dx.
        \end{aligned}
    \end{equation}
    Specifically, we have applied \eqref{eq:critical_fringes_expectation_3} in the context of the linear form $-L$. We conclude that
    \begin{align*} 
        \E\left[ |L_{\neq}(A)^c| \right] \stackrel{\eqref{eq:critical_expectation_left_right}, \eqref{eq:critical_fringes_expectation_3}, \eqref{eq:critical_fringes_expectation_5}}{\sim} 2N\int_0^{m/2} e^{-c^h\Phi_L(x)} dx, 
    \end{align*}
    thus establishing \cref{prop:critical_expectations} for $|L_{\neq}(A)|$ and $|L_{\neq}(A)^c|$. 

    \smallskip
    
    We now extend these results to $\E[|L(A)|]$ and $\E[|L(A)^c|]$. Towards this end, we let $\Pi_h$ denote the (finite) collection of set partitions $\pi = \{B_1, \dots, B_r\}$ of $[h]$ with $r<h$. For $\pi \in \Pi_h$ we define the \emph{collapsed linear form} $L^\pi$ via
    \begin{align*}
        L^\pi(y_1, \dots, y_r) := \sum_{j=1}^r \left( \sum_{i\in B_j} u_i \right) y_j.
    \end{align*}
    The collapsed linear form $L_{\pi}$ is in at most $r$ variables (after potentially deleting variables whose coefficient vanishes). Given a non-injective tuple $a = (a_1,\dots,a_h)\in A^h$, we let $\pi = \pi(a)$ be its collision pattern, i.e., the partition of $[h]$ defined by $i\sim_\pi j \iff a_i = a_j$. Writing $b_1, \dots, b_r$ for the distinct values assumed on the blocks $B_1, \dots, B_r$, we have that
    \begin{align*}
        L(a_1, \dots, a_h) = L^\pi(b_1, \dots, b_r).
    \end{align*}
    We thus have that
    \begin{align} \label{eq:critical_expectation_non_injective}
        L(A) \subseteq L_{\neq}(A) \cup \bigcup_{\pi \in \Pi_h} L_{\neq}^\pi(A) \implies |L(A) \setminus L_{\neq}(A)| \leq \sum_{\pi \in \Pi_h} |L^\pi(A)|.
    \end{align}
    It therefore follows that
    \begin{align*}
        \E[|L(A) \setminus L_{\neq}(A)|] \stackrel{\eqref{eq:critical_expectation_non_injective}}{\leq} \sum_{\pi \in \Pi_h} \E[|L^\pi(A)|] \stackrel{\eqref{eq:fast_decay_ubd}}{\lesssim} \E[|A|^{h-1}] \lesssim (Np)^{h-1} \stackrel{\eqref{eq:critical_decay_assumption}}{\ll} N,
    \end{align*}
    yielding \cref{prop:critical_expectations}.
\end{proof}

\subsection{Concentration}

We now show that the random variables $|L(A)|$ and $|L(A)^c|$ are strongly concentrated about their expectations. With \cref{prop:critical_expectations}, this proves \cref{thm:Z_linear_forms_ii}.

\begin{proposition} \label{prop:critical_concentration}
    If \eqref{eq:critical_decay_assumption} holds, then $|L(A)| \sim \E\left[ |L(A)| \right]$ and $|L(A)^c| \sim \E\left[ |L(A)^c| \right]$.
\end{proposition}

\begin{proof}
    We fix $n \in I_N$. We define
    \begin{align} \label{eq:interval_I}
        \mathscr{I}_n := \left[-dN+\min\{n,N-n\}, sN-\min\{n,N-n\}\right].
    \end{align}
    We also define\footnote{The fact that we introduce $K_3$ before $K_2$, as well as the definition of $K_3$ and much of the following argument, will seem unnatural for the critical regime. Indeed, in the present proof, the interval $[-dN + K_3(N), sN - K_3(N)]$ will be empty for large $N$, and some claims follow vacuously. We do this for notational compatibility with the proof of \cref{prop:slow_concentration}, where we follow exactly the same argument to prove the concentration of $|L(A)^c|$ about its expectation with some modifications mentioned there.} the function $K_3: \N \to \N$ via
    \begin{align} \label{eq:g_tilde}
        K_3(N) := (1/p)^{\frac{h}{h-1}}\log(1/p).
    \end{align}
    We respectively think of
    \begin{align*}
        & \left[-dN+K_3(N), sN-K_3(N)\right];
        & \left[-dN+K_3(N), sN-K_3(N)\right]^c
    \end{align*}
    as the ``middle" and the ``fringes" of the interval $\left[-dN, sN\right]$. Since 
    \begin{align*}
        L\left(A \setminus \{n\}\right) \subseteq L\left(A \cup \{n\}\right) \implies L\left(A \setminus \{n\}\right)^c \supseteq L\left(A \cup \{n\}\right)^c
    \end{align*}
    it follows that the integrand in the conditional expectation \eqref{eq:delta_function} can be written as
    \begin{align} \label{eq:adding_n}
        \big| L\left(A \setminus \{n\}\right) ^c \big| - \big| L\left(A \cup \{n\}\right)^c \big| & = \left| L\left(A \setminus \{n\}\right)^c \setminus L\left(A \cup \{n\}\right)^c \right| = \left| L\left(A \cup \{n\}\right) \setminus L\left(A \setminus \{n\}\right) \right|.
    \end{align}
    The final expression in \eqref{eq:adding_n} can be understood as the number of new elements that are added to the image set $L(A)$ due to the inclusion of $n$. Any such new element certainly must use $n$ as a summand in any sum that generates it. Therefore, it holds that
    \begin{align*}
        L\left(A \cup \{n\}\right) \setminus L\left(A \setminus \{n\}\right) \subseteq \mathscr{I}_n.
    \end{align*}
    We may thus decompose
    \begin{align} \label{eq:delta_reexpressed}
        \Delta_n(A) & = \E\left[ \left| L\left(A \cup \{n\}\right) \setminus L\left(A \setminus \{n\}\right) \right| \ \Big| \ \mathcal F_{n-1} \right] \nonumber \\
        & = \E\left[ \left| \left(L\left(A \cup \{n\}\right) \setminus L\left(A \setminus \{n\}\right) \right) \cap (\mathscr{I}_n \cap [-dN+K_3(N), sN-K_3(N)]) \right| \ \Big| \ \mathcal F_{n-1} \right] \nonumber \\
        & \qquad + \E\left[ \left| \left(L\left(A \cup \{n\}\right) \setminus L\left(A \setminus \{n\}\right) \right) \cap (\mathscr{I}_n \setminus [-dN+K_3(N), sN-K_3(N)]) \right| \ \Big| \ \mathcal F_{n-1} \right] \nonumber \\
        & =: \Delta_{n,1}(A) + \Delta_{n,2}(A).
    \end{align}
    Using \eqref{eq:critical_decay_assumption}, we deduce that $K_3(N) \gg N$, so it certainly holds that $\Delta_{n,1}(A) = 0$. We now turn to showing that $\Delta_{n,2}(A)$ is modest with high probability. If $n \in [K_3(N), N - K_3(N)]$, then
    \begin{align*}
        \min\{n, N-n\} \geq K_3(N),
    \end{align*}
    from which it follows that
    \begin{align*}
        & \mathscr{I}_n \setminus [-dN+K_3(N), sN-K_3(N)] \\
        & \qquad \stackrel{\eqref{eq:interval_I}}{=} \left[-dN+\min\{n,N-n\}, sN-\min\{n,N-n\}\right] \setminus [-dN+K_3(N), sN-K_3(N)] = \emptyset.
    \end{align*}
    This implies that
    \begin{align} \label{eq:delta_2_middle_bound}
        \Delta_{n,2}(A) = 0 \text{ for all } n \in [K_3(N), N - K_3(N)].
    \end{align}
    Combining \eqref{eq:delta_2_middle_bound} with a union bound for $\Delta_{n,1}(A)$ yields that with probability $1-o(1)$,
    \begin{align} \label{eq:delta_bound_middle}
        \Delta_n(A) = \Delta_{n,1}(A) + \Delta_{n,2}(A) = 0 \text{ for all } n \in [K_3(N), N - K_3(N)].
    \end{align}
    We now consider values $n \notin [K_3(N), N - K_3(N)]$. Here, we observe that
    \begin{equation} \label{eq:candidate_summands}
        \begin{aligned}
            & \Big| \left(L\left(A \cup \{n\}\right) \setminus L\left(A \setminus \{n\}\right) \right) \cap (\mathscr{I}_n \setminus [-dN+K_3(N), sN-K_3(N)]) \Big| \\
            & \qquad \leq \left(\big|A \cap [0, K_3(N)]\big| + \big|A \cap [N-K_3(N), N]\big| \right)^{h-1} \\
            & \qquad \lesssim \big|A \cap [0, K_3(N)]\big|^{h-1} + \big|A \cap [N-K_3(N), N]\big|^{h-1}.
        \end{aligned}
    \end{equation}
    Indeed, including the element $n$ in $A \setminus \{n\}$ may generate no more than
    \begin{align*}
        \left(\big|A \cap [0, K_3(N)]\big| + \big|A \cap [N-K_3(N), N]\big| \right)^{h-1}
    \end{align*}
    new elements from $\mathscr{I}_n \setminus [-dN+K_3(N), sN-K_3(N)]$. As before, any new element in $L(A)$ resulting from including $n$ in $A \setminus \{n\}$ must use $n$ as a summand in any sum which generates it. Furthermore, the remaining $h-1$ summands in any such sum must lie in
    \begin{align*}
        (A \cap [0, K_3(N)]) \cup (A \cap [N-K_3(N), N]),
    \end{align*}
    since the resulting sum lies in $[-dN+K_3(N), sN-K_3(N)]$ otherwise. By standard Chernoff bounds,
    \begin{equation} \label{eq:chernoff_fringes}
        \begin{aligned}
            \Prob\left( |A \cap [0, K_3(N)]| \leq 2(1/p)^{\frac{1}{h-1}}\log(1/p) \right) & \geq 1 - \exp\left( - (1/p)^{\frac{1}{h-1}}\log(1/p)/2 \right); \\
            \Prob\left( |A \cap [N-K_3(N), N]| \leq 2(1/p)^{\frac{1}{h-1}}\log(1/p) \right) & \geq 1 - \exp\left( -(1/p)^{\frac{1}{h-1}}\log(1/p)/2 \right).
        \end{aligned}
    \end{equation}
    We now study the conditional expectations $\Delta_{n,2}(A)$. If $n \leq K_3(N)$, then with implicit constants uniform over such $n$, we have that
    \begin{align*}
        \Delta_{n,2}(A) & \stackrel{\eqref{eq:candidate_summands}}{\leq} \E\left[ \big|A \cap [0, K_3(N)]\big|^{h-1} + \big|A \cap [N-K_3(N), N]\big|^{h-1} \ \Big| \ \mathcal F_{n-1} \right] \\
        & \lesssim \E\left[ \big|A \cap [0, n-1]\big|^{h-1} + \big|A \cap [n, K_3(N)]\big|^{h-1} + \big|A \cap [N-K_3(N), N]\big|^{h-1} \ \Big| \ \mathcal F_{n-1} \right] \\
        & = \big|A \cap [0, n-1]\big|^{h-1} + \E\left[ \big|A \cap [n, K_3(N)]\big|^{h-1} + \big|A \cap [N-K_3(N), N]\big|^{h-1} \right] \\
        & \stackrel{\left( n \leq K_3(N) \right)}{\lesssim} \big|A \cap [0, K_3(N)]\big|^{h-1} + \E\left[ \big|\Bin\left( K_3(N)+1, p \right)\big|^{h-1} \right] \\
        & \stackrel{\eqref{eq:g_tilde}}{\lesssim} \big|A \cap [0, K_3(N)]\big|^{h-1} + \left( 1/p \right) \left(\log (1/p)\right)^{h-1}.
    \end{align*}
    A similar computation gives for $n \geq N - K_3(N)$ that,  with implicit constants uniform over such $n$,
    \begin{align*}
        \Delta_{n,2}(A) \lesssim \big|A \cap [0, K_3(N)]\big|^{h-1} + \big|A \cap [N-K_3(N), N]\big|^{h-1} + \left( 1/p \right) \left(\log (1/p)\right)^{h-1}.
    \end{align*}
    Altogether, we conclude that there exists a constant $C > 0$ such that the following holds. For any $n \notin [K_3(N), N - K_3(N)]$, it holds that
    \begin{align} \label{eq:first_sum_bound_1}
        \Delta_{n,2}(A) > C\left((1/p)^{\frac{1}{h-1}}\log(1/p) \right)^{h-1} = C(1/p)\left(\log (1/p)\right)^{h-1} \text{ w.p. } \leq 2\exp\left( -\frac{\log(1/p)}{2p^{\frac{1}{h-1}}} \right).
    \end{align}
    A union bound over $n \notin [K_3(N), N-K_3(N)]$ now implies that
    \begin{align*}
        & \Prob\left( \Delta_{n,2}(A) > C(1/p)\left(\log (1/p)\right)^{h-1} \text{ for some } n \notin [K_3(N), N-K_3(N)] \right) \\
        & \qquad \stackrel{\eqref{eq:first_sum_bound_1}}{\lesssim} K_3(N) \exp\left( - \frac{\log(1/p)}{2p^{\frac{1}{h-1}}} \right) \lesssim \frac{\log (1/p)}{p^{\frac{h}{h-1}}} \exp\left( - \frac{\log(1/p)}{2p^{\frac{1}{h-1}}} \right) \\
        & \qquad \leq \left( \frac{\log (1/p)}{p^{\frac{1}{h-1}}} \right)^h \exp\left( - \frac{\log(1/p)}{2p^{\frac{1}{h-1}}} \right) \ll 1.
    \end{align*}
    Therefore, with probability $1-o(1)$,
    \begin{align} \label{eq:delta_2_fringe_bound}
        \Delta_{n,2}(A) \lesssim (1/p) \left(\log (1/p)\right)^{h-1} \text{ for all } n \notin [K_3(N), N - K_3(N)].
    \end{align}
    Combining \eqref{eq:delta_2_fringe_bound} with a union bound for $\Delta_{n,1}(A)$ yields that with probability $1 - o(1)$,
    \begin{equation} \label{eq:delta_bound_fringes}
        \begin{aligned}
            \Delta_n(A) = \Delta_{n,1}(A) + \Delta_{n,2}(A) & \lesssim o(1) + (1/p) \left(\log (1/p)\right)^{h-1} \\
            & \lesssim (1/p) \left(\log (1/p)\right)^{h-1} \text{ for all } n \notin [K_3(N), N - K_3(N)].
        \end{aligned}
    \end{equation}
    Therefore, by combining \eqref{eq:delta_bound_middle} and \eqref{eq:delta_bound_fringes} under a union bound, we deduce that with probability $1-o(1)$,
    \begin{align} \label{eq:C(A)}
        C(A) \stackrel{\eqref{eq:C(A)_reformulation}}{\sim} \max_{n \in I_N} \Delta_n(A) \lesssim (1/p) \left(\log (1/p)\right)^{h-1}.
    \end{align}
    On this event with $1-o(1)$ probability, we also deduce from \eqref{eq:delta_bound_middle} and \eqref{eq:delta_bound_fringes} that
    \begin{equation} \label{eq:V(A)}
        \begin{aligned}
            V(A) & \stackrel{\eqref{eq:V(A)_reformulation}}{\sim} p\sum_{n=0}^N \left(\Delta_n(A)\right)^2 = p\sum_{n=K_3(N)}^{N-K_3(N)} \left(\Delta_n(A)\right)^2 + p\sum_{n \notin [K_3(N), N-K_3(N)]} \left(\Delta_n(A)\right)^2 \\
            & \lesssim p \cdot o(1) +  p \cdot K_3(N) \left( (1/p) \left(\log (1/p)\right)^{h-1} \right)^2 \\
            & \lesssim o(1) + \frac{\log (1/p)}{p^{\frac{1}{h-1}}} (1/p)^2 \left( \log (1/p) \right)^{2(h-1)} \lesssim (1/p)^{2+\frac{1}{h-1}} \left(\log (1/p)\right)^{2h-1}.
        \end{aligned}
    \end{equation}
    We finish the proof by invoking \cref{thm:vu_lemma}. Specifically, we take
    \begin{align} \label{eq:vu_parameters}
        \lambda \asymp \log (1/p), \qquad \mathbf{V} \asymp (1/p)^{2+\frac{1}{h-1}} \left(\log (1/p)\right)^{2h-1}, \qquad \mathbf{C} \asymp (1/p) \left( \log (1/p)\right)^{h-1},
    \end{align}
    and take the constant factors implicit in our expressions for $\mathbf{C}$ and $\mathbf{V}$ to agree with those implied in \eqref{eq:C(A)} and \eqref{eq:V(A)}, respectively. It follows from our choices in \eqref{eq:vu_parameters} that
    \begin{align}
        & \lambda \lesssim \log (1/p) \lesssim \mathbf{V}/\mathbf{C}^2 = (1/p)^{\frac{1}{h-1}}\log (1/p), \label{eq:lambda} \\
        & \sqrt{\lambda \mathbf{V}} \asymp \sqrt{\log (1/p) \cdot (1/p)^{2+\frac{1}{h-1}} \left(\log (1/p)\right)^{2h-1}} = (1/p)^{1+\frac{1}{2(h-1)}} \left(\log (1/p)\right)^h \ll (1/p)^{\frac{h}{h-1}}. \label{eq:root_lambda_V}
    \end{align}
    We now observe that 
    \begin{align*}
        N \stackrel{\eqref{eq:critical_decay_assumption}}{\asymp} (1/p)^{\frac{h}{h-1}} \stackrel{\eqref{eq:root_lambda_V}}{\implies} \sqrt{\lambda \mathbf{V}} \ll N.
    \end{align*}
    Since the right-hand sides of both expressions in \eqref{eq:critical_decay} are $\Omega(N)$, it follows from \cref{thm:vu_lemma} and \eqref{eq:root_lambda_V} that to prove the proposition, it suffices to show that
    \begin{align} \label{eq:vanishing_RHS}
        2e^{-\lambda/4} & + \Prob(C(A) \geq \mathbf{C}) + \Prob(V(A) \geq \mathbf{V}) \ll 1.
    \end{align}
    Since $\lambda \gg 1$, we have $2e^{-\lambda/4} \ll 1$. The latter two terms in the LHS of \eqref{eq:vanishing_RHS} vanish by \eqref{eq:C(A)} and \eqref{eq:V(A)}.
\end{proof}

\cref{prop:critical_expectations,prop:critical_concentration} together prove \cref{thm:Z_linear_forms_ii}.

\section{Supercritical Decay} \label{sec:slow_decay}

Finally, we prove \cref{thm:Z_linear_forms_iii}, completing the proof of \cref{thm:Z_linear_forms}. The $h=2$ case was done by \cite[Theorem 3.1(iii)]{hegarty2009almost}, so unless otherwise stated, we assume that $h \geq 3$ throughout this section. We also assume throughout this section that 
\begin{align} \label{eq:slow_decay_assumption}
    Np^{\frac{h}{h-1}} \gg 1 \iff N^{h-1}p^h \gg 1.
\end{align}

\subsection{Reduction to Fringe Elements} \label{subsec:reductions}

We set up expressions for the setting in which \eqref{eq:slow_decay_assumption} holds. We let $w: \N \to \R_{\geq 0}$ be a function (a \textit{width} parameter) satisfying
\begin{equation} \label{eq:w_asymptotics}
    \begin{aligned}
        & 1 \ll w(N) \ll \min\left\{Np^{\frac{h}{h-1}}, \left(1/\delta(N)\right)^{\frac{1}{h-1}}, (1/p)^{1/(h-1)^2} \right\}; \\
        & \exp\left(\frac{w(N)^{h-1}}{(h-1)! \cdot |\Aut(L)| \cdot \prod_{i=1}^h |u_i|} \right)w(N)^{2h-3}p^{\frac{1}{h-1}} \ll 1.
    \end{aligned}
\end{equation}
By taking $w$ to be sufficiently slowly growing, it is readily observed that such a function satisfying all of the conditions in \eqref{eq:w_asymptotics} exists. We define the cutoff function $K_2: \N \to \N$ via
\begin{align} \label{eq:K2}
    K_2(N) := (1/p)^{\frac{h}{h-1}} \cdot w(N).
\end{align}
From \eqref{eq:w_asymptotics} and \eqref{eq:K2}, it follows that $K_2(N) \ll N$. We respectively think of 
\begin{align*}
    & [-dN + K_2(N), sN - K_2(N)];
    & [-dN + K_2(N), sN - K_2(N)]^c
\end{align*}
as the ``middle" and the ``fringes" of the interval $[-dN, sN]$. In this subsection, we reduce the computation of $\E\left[ | L(A)^c | \right]$ to the fringes. We begin by deriving an asymptotic lower bound for the expected number of elements in $[-dN, sN]$ which are missing sums in $L(A)$.
\begin{lemma} \label{lem:missing_fringes}
    It holds that
    \begin{align*}
        \E \left[ | L_{\neq}(A)^c | \right] \gtrsim \left( 1/p \right)^{\frac{h}{h-1}}.
    \end{align*}
\end{lemma}

\begin{proof}
    We restrict our attention to the subinterval
    \begin{align*}
        \mathcal J := \left[-dN, -dN + (1/p)^{\frac{h}{h-1}} \right] \stackrel{\eqref{eq:slow_decay_assumption}}{\subseteq} [-dN, sN].
    \end{align*}
    An element in $L_{\neq}(A)$ which is generated by adding a term greater than $(1/p)^{\frac{h}{h-1}}$ or subtracting a term less than $N-(1/p)^{\frac{h}{h-1}}$ fails to lie in $\mathcal J$. Thus, any element in $L_{\neq}(A) \cap \mathcal J$ must strictly add terms that are at most $(1/p)^{\frac{h}{h-1}}$ and subtract terms that are at least $N-(1/p)^{\frac{h}{h-1}}$. It follows that
    \begin{align} \label{eq:missing_fringes_1}
        |\mathcal J \cap L_{\neq}(A)| \leq \frac{1}{|\Aut(L)|} \cdot \prod_{i=1}^{h_{\text{pos}}} \left| A \cap \left[ 0, (1/p)^{\frac{h}{h-1}}/u_i \right] \right| \cdot \prod_{j = h_{\text{pos}}+1}^h \left| A \cap \left[ N- (1/p)^{\frac{h}{h-1}}/|u_j|, N \right] \right|.
    \end{align}
    It follows from standard Chernoff bounds that for all $i \in [h_{\text{pos}}]$ and $j \in \{h_{\text{pos}}+1, \dots, h\}$,
    \begin{align} \label{eq:missing_fringes_a}
        & \left| A \cap \left[ 0, (1/p)^{\frac{h}{h-1}}/u_i \right] \right| \sim (1/p)^{\frac{1}{h-1}}/u_i,
        & \left| A \cap \left[ N- (1/p)^{\frac{h}{h-1}}/|u_j|, N \right] \right| \sim (1/p)^{\frac{1}{h-1}}/|u_j|.
    \end{align}
    Thus, it holds with high probability that
    \begin{align} \label{eq:missing_fringes_bd}
         |\mathcal J \cap L_{\neq}(A)| \stackrel{\eqref{eq:missing_fringes_1}, \eqref{eq:missing_fringes_a}}{\leq} \frac{1+o(1)}{|\Aut(L)|} \cdot \prod_{i=1}^{h_{\text{pos}}} \left(\frac{(1/p)^{\frac{1}{h-1}}}{u_i} \right)  \cdot \prod_{j = h_{\text{pos}}+1}^h \left( \frac{(1/p)^{\frac{1}{h-1}}}{|u_j|} \right) = \frac{(1+o(1))(1/p)^{\frac{h}{h-1}}}{|\Aut(L)| \cdot \prod_{i=1}^h |u_i|}.
    \end{align}
    Therefore, it holds with high probability that 
    \begin{align} \label{eq:missing_fringes_4}
        |\mathcal J \cap L_{\neq}(A)^c| \stackrel{\eqref{eq:missing_fringes_bd}}{\geq} (1/p)^{\frac{h}{h-1}} - \frac{(1+o(1))(1/p)^{\frac{h}{h-1}}}{|\Aut(L)| \cdot \prod_{i=1}^h |u_i|} = \left[1 - \frac{1+o(1)}{|\Aut(L)| \cdot \prod_{i=1}^h |u_i|} \right](1/p)^{\frac{h}{h-1}} \gtrsim (1/p)^{\frac{h}{h-1}}.
    \end{align}
    In particular, the final claim of \eqref{eq:missing_fringes_4} follows since 
    \begin{align*}
        |\Aut(L)| \cdot \prod_{i=1}^h |u_i| \geq 2.
    \end{align*}
    Indeed, as $h \geq 3$, $|u_i| = 1$ for all $i \in [h]$ would imply that $|\Aut(L)| \geq 2$, since either $1$ or $-1$ appears as a coefficient at least twice. This establishes \cref{lem:missing_fringes}.
\end{proof}

We now prove \cref{lem:missing_middle}, which shows that the lower bound of \cref{lem:missing_fringes} dominates the number of elements in the middle of $[-dN, sN]$ that are missing from $L_{\neq}(A)$.

\begin{lemma} \label{lem:missing_middle}
    The expected number of elements in the interval $[-dN+K_2(N), sN-K_2(N)]$ that are missing from $L_{\neq}(A)$ satisfies
    \begin{align*}
        \E \left[ \left| L_{\neq}(A)^c \cap [-dN+K_2(N), sN-K_2(N)] \right| \right] \ll \left( 1/p \right)^{\frac{h}{h-1}}.
    \end{align*}
\end{lemma}

\begin{proof}
    We fix $k \in [K_2(N), mN/2)$. It follows from \eqref{eq:w_asymptotics} and \eqref{eq:K2} that $k^{h-1}p^h \gtrsim 1$ uniformly over such $k$. We take $\mu_k$ and $b_2(k)$ as respectively defined in the proof of \cref{lem:slow_poisson_conv} and in \cref{sec:computations}. Uniformly over $k \in [K_2(N), mN/2)$, it holds that
    \begin{align} \label{eq:missing_middle_1}
        -\frac{\mu_k^2}{b_2(k)} \stackrel{\eqref{eq:basic_computation_mean}, \eqref{eq:b_2(k)_bound}}{\lesssim} -\frac{k^{2h-2}p^{2h}}{k^{2h-3}p^{2h-1}} = -kp.
    \end{align}
    There exists $\Lambda \in \mathscr{E}_k$ for which $S(\Lambda) \subseteq A$ if and only if $-dN+k \in L_{\neq}(A)$. Following \cref{thm:janson_ineqs}, we let $E_\Lambda$ denote the event that $S(\Lambda) \subseteq A$. Now, it holds for some constant $C > 0$ that
    \begin{align} \label{eq:missing_middle_3}
        \Prob\left[-dN+k \notin L_{\neq}(A)\right] = \Prob\left[ \bigwedge_{\Lambda \in \mathscr{E}_k} \overline{E_\Lambda} \right] \stackrel{\eqref{eq:janson_main_bound}}{\leq} e^{-\frac{\mu_k^2}{2b_2(k)}} + e^{-\mu_k/2} \stackrel{\eqref{eq:basic_computation_mean},\eqref{eq:missing_middle_1}}{\leq} e^{-Ckp} + e^{-Ck^{h-1}p^h},
    \end{align}
    where \eqref{eq:missing_middle_3} holds uniformly over $k \in [K_2(N), mN/2)$. We deduce that
    \begin{align}
        & \E \left[ \left| L_{\neq}(A)^c \cap [-dN+K_2(N), (s-m/2)N) \right| \right] = \sum_{k = K_2(N)}^{mN/2-1} \Prob\left[-dN+k \notin L_{\neq}(A)\right] \nonumber \\
        & \qquad \leq \sum_{k = K_2(N)}^{mN/2-1} e^{-Ckp} + e^{-Ck^{h-1}p^h/2} \leq \int_{K_2(N)-1}^\infty e^{-Cxp} \ dx + \int_{K_2(N)-1}^\infty e^{-Cx^{h-1}p^h} \ dx \nonumber \\
        & \qquad \leq \int_{K_2(N)-1}^\infty e^{-Cxp} \ dx + \left( 1/p \right)^{h/(h-1)} \int_{(K_2(N)-1)p^{h/(h-1)}}^\infty e^{-Cx^{h-1}} \ dx \nonumber \\
        & \qquad = \frac{1}{Cp}\exp\left( -Cp(K_2(N)-1)\right) + \left( 1/p \right)^{h/(h-1)} \int_{w(N) - p^{h/(h-1)}}^\infty e^{-Cx^{h-1}} \ dx \nonumber \\
        & \qquad \stackrel{\eqref{eq:w_asymptotics}, \eqref{eq:K2}}{\lesssim} (1/p) \exp\left( C(p - w(N))\right) + \left( 1/p \right)^{h/(h-1)} o(1) \stackrel{\eqref{eq:w_asymptotics}}{\ll} \left( 1/p \right)^{\frac{h}{h-1}}.  \label{eq:missing_middle_4}
    \end{align}
    It now follows from \eqref{eq:left_right_symmetry} that 
    \begin{align}
        \E \left[ \left| L_{\neq}(A)^c \cap ((-d+m/2)N, sN - K_2(N)] \right| \right] = \sum_{k \in [K_2(N), mN/2)} \Prob[sN-k \notin L_{\neq}(A)] \nonumber \\
        \stackrel{\eqref{eq:left_right_symmetry}}{=} \sum_{k \in [K_2(N), mN/2)} \Prob[-sN+k \notin (-L)_{\neq}(A)] \stackrel{\eqref{eq:missing_middle_4}}{\ll} \left( 1/p \right)^{\frac{h}{h-1}}, \label{eq:missing_middle_5}
    \end{align}
    where we have invoked \eqref{eq:missing_middle_4} with respect to $-L$ to observe \eqref{eq:missing_middle_5}. We conclude that
    \begin{align*}
        & \E \left[ \left| L_{\neq}(A)^c \cap [-dN+K_2(N), sN-K_2(N)] \right| \right] \\
        & \quad = \E \left[ \left| L_{\neq}(A)^c \cap [-dN+K_2(N), (s-m/2)N] \right| \right] + \E \left[ \left| L_{\neq}(A)^c \cap ((-d+m/2)N, sN - K_2(N)] \right| \right] \\
        & \quad \stackrel{\eqref{eq:missing_middle_4},\eqref{eq:missing_middle_5}}{\ll} \left( 1/p \right)^{\frac{h}{h-1}}.
    \end{align*}
    This yields the desired statement.
\end{proof}

Altogether, we conclude that
\begin{equation} \label{eq:reduction_to_fringes_obj_2}
    \begin{aligned}
        & \E\left[ | L_{\neq}(A)^c | \right] \\
        & = \E \left[ \left| L_{\neq}(A)^c \cap [-dN+K_2(N), sN-K_2(N)] \right| \right] + \E \left[ \left| L_{\neq}(A)^c \cap [-dN+K_2(N), sN-K_2(N)]^c \right| \right] \\
        & \stackrel{\text{(\cref{lem:missing_fringes,lem:missing_middle})}}{\sim} \E\left[ \left| L_{\neq}(A)^c \cap [-dN+K_2(N), sN-K_2(N)]^c \right| \right],
    \end{aligned}
\end{equation}
which provides the desired reduction.

\begin{remark} \label{rmk:forbidden_summand}
    It is straightforward to adapt the exponentially decaying bound \eqref{eq:missing_middle_3} of \cref{lem:missing_middle} if we were to enforce the condition that some particular value $n \in I_N$ cannot be in any subset of $\mathscr{E}_k$. Indeed, the asymptotic claim implicit in \eqref{eq:missing_middle_1} would still hold uniformly over $k \in [K_2(N), mN/2]$. This is since for such $k$, the number of $h$-tuples $(a_1, \dots, a_h)$ with at least one instance of $n$ such that 
    \begin{align*}
        L(a_1, \dots, a_h) = -dN+k
    \end{align*}
    is $O(k^{h-2})$, so $\mu_k$ would remain of the same order. More specifically, by slightly adapting the proof of \cref{lem:missing_middle}, we can show that there exists a constant $C > 0$ (independent of $n$) for which it holds for all $k \in \left[(1/p)^{\frac{h}{h-1}}, mN/2 \right]$ that
    \begin{gather*} 
        \Prob\left[-dN + k \notin L\left(A \setminus \{n\}\right)\right] \leq \Prob\left[-dN + k \notin L_{\neq} \left(A \setminus \{n\}\right)\right] \leq e^{-Ckp} + e^{-Ck^{h-1}p^h}; \\
        \Prob\left[sN - k \notin L\left(A \setminus \{n\}\right)\right] \leq \Prob\left[sN - k \notin L_{\neq}\left(A \setminus \{n\}\right)\right] \leq e^{-Ckp} + e^{-Ck^{h-1}p^h}.
    \end{gather*}
    We make use of this remark in the proof of \cref{prop:slow_concentration}.
\end{remark}

\subsection{Expectation} \label{subsec:slow_expectation}

In this subsection, we compute $\E\left[ | L(A)^c | \right]$. We begin with the following observation, which is the analogue of \cref{lem:critical_poisson_conv} for this regime. Notably, \cref{lem:slow_poisson_conv} and its proof hold without modification for the $h=2$ setting. We make use of this observation later on in the proof of \cref{prop:slow_fringes_expectation}.

\begin{lemma} \label{lem:slow_poisson_conv}
    Uniformly over all $k \in [K_1(N), K_2(N)]$, it holds that
    \begin{align*}
       \left|\Prob\left[-dN+k \notin L_{\neq}(A)\right] - e^{-\Phi_L(k/N)N^{h-1}p^h} \right| \ll e^{-\Phi_L(k/N)N^{h-1}p^h}.
    \end{align*}
\end{lemma}

\begin{proof}
We proceed as in the proof of \cref{lem:critical_poisson_conv}. Uniformly over $k \in [ K_1(N), K_2(N) ]$,
\begin{align} \label{eq:gamma_1}
    \mu_k \sim |\mathscr{E}_k| \cdot p^h \sim \Phi_L(k/N) N^{h-1}p^h \stackrel{\eqref{eq:critical_regime_exponential_term}, (K_2(N) \ll N))}{=} \frac{p^hk^{h-1}}{(h-1)! \cdot |\Aut(L)| \cdot \prod_{i=1}^h |u_i|}. 
\end{align}
Here, \eqref{eq:gamma_1} follows as in \eqref{eq:critical_gamma_1}. Furthermore, it holds uniformly over $k \in [K_1(N), K_2(N)]$ that
\begin{equation} \label{eq:margin_bd}
    \begin{aligned}
        & \left| \Phi_L(k/N)N^{h-1}p^h\left(1 - \frac{\mu_k}{\Phi_L(k/N)N^{h-1}p^h}\right) \right| \stackrel{\eqref{eq:delta(N)}, \eqref{eq:gamma_1}}{\leq} \frac{p^hk^{h-1}}{(h-1)! \cdot |\Aut(L)| \cdot \prod_{i=1}^h |u_i|} \cdot \delta(N) \\
        & \qquad \stackrel{(k \leq K_2(N))}{\leq} \frac{p^hK_2(N)^{h-1}}{(h-1)! \cdot |\Aut(L)| \cdot \prod_{i=1}^h |u_i|} \cdot \delta(N) \stackrel{\eqref{eq:K2}}{\leq} w(N)^{h-1} \delta(N) \stackrel{\eqref{eq:w_asymptotics}}{\ll} 1.
    \end{aligned}
\end{equation}
It now follows that uniformly over $k \in [K_1(N), K_2(N)]$,
\begin{align*}
    & e^{\Phi_L(k/N)N^{h-1}p^h}\left|\Prob\left[-dN+k \notin L_{\neq}(A) \right] - e^{-\Phi_L(k/N)N^{h-1}p^h} \right| \\
    & = e^{\Phi_L(k/N)N^{h-1}p^h} \left|(\Prob\left[W_k = 0 \right] - e^{-\mu_k}) + (e^{-\mu_k} - e^{-\Phi_L(k/N)N^{h-1}p^h}) \right| \\
    & \stackrel{\text{(\cref{thm:stein_chen})}}{\leq} \exp\left(\frac{p^hk^{h-1}}{(h-1)! \cdot |\Aut(L)| \cdot \prod_{i=1}^h |u_i|} \right) \left(b_1(k)+b_2(k)\right) \\
    & \qquad \qquad \qquad \qquad + e^{\Phi_L(k/N)N^{h-1}p^h}\left|e^{-\mu_k} - e^{-\Phi_L(k/N)N^{h-1}p^h}\right| \\
    & \stackrel{\eqref{eq:AGG_1_not_dN}, \eqref{eq:K1_asymptotics}}{\lesssim} \exp\left(\frac{\left(K_2(N) p^{\frac{h}{h-1}} \right)^{h-1}}{(h-1)! \cdot |\Aut(L)| \cdot \prod_{i=1}^h |u_i|} \right) k^{2h-3}p^{2h-1} + \left|e^{\Phi_L(k/N)N^{h-1}p^h\left(1 - \frac{\mu_k}{\Phi_L(k/N)N^{h-1}p^h}\right)} - 1\right| \\
    & \stackrel{\eqref{eq:margin_bd}, (k \leq K_2(N))}{=} \exp\left(\frac{w(N)^{h-1}}{(h-1)! \cdot |\Aut(L)| \cdot \prod_{i=1}^h |u_i|} \right) K_2(N)^{2h-3}p^{2h-1} + o(1) \\
    & \stackrel{\eqref{eq:K2}}{=} \exp\left(\frac{w(N)^{h-1}}{(h-1)! \cdot |\Aut(L)| \cdot \prod_{i=1}^h |u_i|} \right) \left( w(N) p^{-\frac{h}{h-1}} \right)^{2h-3} p^{2h-1} + o(1) \\
    & \lesssim \exp\left(\frac{w(N)^{h-1}}{(h-1)! \cdot |\Aut(L)| \cdot \prod_{i=1}^h |u_i|} \right) w(N)^{2h-3} p^{\frac{1}{h-1}} + o(1) \stackrel{\eqref{eq:w_asymptotics}}{\ll} 1.
\end{align*}
This proves the lemma.
\end{proof}

Equipped with \cref{lem:slow_poisson_conv}, we are ready to compute the expectation of interest.
\begin{proposition} \label{prop:slow_fringes_expectation}
    We have that
    \begin{align*} 
        \E\left[ |L(A)^c| \right] \sim \frac{2 \cdot \Gamma\left( \frac{1}{h-1} \right) \sqrt[h-1]{(h-1)! \cdot |\Aut(L)| \cdot \prod_{i=1}^{h} |u_i|}}{(h-1) \cdot p(N)^{\frac{h}{h-1}}}.
    \end{align*}
\end{proposition}

\begin{proof}
    We first prove the expectation asymptotic for $|L_{\neq}(A)^c|$. By \eqref{eq:reduction_to_fringes_obj_2}, it suffices to show that 
    \begin{align*}
        \E \left[ \left| L_{\neq}(A)^c \cap [-dN+K_2(N), sN-K_2(N)]^c \right| \right] \sim \frac{2 \cdot \Gamma\left( \frac{1}{h-1} \right) \sqrt[h-1]{(h-1)! \cdot |\Aut(L)| \cdot \prod_{i=1}^{h} |u_i|}}{(h-1) \cdot p(N)^{\frac{h}{h-1}}}.
    \end{align*}
    We begin with the left fringe. Using \eqref{eq:K1_asymptotics}, we write
    \begin{align}
        & \E\left[ \left| L_{\neq}(A)^c \cap [-dN,-dN+K_2(N)] \right| \right] \nonumber \\
        & \qquad = o\left( (1/p)^{\frac{h}{h-1}} \right) + \E\left[ \left| L_{\neq}(A)^c \cap [-dN + K_1(N),-dN+K_2(N)] \right| \right]. \label{eq:fringes_expectation_1}
    \end{align}
    We express the latter term as
    \begin{align}
        & \E\left[ \left| L_{\neq}(A)^c \cap [-dN + K_1(N), -dN+K_2(N)] \right| \right] = \sum_{k = K_1(N)}^{K_2(N)} \Prob\left[-dN+k \notin L_{\neq}(A) \right] \nonumber\\
        & \quad \stackrel{\text{(\cref{lem:slow_poisson_conv})}}{=} \sum_{k = K_1(N)}^{K_2(N)} \left( \Prob\left[-dN+k \notin L_{\neq}(A) \right] - e^{-\Phi_L(k/N)N^{h-1}p^h} + e^{-\Phi_L(k/N)N^{h-1}p^h} \right) \nonumber \\
        & \quad \sim \sum_{k=K_1(N)}^{K_2(N)} e^{-\Phi_L(k/N)N^{h-1}p^h} \stackrel{\eqref{eq:gamma_1}}{=} \sum_{k=K_1(N)}^{K_2(N)} \exp\left(-\frac{p^hk^{h-1}}{(h-1)! \cdot |\Aut(L)| \cdot \prod_{i=1}^h |u_i|} \right) \nonumber \\
        & \quad \sim \int_{K_1(N)}^{K_2(N)} \exp\left(-\frac{p^hx^{h-1}}{(h-1)! \cdot |\Aut(L)| \cdot \prod_{i=1}^h |u_i|} \right) \ dx \nonumber \\
        & \quad \sim \sqrt[h-1]{\frac{(h-1)! \cdot |\Aut(L)| \cdot \prod_{i=1}^h |u_i|}{p^h}} \int_{K_1(N) \sqrt[h-1]{\frac{p^h}{(h-1)! \cdot |\Aut(L)| \cdot \prod_{i=1}^h |u_i|}}}^{K_2(N) \sqrt[h-1]{\frac{p^h}{(h-1)! \cdot |\Aut(L)| \cdot \prod_{i=1}^h |u_i|}}} e^{-x^{h-1}} \ dx \nonumber \\
        & \quad \stackrel{\text{(DCT)}}{\sim} \frac{\sqrt[h-1]{(h-1)! \cdot |\Aut(L)| \cdot \prod_{i=1}^h |u_i|}}{p^{\frac{h}{h-1}}} \int_0^\infty e^{-x^{h-1}} \ dx \nonumber \\
        & \quad = \frac{\Gamma\left( \frac{1}{h-1} \right)\sqrt[h-1]{(h-1)! \cdot |\Aut(L)| \cdot \prod_{i=1}^h |u_i|}}{(h-1) \cdot p^{\frac{h}{h-1}}}. \label{eq:fringes_expectation_5}
    \end{align}
    Notably, the invocation of the dominated convergence theorem holds since \eqref{eq:K1_asymptotics} and \eqref{eq:w_asymptotics} imply that the lower and upper limits of the integral respectively tend to zero and infinity. Now, we have that
    \begin{align*}
        \E\left[ \left| L_{\neq}(A)^c \cap [-dN, -dN+K_2(N)] \right| \right] \stackrel{\eqref{eq:fringes_expectation_1},\eqref{eq:fringes_expectation_5}}{\sim} \frac{\Gamma\left( \frac{1}{h-1} \right)\sqrt[h-1]{(h-1)! \cdot |\Aut(L)| \cdot \prod_{i=1}^h |u_i|}}{(h-1) \cdot p^{\frac{h}{h-1}}}.
    \end{align*}
    We now handle the right fringe. Proceeding as in \eqref{eq:fringes_expectation_1},
    \begin{equation} \label{eq:fringes_expectation_7}
        \begin{aligned}
            & \E\left[ \left| L_{\neq}(A)^c \cap [sN-K_2(N), sN] \right| \right] = o\left( (1/p)^{\frac{h}{h-1}} \right) + \sum_{k=K_1(N)}^{K_2(N)} \Prob\left[ sN-k \notin L_{\neq}(A) \right] \\
            & \quad \stackrel{\eqref{eq:left_right_symmetry}}{=} o\left( (1/p)^{\frac{h}{h-1}} \right) + \sum_{k=K_1(N)}^{K_2(N)} \Prob\left[ -sN+k \notin (-L)_{\neq}(A) \right] \\
            & \quad \stackrel{\eqref{eq:fringes_expectation_5}}{\sim} \frac{\Gamma\left( \frac{1}{h-1} \right)\sqrt[h-1]{(h-1)! \cdot |\Aut(L)| \cdot \prod_{i=1}^h |u_i|}}{(h-1) \cdot p^{\frac{h}{h-1}}}, 
        \end{aligned}
    \end{equation}
    where we have applied \eqref{eq:fringes_expectation_5} in the context of the linear form $-L$. Altogether, we conclude that
    \begin{equation} \label{eq:injective_model_expectation}
        \begin{aligned}
            & \E \left[ \left| L_{\neq}(A)^c \cap [-dN+K_2(N), sN-K_2(N)]^c \right| \right] \\
            & \qquad \stackrel{\eqref{eq:fringes_expectation_5}, \eqref{eq:fringes_expectation_7}}{\sim} \frac{2 \cdot \Gamma\left( \frac{1}{h-1} \right)\sqrt[h-1]{(h-1)! \cdot |\Aut(L)| \cdot \prod_{i=1}^h |u_i|}}{(h-1) \cdot p^{\frac{h}{h-1}}},
        \end{aligned}
    \end{equation}
    thus establishing \cref{prop:slow_fringes_expectation} for $|L_{\neq}(A)^c|$. 

    \smallskip
    
    We conclude by extending this result to $|L(A)^c|$. We observe that
    \begin{align*}
        L_{\neq}(A)^c = L(A)^c \sqcup \left( L(A) \setminus L_{\neq}(A) \right) \implies |L(A)^c| = |L_{\neq}(A)^c| - |L(A) \setminus L_{\neq}(A)|.
    \end{align*}
    It thus suffices to show that 
    \begin{align*}
        \E[|L(A) \setminus L_{\neq}(A)|] \ll (1/p)^{\frac{h}{h-1}}.
    \end{align*}
    Towards this end, we observe that (using the notation introduced in the proof of \cref{prop:critical_expectations})
    \begin{align*}
        & \E\left[ \left| L(A) \setminus L_{\neq}(A) \right| \right] \\
        & \qquad \leq \E\left[ \left| (L(A) \setminus L_{\neq}(A)) \cap [-dN+K_2(N), sN-K_2(N)]^c \right| \right] \\
        & \qquad \qquad + \E\left[ \left| L_{\neq}(A)^c \cap [-dN+K_2(N), sN-K_2(N)] \right| \right] \\
        & \qquad \stackrel{\eqref{eq:reduction_to_fringes_obj_2},\eqref{eq:injective_model_expectation}}{=} \E\left[ \left| (L(A) \setminus L_{\neq}(A)) \cap [-dN+K_2(N), sN-K_2(N)] \right| \right] +  o\left( (1/p)^{\frac{h}{h-1}} \right) \\ 
        & \qquad \stackrel{\eqref{eq:critical_expectation_non_injective}}{\leq} \sum_{\pi \in \Pi_h} \E\left[ \left| L^\pi_{\neq}(A) \cap [-dN+K_2(N), sN-K_2(N)]^c \right| \right] +  o\left( (1/p)^{\frac{h}{h-1}} \right).
    \end{align*}
    We fix some collision pattern $\pi \in \Pi_h$, and we let $r < h$ denote the number of variables that the collapsed linear form $L^\pi$ is in. We first consider the setting in which $2 \leq r \leq h-1$. Since $r/(r-1) > h/(h-1)$, we may choose the $r$-variable cutoffs in \cref{lem:slow_poisson_conv} so that their admissible $k$-interval contains $[K_1(N),K_2(N)]$. Therefore, adapting the computations in \eqref{eq:fringes_expectation_5} gives
    \begin{align*}
        & \E\left[ \left| L^\pi_{\neq}(A) \cap [-dN, -dN+K_2(N)] \right| \right] \sim o\left( (1/p)^{\frac{h}{h-1}} \right) + \int_0^{K_2(N)} 1-e^{-C_\pi p^rx^{r-1}} dx \\
        & \quad \leq o\left( (1/p)^{\frac{h}{h-1}} \right) + K_2(N) \cdot \left( 1 - e^{-C_\pi p^r K_2(N)^{r-1}} \right) \stackrel{\left( 1 - e^{-x} \leq x \right)}{\lesssim} o\left( (1/p)^{\frac{h}{h-1}} \right) + K_2(N) \cdot p^r K_2(N)^{r-1} \\
        & \quad = o\left( (1/p)^{\frac{h}{h-1}} \right) + \left( K_2(N) p \right)^r \stackrel{\eqref{eq:K2}}{=} o\left( (1/p)^{\frac{h}{h-1}} \right) + \left( w(N) \cdot (1/p)^{1/(h-1)} \right)^r \\
        & \quad \stackrel{(r < h)}{\leq} o\left( (1/p)^{\frac{h}{h-1}} \right) + w(N)^{h-1}\cdot (1/p) \stackrel{\eqref{eq:w_asymptotics}}{\ll} (1/p)^{h/(h-1)}.
    \end{align*}
    If $r = 1$, then we have that
    \begin{align*}
        \E\left[ \left| L^\pi_{\neq}(A) \cap [-dN, -dN+K_2(N)] \right| \right] \leq K_2(N)p \stackrel{\eqref{eq:K2}}{=} w(N) \cdot (1/p)^{1/(h-1)} \stackrel{\eqref{eq:w_asymptotics}}{\ll} (1/p)^{h/(h-1)}.
    \end{align*}
    If $r = 0$, this expectation is zero. An analogous argument for the right fringe now completes the proof.
\end{proof}

\subsection{Concentration} \label{subsec:concentration}

We now show that the random variable $|L(A)^c|$ is strongly concentrated about its expectation. Together with \cref{prop:slow_fringes_expectation}, this will prove \cref{thm:Z_linear_forms_iii}.

\begin{proposition} \label{prop:slow_concentration}
    If \eqref{eq:slow_decay_assumption} holds, then $|L(A)^c| \sim \E\left[ |L(A)^c| \right]$.
\end{proposition}

\begin{proof}
    We proceed exactly as in the proof of \cref{prop:critical_concentration}, but with the following modifications. First, unlike in the proof of \cref{prop:critical_concentration}, we will need to show that $\Delta_{n,1}(A)$ is modest with high probability. Towards this end, we introduce the shorthand
    \begin{align*}
        \kappa(n) := \max\left\{K_3(N), \min\{n,N-n\}\right\}.
    \end{align*}
    The expectation of $\Delta_{n,1}(A)$ satisfies, for a constant $C > 0$ independent of $n$, 
    \begin{align}
        & \E\left[ \Delta_{n,1}(A) \right] \\
        & \qquad = \E\left[\E\left[ \left| \left(L\left(A \cup \{n\}\right) \setminus L\left(A \setminus \{n\}\right) \right) \cap (\mathscr{I}_n \cap [-dN+K_3(N), sN-K_3(N)]) \right| \ \Big| \ \mathcal F_{n-1} \right]\right] \nonumber \\
        & \qquad = \E\left[ \left| \left(L\left(A \cup \{n\}\right) \setminus L\left(A \setminus \{n\}\right) \right) \cap (\mathscr{I}_n \cap [-dN+K_3(N), sN-K_3(N)]) \right| \right] \nonumber \\
        & \qquad = \sum_{k \in \mathscr{I}_n \cap [-dN+K_3(N), sN-K_3(N)]} \Prob\left[ \left( k \in L\left(A \cup \{n\}\right) \right) \land \left( k \notin L\left(A \setminus \{n\}\right)\right) \right] \nonumber \\
        & \qquad \leq \sum_{k=\kappa(n)}^{mN-\kappa(n)} \Prob\left[ -dN+k \notin L\left(A \setminus \{n\}\right) \right] \stackrel{\text{(\cref{rmk:forbidden_summand})}}{\leq} \sum_{k=\kappa(n)}^{mN/2} e^{-Ckp} + e^{-Ck^{h-1}p^h} \label{eq:restricted_sum_1} \\
        & \qquad \lesssim (1/p)\exp\left( -Cp\left(\kappa(n)-1\right) \right) + \left( 1/p \right)^{h/(h-1)} \int_{(\kappa(n)-1)p^{h/(h-1)}}^\infty e^{-Cx^{h-1}} \ dx \label{eq:restricted_sum_2} \\
        & \qquad \lesssim (1/p)\exp\left( -(C/2)p \cdot \kappa(n) \right) + \left( 1/p \right)^{h/(h-1)}\exp\left( - (C/2) \kappa(n)^{h-1} p^h \right). \label{eq:restricted_sum_3}
    \end{align}
    We mention that we may appeal to \cref{rmk:forbidden_summand} in \eqref{eq:restricted_sum_1} since $K_3(N) \gg (1/p)^{\frac{h}{h-1}}$ by definition. The constants implicit in \eqref{eq:restricted_sum_2} and \eqref{eq:restricted_sum_3} hold uniformly over $n$. The inequality \eqref{eq:restricted_sum_2} follows from computations entirely analogous to those performed in the proof of \cref{lem:missing_middle}. Therefore, we have that
    \begin{align}
        & \Prob\left( \Delta_{n,1}(A) \geq (1/p)\exp\left( -(C/4)p \cdot \kappa(n) \right) + \left( 1/p \right)^{h/(h-1)}\exp\left( - (C/4) \kappa(n)^{h-1} p^h \right) \right) \nonumber \\
        & \stackrel{\text{(Markov)}, \eqref{eq:restricted_sum_3}}{\lesssim} \frac{(1/p)\exp\left( -(C/2)p \cdot \kappa(n) \right) + \left( 1/p \right)^{h/(h-1)}\exp\left( - (C/2) \kappa(n)^{h-1} p^h \right)}{(1/p)\exp\left( -(C/4)p \cdot \kappa(n) \right) + \left( 1/p \right)^{h/(h-1)}\exp\left( - (C/4) \kappa(n)^{h-1} p^h \right)} \nonumber \\
        & \qquad \lesssim \exp\left( -(C/4)p \cdot \kappa(n) \right) + \exp\left( -(C/4)\kappa(n)^{h-1} p^h \right). \label{eq:markov_1}
    \end{align}
    By a union bound, it now follows from \eqref{eq:markov_1} that
    \begin{align}
        & \Prob\left( \Delta_{n,1}(A) \geq (1/p)\exp\left( -(C/4)p \cdot \kappa(n) \right) + \left( 1/p \right)^{h/(h-1)}\exp\left( - (C/4) \kappa(n)^{h-1} p^h \right) \text{ for some } n \right) \nonumber \\
        & \qquad \lesssim \sum_{n = 0}^{N} \exp\left( -(C/4)p \cdot \kappa(n) \right) + \exp\left( -(C/4)\kappa(n)^{h-1} p^h \right) \nonumber \\
        & \qquad = \sum_{n = 0}^{N} \exp\left( -(C/4)p \cdot \kappa(n) \right) + \sum_{n = 0}^{N} \exp\left( -(C/4)\kappa(n)^{h-1} p^h \right). \label{eq:tail_bd_two_sums}
    \end{align}
    We study each of the two sums in \eqref{eq:tail_bd_two_sums} separately. We may bound the first sum via
    \begin{align}
        & \sum_{n = 0}^{N} \exp\left( -(C/4)p \cdot \kappa(n) \right) \\
        & \quad \leq \sum_{n=0}^{K_3(N)} \exp\left( -\frac{CpK_3(N)}{4} \right) + \sum_{n=N-K_3(N)}^N \exp\left( -\frac{CpK_3(N)}{4} \right) + 2\sum_{n=K_3(N)+1}^{N/2} \exp\left( -\frac{Cpn}{2}\right) \nonumber \\
        & \quad \lesssim \frac{\log(1/p)}{p^{\frac{h}{h-1}}}\exp\left(-\frac{C \log(1/p)}{4p^{\frac{1}{h-1}}} \right) + \int_{K_3(N)}^\infty \exp\left( -\frac{Cpx}{2}\right) \ dx \nonumber \\
        & \quad \lesssim \left((1/p)^{\frac{1}{h-1}} \log(1/p) \right)^h \exp\left(-(C/4) (1/p)^{\frac{1}{h-1}}\log(1/p) \right) + (1/p) \exp\left(-(C/4)(1/p)^{\frac{1}{h-1}} \log(1/p) \right) \nonumber \\
        & \quad \ll 1. \label{eq:tail_bd_sum_1}
    \end{align}
    We may similarly bound the second sum via
    \begin{align}
        & \sum_{n = 0}^{N} \exp\left( -(C/4)\kappa(n)^{h-1} p^h \right) \nonumber \\
        & \quad \leq \sum_{n=0}^{K_3(N)} \exp\left( -\frac{Cp^hK_3(N)^{h-1}}{4} \right) + \sum_{n=N-K_3(N)}^N \exp\left( -\frac{Cp^hK_3(N)^{h-1}}{4} \right) \\
        & \qquad \qquad \qquad + 2\sum_{n=K_3(N)+1}^{N/2} \exp\left( -\frac{Cp^hn^{h-1}}{2}\right) \nonumber \\
        & \quad \lesssim \frac{\log(1/p)}{p^{\frac{h}{h-1}}}\exp\left(-\frac{C \left( \log(1/p) \right)^{h-1}}{4p^{\frac{1}{h-1}}} \right) + \int_{K_3(N)}^\infty \exp\left( -\frac{Cp^hx^{h-1}}{2}\right) \ dx \nonumber \\
        & \quad \lesssim o(1) + \left( 1/p \right)^{h/(h-1)} \int_{(K_3(N)-1)p^{h/(h-1)}}^\infty e^{-Cx^{h-1}/2} \ dx \nonumber \\
        & \quad \lesssim o(1) + \left( 1/p \right)^{h/(h-1)} \int_{\log(1/p)-p^{h/(h-1)}}^\infty e^{-Cx^{h-1}/2} \ dx \nonumber \\
        & \quad \lesssim o(1) + \left( 1/p \right)^{h/(h-1)} \int_{\log(1/p)/2}^\infty e^{-3hx} \ dx \lesssim o(1) + \left( 1/p \right)^{h/(h-1)}\exp\left( -\frac{3h\log(1/p)}{2} \right) \ll 1. \label{eq:tail_bd_sum_2}
    \end{align}
    Combining \eqref{eq:tail_bd_sum_1} and \eqref{eq:tail_bd_sum_2}, we conclude that with probability $1-o(1)$, it holds for all $n \in I_N$ that
    \begin{equation} \label{eq:delta_1_bound}
    	\begin{aligned} 
            & \Delta_{n,1}(A) \leq (1/p)\exp\left( -(C/4)p \cdot \kappa(n) \right) + \left( 1/p \right)^{h/(h-1)}\exp\left( - (C/4) \kappa(n)^{h-1} p^h \right) \\
            & \stackrel{\left( -\kappa(n) \leq K_3(N) \right)}{\leq} (1/p) \exp\left(-(C/4) (1/p)^{\frac{1}{h-1}}\log(1/p) \right) + \left( 1/p \right)^{h/(h-1)}\exp\left( - (C/4) \left( \log \left( 1/p \right) \right)^{h-1} \right) \\
            & \qquad \ll 1.
        \end{aligned}
    \end{equation}
    Furthermore, on this $1-o(1)$ probability event, it holds that
    \begin{align*}
        & V(A) \stackrel{\eqref{eq:V(A)_reformulation}}{\sim} p \sum_{n=0}^N \left( \Delta_{n,1}(A) \right)^2 \\
        & \stackrel{\eqref{eq:delta_1_bound}}{\leq} p \left[ (1/p)^2 \exp\left(-(C/2)(1/p)^{\frac{1}{h-1}} \log(1/p) \right) + \left( 1/p \right)^{2h/(h-1)}\exp\left( - (C/2) \left( \log \left( 1/p \right) \right)^{h-1} \right) \right] \\
        & = (1/p) \exp\left(-(C/2)(1/p)^{\frac{1}{h-1}} \log(1/p) \right) + \left( 1/p \right)^{2h/(h-1)-1}\exp\left( - (C/2) \left( \log \left( 1/p \right) \right)^{h-1} \right) \ll 1.
    \end{align*}
    Finally, if we assume \eqref{eq:slow_decay_assumption}, then \eqref{eq:root_lambda_V} and \cref{lem:missing_fringes} together imply the sufficiency of showing \eqref{eq:vanishing_RHS}. With these adjustments, tracing the proof of \cref{prop:critical_concentration} proves the desired.
\end{proof}

This completes the proof of \cref{thm:Z_linear_forms}.

\section{Local Poisson Convergence} \label{sec:poisson_convergence}

The methods we used to prove \cref{thm:Z_linear_forms} lend themselves quickly to a number of observations concerning the asymptotic behavior of the number of representations, which is captured in \cref{thm:poisson_convergence}. In the following computations, we assume that $k \in [0, mN/2]$. As remarked in \cref{sec:computations}, we may extend these results to $k \in [mN/2, mN]$ by appealing to \eqref{eq:left_right_process_equivalence}.

\subsection{Upper Bound}

We begin by proving an upper bound for $d_\TV\left(\mathcal L(W_k), \Pois(\mu_k)\right)$. In accordance with the condition of \cref{thm:poisson_convergence_i}, we assume throughout the proof of the upper bound that
\begin{align} \label{eq:poisson_conv_ub_assumption}
    p(N) \ll N^{-\frac{h-2}{h-1}} \iff N^{h-2}p^{h-1} \ll 1.
\end{align}
We let $K_4: \N \to \N$ be a function satisfying the conditions
\begin{align} \label{eq:K3_asymptotics}
    1/p \ll K_4(N) \ll (1/p)^{\frac{2h-1}{2h-3}}.
\end{align}
If $k \leq K_4(N)$, then it holds whenever $k > 0$ that
\begin{align*}
    d_\TV\left(\mathcal L(W_k), \Pois(\mu_k) \right) & \stackrel{\text{(\cref{thm:poisson_process_convergence})}}{\leq} \min\{1, \mu_k^{-1} \} \left(b_1(k) + b_2(k)\right) \stackrel{\eqref{eq:b_1(k)_bound}}{\lesssim} k^{h-1}p^{2h} + \sum_{\ell=h}^{2h-1} k^{\ell-2}p^\ell \\
    & \stackrel{\eqref{eq:K3_asymptotics}}{\ll} \sum_{\ell=h}^{2h-1} (1/p)^{\frac{(\ell-2)(2h-1)}{2h-3}}p^\ell = \sum_{\ell=h}^{2h-1} p^{\frac{2(2h-1)}{2h-3} - \frac{2\ell}{2h-3}} \lesssim 1.
\end{align*}
This asymptotic readily follows for $k = 0$ via bounding $b_1(0) + b_2(0)$ since $|\mathscr{E}_0| \leq 1$. We now consider the $k > K_4(N)$ regime. If we work in the exceptional setting of \cref{lem:L_expression_tuples_lower_bounds_i} in which $h = 2$, $L(x_1, x_2) = x_1 - x_2$, and $k = N$, then $W_k \equiv 0$ and $\mu_k = 0$ (as there are no injective $2$-tuples $(a_1, a_2) \in I_N^{\underline{2}}$ such that $a_1 - a_2 = 0$), so \cref{thm:poisson_convergence_i} immediately follows. Otherwise, it holds that
\begin{align*}
    d_\TV\left(\mathcal L(W_k), \Pois(\mu_k) \right) & \stackrel{\text{(\cref{thm:poisson_process_convergence})}}{\leq} \min\{1, \mu_k^{-1} \} \left(b_1(k) + b_2(k)\right) \stackrel{\eqref{eq:basic_computation_mean}, \eqref{eq:AGG_1_not_dN}}{\lesssim} \frac{k^{2h-3}p^{2h-1}}{k^{h-1}p^h} \\
    & = k^{h-2}p^{h-1} \leq N^{h-2}p^{h-1} \stackrel{\eqref{eq:poisson_conv_ub_assumption}}{\ll} 1.
\end{align*}
This proves \cref{thm:poisson_convergence_i}. 

\smallskip

\subsection{Lower Bound}

We now turn to proving a lower bound for this total variation distance, which will show that the regime of $p(N)$ for which we established the limiting Poisson behavior of $W_k$ for all $k \in \left[ 0, mN \right]$ is sharp. We assume $h \geq 3$, as the regime in \cref{thm:poisson_convergence_ii} cannot arise when $h=2$ under the standing assumptions on $p(N)$. In what follows, we assume that
\begin{align} \label{eq:local_critical_decay}
    p(N) \gtrsim N^{-\frac{h-2}{h-1}} \iff N^{h-2}p^{h-1} \gtrsim 1.
\end{align}
We fix a constant $C > 0$, and we study values $k \in [CN, mN/2]$, for which we will be able to obtain a meaningful lower bound on $d_\TV\left(\mathcal L(W_k), \Pois(\mu_k)\right)$ via \cref{thm:positively_related_lower_bound}. We begin by observing that the collection of random variables $\{X_\Lambda: \Lambda \in \mathscr{E}_k\}$ is positively related. Indeed, for $\Lambda, \Lambda' \in \mathscr{E}_k$, let
\begin{align*}
    Y_{\Lambda'\Lambda} := I_{\{S(\Lambda') \setminus S(\Lambda) \subseteq A\}}.
\end{align*}
In particular, if $\Lambda = \Lambda'$, then $Y_{\Lambda'\Lambda} \equiv 1$. It readily follows that
\begin{align*}
    & \mathcal L\left(Y_{\Lambda'\Lambda}: \Lambda' \in \mathscr{E}_k\right) = \mathcal L\left(X_{\Lambda'}: \Lambda' \in \mathscr{E}_k \ | \ X_\Lambda = 1 \right),
    & Y_{\Lambda'\Lambda} \geq X_{\Lambda'} \text{ for all } \Lambda' \in \mathscr{E}_k \setminus \{\Lambda\}.
\end{align*}
Thus, we may invoke \cref{thm:positively_related_lower_bound}. Towards this end, we define for each $k \in [CN,mN/2]$ the quantities
\begin{align}
    \gamma_k & := \frac{\E\left[(W_k-\E[W_k])^4 \right]}{\mu_k} - 1; \nonumber \\
    \epsilon_k & := \frac{\Var(W_k)}{\mu_k} - 1 \stackrel{\eqref{eq:basic_computation_mean}}{\lesssim} \frac{\Var(W_k)}{k^{h-1}p^h} \stackrel{\eqref{eq:basic_computation_2},\eqref{eq:local_critical_decay},\left( k \geq CN \right)}{\lesssim} \frac{N^{2h-3}p^{2h-1}}{k^{h-1}p^h} \stackrel{\left( k \geq CN \right)}{\lesssim} N^{h-2}p^{h-1} \stackrel{\eqref{eq:local_critical_decay}}{=} \Omega(1). \label{eq:epsilon_k}
\end{align}
It also holds that, with all asymptotic statements holding uniformly over $k \in [CN, mN/2]$,
\begin{align} \label{eq:epsilon_k_lower_bound}
    \epsilon_k = \frac{\Var(W_k)-\mu_k}{\mu_k} \stackrel{\eqref{eq:frugal_variance_bounds}}{\gtrsim} \frac{k^{2h-3}p^{2h-1} - p \mu_k}{k^{h-1}p^h} \gtrsim k^{h-2}p^{h-1} - p \stackrel{\left( k \in [ CN, mN/2] \right)}{\gtrsim} N^{h-2}p^{h-1},
\end{align}
from which we conclude that
\begin{align} \label{eq:epsilon_k_asymp}
    \epsilon_k \stackrel{\eqref{eq:epsilon_k}, \eqref{eq:epsilon_k_lower_bound}}{\asymp} N^{h-2}p^{h-1} = \Omega(1).
\end{align}
We now proceed with the relevant computations. We have that
\begin{align*}
    \left( \frac{\gamma_k}{\mu_k \epsilon_k} \right)_+ \stackrel{(\gamma_k+1 \geq 0)}{\leq} \frac{\gamma_k+1}{\mu_k \epsilon_k} = \frac{\E\left[(W_k-\E[W_k])^4 \right]}{\mu_k^2 \epsilon_k} \stackrel{\eqref{eq:basic_computation_4}}{\lesssim} \frac{N^{4h-6}p^{4h-2}}{N^{3h-4}p^{3h-1}} = N^{h-2}p^{h-1} \stackrel{\eqref{eq:epsilon_k_asymp}}{\asymp} \epsilon_k.
\end{align*}
We also have that
\begin{align*}
    \left( \frac{15}{2} + \frac{7}{\mu_k} \right) \left( 1 + \epsilon_k \right)^2 \frac{\max_{\Lambda \in \mathscr{E}_k} p_\Lambda}{\epsilon_k} & \lesssim p\left( 1 + \frac{1}{\mu_k} \right) \frac{1 + \epsilon_k + \epsilon_k^2}{\epsilon_k} \stackrel{\left( \epsilon_k = \Omega(1) \right)}{\lesssim} p\left( 1 + \frac{1}{\mu_k} \right) \epsilon_k \\
    & \stackrel{\eqref{eq:basic_computation_mean}, \eqref{eq:epsilon_k}}{\lesssim} \epsilon_k + \frac{p \cdot N^{h-2}p^{h-1}}{N^{h-1}p^h} = \epsilon_k + \frac{1}{N} \lesssim \epsilon_k.
\end{align*}
Therefore, it follows that
\begin{align*}
    \psi_k := \left( 1 + \frac{3}{2}\max_{\Lambda \in \mathscr{E}_k} p_\Lambda \right) \left(\frac{\gamma_k}{\mu_k \epsilon_k}\right)_+ + 3\epsilon_k + \left( \frac{15}{2} + \frac{7}{\mu_k} \right)\left( 1 + \epsilon_k \right)^2\frac{\max_{\Lambda \in \mathscr{E}_k} p_\Lambda}{\epsilon_k} \asymp \epsilon_k. 
\end{align*}
Altogether, we conclude that uniformly over $k \in [CN, mN/2]$,
\begin{align*}
    d_\TV\left( \mathcal L(W_k), \Pois(\mu_k) \right) \geq \frac{\epsilon_k}{11+3\psi_k} \geq \frac{\epsilon_k}{11+ O(1)\epsilon_k} = \Omega(1),
\end{align*}
which proves \cref{thm:poisson_convergence_ii}. 

\section{Future Directions} \label{sec:future_directions}

We conclude by recording some future directions suggested by the results and by the proof framework developed in this paper. Answers to these questions may help clarify which aspects of the present threshold landscape are robust and which are specific to this one-dimensional, linear, independent setting.

\subsection{Generalizations}

At a high level, our theorems treat degree-one transformations (integer linear forms) applied to a $p$-random subset of an interval, and they identify distinct threshold scales for 
\begin{enumerate}
    \item[(i)] \emph{global coverage} of the feasible range, and 
    \item[(ii)] \emph{local multiplicity} / Poisson approximation for representation counts.
\end{enumerate}
It seems natural to ask how much of this ``two-scale" threshold picture persists when one changes either the ambient setting or the notion of local limit.

\subsubsection{Higher-dimensional lattices.}

One natural generalization is to replace the one-dimensional interval $\{0,1,\dots,N\} \subset \mathbb{Z}$ by the $d$-dimensional box $\{0,1,\dots,N\}^d \subset \mathbb{Z}^d$, and more generally to allow the linear map to take values in a different lattice dimension.  Concretely, we fix integers $d, r \geq 1$ and $h \geq 2$, and take a linear map
\begin{align*}
    & T:(\mathbb{Z}^d)^h \to \mathbb{Z}^r;
    & T(x_1,\dots,x_h)=U_1x_1+\cdots+U_hx_h.
\end{align*}
where the coefficient matrices $U_i \in \mathbb{Z}^{r \times d}$ are fixed. If $A\subseteq \{0,1,\dots,N\}^d$ is formed by including each lattice point in this $d$-dimensional box independently with probability
$p=p(N)$, one can ask for threshold behavior of the image
\begin{align*}
    T(A) := \left\{ T(a_1,\dots,a_h) : a_i \in A \right\} \subseteq \mathbb{Z}^r,
\end{align*}
its complement inside the feasible set $T\big((\{0,1,\dots,N\}^d)^h\big)$, and the associated representation counts (i.e., the number of essentially distinct $h$-tuples mapping to a candidate value in $\mathbb{Z}^r$).

Heuristically, the relevant scales should depend on both $d$ and $r$. Since it holds that $ |A| \sim N^d p$ while the feasible region typically contains $\asymp N^r$ lattice points (when $T$ has rank $r$ and is not degenerate), one expects the onset of collisions in the subcritical, essentially injective regime when
\begin{align*}
    |A|^h \asymp N^r \qquad \iff \qquad (N^d p)^h \asymp N^r,
\end{align*}
suggesting a global threshold on the order of $p(N) \asymp N^{-d + r/h}$. Likewise, for a bulk value $t \in \Z^r$, the expected number of representations is heuristically $\asymp N^{dh-r}p^h$, and a second (local) threshold should be governed by when overlaps among representations cease to be negligible. A naive pair-overlap count (as was done to prove the results in \cref{subsec:asymptotic_enumeration}) suggests this occurs when $N^{(h-1)d-r}p^{\,h-1}\asymp 1$, i.e., around $p(N)\asymp N^{-d + r/(h-1)}$ (in the underdetermined range where this lies below a constant). Making these heuristics rigorous would require uniform lattice-point enumeration for the system $T(\mathbf{x})=t$ with $\mathbf{x}\in(\{0,1,\dots,N\}^d)^h$, together with concentration guarantees that remain effective near the boundary of $T\big((\{0,1,\dots,N\}^d)^h\big)$.

\subsubsection{Higher degree polynomials.}
Our work settles the degree-one regime of this problem, i.e., the setting of linear forms. A more ambitious extension is to replace a linear form $L$ by a higher-degree polynomial map $P$, e.g., $P(x_1,x_2)=x_1^2+x_2$ or (more generally) a fixed polynomial in $h$ variables with integer coefficients, and to ask for the size and coverage properties of the mapping
\begin{align*}
    P(A) := \left\{ P(a_1, \dots, a_h) : a_i \in A \right\}
\end{align*}
as well as local statistics for representation counts. In this setting, even the correct analogue of the ``feasible interval" may be more delicate, and uniform enumeration of solutions to $P(\mathbf{x})=t$ typically becomes a genuinely arithmetic problem.  A basic question is whether one should still expect sharp threshold phenomena (and possibly multiple scales) for coverage and for local multiplicity, or whether new behavior emerges.

\subsubsection{Other ambient groups.}

A related family of questions is obtained by changing the ambient group from $\Z$ to some other setting. For example, one could take a sequence of finite subgroups $G_N$ inside an abelian group $\left( G, + \right)$ or define a natural sequence of finite groups $\left( G_N, + \right)$, sample $A \subseteq G_N$ from a product measure, then study the image set
\begin{align*}
    L(A) = \left\{ L(a_1, \dots, a_h) : a_i\in A \right\} \subseteq G
\end{align*}
for a linear form $L: G^h \to G$. Even in we work in the cyclic group $G_N = \mathbb{Z}/N\mathbb{Z}$, which seems the most natural generalization of our work to a different setting, the absence of ``fringes" and the presence of wrap-around effects suggest that the complement behavior of $L(A)$ may have a different shape than in $\mathbb{Z}$. It would be interesting to understand which aspects of the present threshold picture are intrinsically one-dimensional/boundary-driven, and which are robust across ambient groups.

\subsection{Point Process Limits}

A closely related question is whether there is some sort of meaningful point process limit for the local representation landscape. Specifically, we define the dependent Bernoulli process $\Xi_k$ via, where $\delta_\Lambda$ denotes the unit point mass at $\Lambda$,\footnote{We may construct this dependent Bernoulli process by enumerating the $L$-expressions in $\mathscr{E}_k$ by $\Lambda_1, \dots, \Lambda_{\left|\mathscr{E}_k\right|}$ and constructing a map $\Xi_k: \big(\Omega_N, 2^{\Omega_N}, \Prob_N\big) \times ([0,1], \mathcal B_{[0,1]}) \to \left( \R, \mathcal B\right)$ by
\begin{align*}
    \Xi_k(\omega, B) := \sum_{i=1}^{\left| \mathscr{E}_k \right|} I_{\Lambda_i}(\omega) \delta_{i/| \mathscr{E}_k |}(B),
\end{align*}
where $I_{\Lambda_i}$ denotes the indicator for the event $S(\Lambda_i) \subseteq A$. The intensity $\pi$ of this point process is then a measure on $\mathcal B_{[0,1]}$ which is defined by, for $B \in \mathcal B_{[0,1]}$,
\begin{align*}
    \pi_k(B) = \sum_{i=1}^{\left| \mathscr{E}_k \right|} p_{\Lambda_i} \delta_{i/| \mathscr{E}_k |}(B).
\end{align*}
The corresponding Poisson point process with intensity $\pi$ is then defined accordingly on $\big(\Omega_N, 2^{\Omega_N}, \Prob_N\big) \times ([0,1], \mathcal B_{[0,1]})$.}
\begin{align*}
    \Xi_k := \sum_{\Lambda \in \mathscr{E}_k} I_\Lambda \delta_\Lambda
\end{align*}
denote a corresponding dependent Bernoulli process with intensity $\pi_k$, which assigns a weight of $p_\Lambda$ to $\Lambda$. Lifting to the point process level, our techniques yield the following result.
\begin{corollary} \label{cor:poisson_process}
    Assume the setup of \cref{thm:Z_linear_forms}. Let $C > 0$ be a constant. The following two situations arise.
    \begin{corparts}
        \item \label{cor:poisson_process_i}

        If $p(N) \ll N^{-\frac{h-1}{h}}$, then $d_\TV\left(\mathcal L(\Xi_k), \Pois(\pi_k) \right) \ll 1$ uniformly over $k \in \left[ 0, mN \right]$.

        \item \label{cor:poisson_process_ii}

        If $p(N) \gtrsim N^{-\frac{h-2}{h-1}}$, then $d_\TV\left(\mathcal L(\Xi_k), \Pois(\pi_k) \right) = \Omega(1)$ uniformly over $k \in \left[CN, (m-C)N\right]$.
    \end{corparts}
\end{corollary}

\begin{proof}
    \cref{cor:poisson_process_i} quickly follows from \cref{thm:poisson_process_convergence} since if we have that $p(N) \ll N^{-\frac{h-1}{h}}$, then
    \begin{align*}
        d_\TV\left(\mathcal L(\Xi_k), \Pois(\pi_k) \right) \leq b_1(k) + b_2(k) & \stackrel{\eqref{eq:b_1(k)_bound}}{\lesssim} N^{h-1}p^{2h} + \sum_{\ell=h}^{2h-1} N^{\ell-2}p^\ell \lesssim N^{2h-3}p^{2h-1} \ll 1
    \end{align*}
    if $k \in (0, mN]$. This asymptotic readily follows for $k = 0$ via bounding $b_1(0) + b_2(0)$ since $|\mathscr{E}_0| \leq 1$. On the other hand, \cref{cor:poisson_process_ii} follows since 
    \begin{align*}
        d_\TV\left(\mathcal L(\Xi_k), \Pois(\pi_k) \right) \geq d_\TV\left( \mathcal L(W_k), \Pois(\mu_k) \right),
    \end{align*}
    establishing \cref{cor:poisson_process_ii}.
\end{proof}

Concretely, \cref{cor:poisson_process} identifies a Poisson point process regime for the dependent Bernoulli process $\Xi_k$ when $p$ is sufficiently small, and it also shows that Poisson behavior fails macroscopically throughout the bulk once $p$ is sufficiently large. However, it does not address the intermediate range of densities between these behaviors, nor does it address the nature of the transition itself. A natural problem, therefore, is to determine whether $\Xi_k$ admits a nontrivial limiting description in this intermediate regime, and if so, to identify the limiting objects governing the representation landscape.

\section*{Acknowledgments}

This research was supported, in part, by NSF Grants DMS1947438 and DMS2241623. The first author thanks Professor Robin Pemantle for a very helpful conversation on the problem. Finally, we thank Nathan Tung for reading a draft of this article and for many helpful comments which improved the exposition.

\appendix

\section{Standard Asymptotic Notation} \label{sec:asymptotic_notation}

Let $X = X_N$ be a real-valued random variable depending on a positive integer parameter $N$, and let $E = E_N$ be an event depending on $N$. Let $f(N)$ and $g(N)$ be deterministic, positive real-valued functions. We employ the following standard notation.
\begin{itemize}
    \item We write $X \sim f(N)$ to mean that $X/f(N) \xrightarrow{p} 1$ as $N\to\infty$.

    \item We write $f(N) \sim g(N)$ if $\lim_{N\to\infty} f(N)/g(N)=1$.

    \item We write $f(N) \lesssim g(N)$ (respectively $f(N) \gtrsim g(N)$) if there exists a constant $C>0$ such that $f(N) \leq Cg(N)$ (respectively $f(N) \geq Cg(N)$) for all sufficiently large $N$. The statements $f(N)=O(g(N))$ and $f(N)=\Omega(g(N))$ have the same respective meanings. Finally, $f(N)=\Theta(g(N))$, or equivalently $f(N) \asymp g(N)$, means both $f(N)=O(g(N))$ and $g(N)=O(f(N))$.

    \item We write $f(N)=o(g(N))$ (equivalently, $f(N)\ll g(N)$) means $\lim_{N\to\infty} f(N)/g(N)=0$, and $f(N)=\omega(g(N))$ (equivalently, $f(N)\gg g(N)$) means $g(N)=o(f(N))$.

    \item We say that $E$ holds \emph{with high probability} if
    \begin{align*}
        \Prob(E) \to 1.
    \end{align*}
    as $N \to \infty$. Equivalently, $E$ holds with high probability if $E$ fails with probability $o(1)$.
\end{itemize}

\section{Proofs for \texorpdfstring{\cref{sec:preliminaries}}{Section \ref{sec:preliminaries}}} \label{sec:preliminaries_proofs}

We work towards a proof of \cref{lem:number_of_subsets_asymptotics}. We begin with the following standard definitions.

\begin{definition}
    A \textit{partition with $h$ parts} of a positive integer $k$ is a finite nonincreasing sequence of $h$ positive integers $\lambda_1 \geq \cdots \geq \lambda_h \geq 1$ such that $\sum_{i=1}^h \lambda_i = k$. A \textit{weak composition with $h$ parts} is an $h$-tuple $(\lambda_1, \dots, \lambda_h)$ of nonnegative integers such that $\sum_{i=1}^h \lambda_i = k$.
\end{definition}

\begin{definition}
    A sequence $a_0, a_1, \dots, a_n$ of real numbers is \textit{unimodal} if there exists an index $0 \leq j \leq n$ for which it holds that
    \begin{align*}
        a_0 \leq a_1 \leq \cdots \leq a_j \geq a_{j+1} \geq \cdots \geq a_n.
    \end{align*}
    The sequence is \textit{symmetric} if $a_i = a_{n-i}$ for all $0 \leq i \leq n$.
\end{definition}
For later use, we observe that a sequence $a_0, a_1, \dots, a_n$ that is both unimodal and symmetric must have that $a_{\lfloor n/2 \rfloor} = a_{\lceil n/2 \rceil}$, and that this value must be the maximum of the sequence. 

We use the following standard notation. For a formal power series 
\begin{align*}
    F(q) = \sum_{m\ge 0} a_m q^m,
\end{align*}
we write $[q^m]F(q) := a_m$ for the coefficient of $q^m$. For integers $n \geq 0$ and $0 \leq r \leq n$, we define the Gaussian binomial coefficient $\binom{n}{r}_q$ by
\begin{align*}
    \binom{n}{r}_q := \prod_{i=1}^{r}\frac{1-q^{n-r+i}}{1-q^{i}},
\end{align*}
which is well-known to be a polynomial in $q$ with nonnegative integer coefficients.

We now record some auxiliary results that will be used in the proofs to follow.

\begin{theorem}[{\cite[Theorem 2.4]{stanley2016some}}] \label{thm:stanley_zanello}
    Fix a real number $\alpha \geq 0$ and let $h$ be a positive integer. Then
    \begin{align*}
        \left[ q^{\lfloor \alpha k \rfloor}\right]\binom{k+h}{h}_q = \frac{C(\alpha, h)k^{h-1}}{(h-1)!h!} + O(k^{h-2}),
    \end{align*}
    for $k \to \infty$, where $C(\alpha, h)$ is the Euler-Frobenius number
    \begin{align} \label{eq:euler_frob_num}
        C(\alpha, h) := \sum_{i=0}^{\lfloor \alpha \rfloor} (-1)^i \binom{h}{i} (\alpha - i)^{h-i}.
    \end{align}
\end{theorem}

\begin{theorem}[{\cite[Theorem 3.2, Remarks 3.3 and 3.4]{janson2013euler}}] \label{thm:janson_euler_frobenius}
    For a real number $\alpha \geq 0$ and a positive integer $h$, it holds that
    \begin{align*}
        \IH_h(\alpha) = \frac{C(\alpha, h)}{(h-1)!}.
    \end{align*}
    Furthermore, for $h \geq 2$, it holds that
    \begin{align} \label{eq:IH_spline_identity}
        \IH_h(\alpha) = \frac{1}{h-1}\left( \alpha \IH_{h-1}(\alpha) + (h-\alpha)\IH_{h-1}(\alpha-1) \right).
    \end{align}
\end{theorem}

We are now ready to assemble the proof of \cref{lem:number_of_subsets_asymptotics}, beginning with a few intermediate lemmas.

\begin{lemma} \label{lem:partition_asymptotics}
    Fix a function $K: \N \to \N$ satisfying $1 \ll K(N) \ll N$. For $k \in I_{hN}$, let $p_{h,N}(k)$ denote the number of partitions of $k$ into at most $h$ parts, each at most $N$. Uniformly over $k \in [K(N), hN - K(N)]$,
    \begin{align} \label{eq:partition_asymptotics_desired}
        p_{h,N}(k) \sim \frac{\IH_h\left(k/N\right)}{h!}N^{h-1}.
    \end{align}
\end{lemma}

\begin{proof}
    We handle the cases $k \notin [N, (h-1)N ]$ and $k \in [N, (h-1)N ]$ separately. It is well known that the sequence
    \begin{align} \label{eq:partition_asymptotics_sequence}
        p_{h,N}(0), \ p_{h,N}(1), \ \dots, \ p_{h,N}(hN-1), \ p_{h,N}(hN)
    \end{align}
    has the Gaussian binomial coefficient $\binom{N+h}{h}_q$ as its generating function (e.g., see \cite[Theorem 3.1]{andrews1998theory}), i.e., it holds for all $k \in I_{hN}$ that
    \begin{align} \label{eq:q_binom_coeff}
        p_{h,N}(k) = [q^k]\binom{N+h}{h}_q.
    \end{align}
    Furthermore, this sequence is both unimodal and symmetric (e.g., see \cite{melczer2020counting, sylvester1878xxv}). For an integer $k \in [K(N), N]$ and a nonnegative real number $\alpha$, we have that
    \begin{align} 
        & p_{h,k}(k) \stackrel{\eqref{eq:q_binom_coeff}}{=} [ q^k ] \binom{k+h}{h}_q \stackrel{\text{(\cref{thm:stanley_zanello})}}{\sim} \frac{k^{h-1}}{(h-1)!h!} \stackrel{\eqref{eq:euler_frob_num},\text{(\cref{thm:janson_euler_frobenius})}}{=} \frac{\IH_h(k/N)}{h!}N^{h-1}; \label{eq:q_binomial_coeff_asymptotics_1} \\
        & \left[ q^{\lfloor \alpha N \rfloor} \right]\binom{N+h}{h}_q \stackrel{\text{(\cref{thm:stanley_zanello,thm:janson_euler_frobenius})}}{\sim} \frac{\IH_h\left( \alpha \right)}{h!}N^{h-1}, \label{eq:q_binomial_coeff_asymptotics_2}
    \end{align}
    where we note that \eqref{eq:q_binomial_coeff_asymptotics_1} holds uniformly over all such $k \in [K(N), N]$. Since any partition of an integer $k \in [K(N), N]$ certainly has parts, each at most $k$, it follows for such $k$ that $p_{h,N}(k)$ is the number of partitions of $k$ into $h$ parts. Since $k \leq N$, this is in turn the number of partitions of $k$ into $h$ parts, each at most $N$, i.e., it holds that $p_{h,k}(k) = p_{h,N}(k)$. Thus, \eqref{eq:q_binomial_coeff_asymptotics_1} implies that \eqref{eq:partition_asymptotics_desired} holds uniformly over all $k \in [K(N), N]$. The symmetry of the sequence \eqref{eq:partition_asymptotics_sequence} and the Irwin-Hall density $\IH_h(x)$ about $h/2$ then imply that \eqref{eq:partition_asymptotics_desired} holds uniformly over all 
    \begin{align} \label{eq:k_fringe_interval}
        k \in (N, (h-1)N)^c \cap [K(N), hN - K(N)].
    \end{align}
    We now turn to values $k \in [N, (h-1)N]$. We observe that $\IH(x)$ is uniformly continuous on $x \in [1,h-1]$ and positive on $x \in (0,h)$; this positivity can be deduced via induction on $h$ and using \eqref{eq:IH_spline_identity}. It thus follows that $\IH_h(k/N)$ has a fixed positive minimum on $k \in [N, (h-1)N]$. From these observations and \eqref{eq:q_binomial_coeff_asymptotics_2}, a standard continuity argument now yields that
    \begin{align} \label{eq:q_binomial_coeff_asymptotics_middle}
        p_{h,N}(k) \stackrel{\eqref{eq:q_binom_coeff}}{=} [ q^k ]\binom{N+h}{h}_q \sim \frac{\IH_h\left( k/N \right)}{h!}N^{h-1} \quad \text{ uniformly over } k \in \left[N, (h-1)N\right].
    \end{align}
    Combining \eqref{eq:q_binomial_coeff_asymptotics_middle} with the uniform validity of \eqref{eq:partition_asymptotics_desired} over values \eqref{eq:k_fringe_interval} proves the lemma.
\end{proof}

\begin{remark} \label{rmk:partition_asymptotics_small_k}
    For $k \leq N$, the expression in \cref{lem:partition_asymptotics} can be written using \eqref{eq:irwin_hall_density} as
    \begin{align*}
        \frac{\IH_h\left(k/N\right)}{h!}N^{h-1} = \frac{\left(k/N\right)^{h-1}}{(h-1)!h!} N^{h-1} = \frac{k^{h-1}}{(h-1)!h!}.
    \end{align*}
    This is consistent with classical results in analytic number theory regarding the number of partitions of a positive integer with a bounded number of parts (e.g., see \cite[Theorem 4.1]{erdos1941distribution}).
\end{remark}

\begin{lemma}\label{lem:ordered_tuples_asymptotics}
    Fix a function $K: \N \to \N$ satisfying $1 \ll K(N) \ll N$, jointly coprime nonzero integers $u_1, \dots, u_h$, and integers $b_1, \dots, b_h$. Let 
    \begin{align*}
        & u := (u_1, \dots, u_h);
        & b := (b_1, \dots, b_h).
    \end{align*}
    For $k \in I_{hN}$, let $c_{h,N}^{u,b}(k)$ denote the number of weak compositions $(a_1, \dots, a_h)$ of $k$ with $h$ parts, each of which is at most $N$, that satisfy
    \begin{align} \label{eq:ordered_tuples_conditions}
        a_i \equiv b_i \pmod{|u_i|} \quad \text{for all } i \in [h].
    \end{align}
    Uniformly over $k \in [K(N), hN-K(N)]$, it holds that
    \begin{align} \label{eq:num_ordered_tuples_asymptotics}
        c_{h,N}^{u,b}(k) \sim \frac{\IH_h(k/N)}{\prod_{i=1}^h |u_i|}N^{h-1}. 
    \end{align}
    Additionally, uniformly over $k_1 < K(N)$ and $k_2 > hN-K(N)$, it holds that
    \begin{align} \label{eq:ordered_tuples_fringes}
        & \IH_h(k_1/N) \ll \IH_h(1+k_1/N);
        & \IH_h(k_2/N) \ll \IH_h(1-k_2/N).
    \end{align}
\end{lemma}

\begin{proof}
    Let $c_{h,N}(k)$ denote the number of weak compositions $(a_1, \dots, a_h)$ of $k$ with $h$ parts, each of which is at most $N$. It is known (e.g., see \cite[Theorem 3]{zhong2022combinatorial}) that the sequence
    \begin{align} \label{eq:ordered_tuples_sequence}
        c_{h,N}(0), \ c_{h,N}(1), \ \dots, \ c_{h,N}(hN)
    \end{align}
    is unimodal and symmetric. It is readily observed that $O(k^{h-2})$ partitions of $k$ into at most $h$ parts, each at most $N$, are such that its parts are not pairwise distinct. Additionally, there are $O(k^{h-2})$ weak compositions of $k$ into $h$ parts with a zero part. Thus, \cref{lem:partition_asymptotics} implies that uniformly over $k \in [K(N), hN/2]$,
    \begin{align} \label{eq:unconstrained_compositions_asymptotics}
        c_{h,N}(k) \sim \IH_h(k/N) N^{h-1} = \Omega( k^{h-1} ).
    \end{align}
    In particular, for $k \leq N$ the lower bound \eqref{eq:unconstrained_compositions_asymptotics} follows from \cref{rmk:partition_asymptotics_small_k}. For $k \in [N,hN/2]$, we have $k/N \in [1,h/2]$. Since $\IH_h(x)$ is continuous and strictly positive on $x \in [1, h/2]$, it attains a minimum $c_* > 0$ on this interval. From here, it holds uniformly for $k \in [N,hN/2]$ that
    \begin{align*}
        c_{h,N}(k) \geq p_{h,N}(k) \gtrsim c_* \cdot N^{h-1} = \Omega(k^{h-1}).
    \end{align*}
    Since $c_{h,N}(k)$ and $\IH_h(x) = \IH_h(h-x)$, we extend \eqref{eq:unconstrained_compositions_asymptotics} uniformly to $k \in [K(N), hN-K(N)]$.

    In what follows, we deviate from our usual practice of assuming a fixed $h \geq 2$ and we prove \eqref{eq:num_ordered_tuples_asymptotics} via induction on $h \geq 2$. We begin with the induction basis $h=2$. For all $k \in [K(N), 2N-K(N)]$, the values of $a_1$ admitting a (unique) weak composition $(a_1, a_2)$ of $k$ with parts at most $N$ comprise an interval with at least $K(N) \gg 1$ elements. The two congruences in \eqref{eq:ordered_tuples_conditions} become
    \begin{align*}
        & a_1 \equiv b_1 \pmod{|u_1|};
        & k-a_1 \equiv b_2 \pmod{|u_2|} \iff a_1 \equiv k-b_2 \pmod{|u_2|}.
    \end{align*}
    Since $\gcd(u_1,u_2) = 1$, the Chinese Remainder Theorem implies that $\sim 1/|u_1u_2|$ such weak compositions $(a_1, a_2)$ satisfy \eqref{eq:ordered_tuples_conditions} uniformly over such $k$. Combined with \eqref{eq:unconstrained_compositions_asymptotics}, this establishes \eqref{eq:num_ordered_tuples_asymptotics} for $h = 2$.

    We proceed to the inductive step by considering $h \geq 3$. Let $\Tilde{K}: \N \to \N$ be a function satisfying 
    \begin{align*}
        1 \ll \Tilde{K}(N) \ll K(N).
    \end{align*}
    We can reformulate the condition that $(a_1, \dots, a_h)$ is a weak composition of $k$ with $h$ parts via
    \begin{align} \label{eq:ordered_tuples_induction}
        \sum_{i=1}^h a_i = k \iff \sum_{i=2}^h a_i = k - a_1.
    \end{align}
    Fix $k \in [K(N), hN/2]$. The values $a_1$ admitting $a_2, \dots, a_h \in I_N$ for which \eqref{eq:ordered_tuples_induction} holds form an interval $J_k \subseteq I_N$ of length at least $K(N)$. Let $\mathscr{J}_k$ be the subinterval obtained by compressing both endpoints of $J_k$ by $\Tilde{K}(N)$. Then $|\mathscr{J}_k| \gg 1$ and, for all $a_1 \in \mathscr{J}_k$,
    \begin{align*}
        k-a_1 \in [\Tilde{K}(N), hN/2 - \Tilde{K}(N)].
    \end{align*}
    Every $|u_1|$\textsuperscript{th} term $a_1 \in \mathscr{J}_k$ satisfies \eqref{eq:ordered_tuples_conditions} for $i=1$. Each such $a_1$ admits $c_{h-1,N}(k-a_1)$ weak compositions of $k$ with $h$ parts, each at most $N$. From \eqref{eq:ordered_tuples_sequence}, we observe the unimodality and symmetry of the sequence
    \begin{align} \label{eq:compositions_induction_sequence}
        c_{h-1,N}(\Tilde{K}(N)), \dots, c_{h-1,N}(hN-\Tilde{K}(N)),
    \end{align}
    and the asymptotics of the values \eqref{eq:compositions_induction_sequence} are uniformly described by \eqref{eq:unconstrained_compositions_asymptotics} with $\Tilde{K}(N)$ playing the role of $K(N)$. Altogether, we thus deduce that $\sim 1/|u_1|$ of the weak compositions $(a_1, \dots, a_h)$ satisfying \eqref{eq:ordered_tuples_induction} and with $a_1 \in \mathscr{J}_k$ also satisfy \eqref{eq:ordered_tuples_conditions} for $a_1$. Furthermore, the induction hypothesis and \eqref{eq:unconstrained_compositions_asymptotics} together imply that uniformly over all $a_1 \in \mathscr{J}_k$, it holds that $\sim 1/\prod_{i=2}^h|u_i|$ of the weak compositions $(a_2, \dots, a_h)$ of $k-a_1$ with $h-1$ parts, each at most $N$, satisfy \eqref{eq:ordered_tuples_conditions}. We conclude that uniformly over $k \in [K(N), hN/2]$,
    \begin{align} \label{eq:compositions_desired_1}
        c_{h,N}^{u,b}(k) \sim \frac{\IH_h(k/N)}{\prod_{i=1}^h |u_i|} N^{h-1} = \Omega(k^{h-1}).
    \end{align}
    In particular, since $\Tilde{K}(N) \ll K(N) \leq k$, it follows that the number of weak compositions $(a_1, \dots, a_h)$ of $k$ with $h$ parts, each at most $N$, for which $a_1 \notin \mathscr{J}_k$ is $o(k^{h-1})$.
    
    To extend \eqref{eq:compositions_desired_1} to values $k \in [hN/2, hN-K(N)]$, we define $b'$ via $b_i' \equiv N-b_i \pmod{|u_i|}$ and we set $k' := hN-k \in [K(N),hN/2]$. The involution 
    \begin{align*}
        (a_1,\dots,a_h)\mapsto (N-a_1,\dots,N-a_h)
    \end{align*}
    readily implies that
    \begin{align*}
        c^{u,b}_{h,N}(k)=c^{u,b'}_{h,N}(k').
    \end{align*}
    Noting that the asymptotic equivalence \eqref{eq:compositions_desired_1} is independent of the choice of $b$, applying \eqref{eq:compositions_desired_1} at $k'$ together with the symmetry $\IH_h(x) = \IH_h(h-x)$ yields the claim for all $k \in [K(N), hN-K(N)]$.

    On the other hand, \eqref{eq:ordered_tuples_fringes} can be seen to hold from an elementary continuity argument.
\end{proof}

We are now ready to prove \cref{lem:number_of_subsets_asymptotics}.

\begin{proof}[Proof of \cref{lem:number_of_subsets_asymptotics}]
    We fix $k \in [0, mN/2)$. The fact that $|\mathscr{E}_k| = |\mathscr{E}_{mN-k}|$ follows quickly from
    \begin{align} \label{eq:number_of_subsets_left_right}
        L\left(a_1, \dots, a_h\right) = -dN+k  \iff L\left(N-a_1,\dots,N-a_h\right) = -dN+(mN-k) = sN-k,
    \end{align}
    which naturally lends itself to a bijection between $\mathscr{E}_k$ and $\mathscr{E}_{mN-k}$. We now prove \eqref{eq:number_of_subsets_fringe_asymptotics} and \eqref{eq:number_of_subsets_asymptotics} for $|\mathscr{E}_k|$, after which we are done. Rearranging the LHS of \eqref{eq:number_of_subsets_left_right} gives the equivalent equation
    \begin{align} \label{eq:number_of_subsets_1}
       \sum_{i=1}^{h_{\text{pos}}} u_ia_i + \sum_{i=h_{\text{pos}}+1}^h |u_i| \left(N-a_i \right) = k.
    \end{align}
    For values $a_1, \dots, a_h \in I_N$, \eqref{eq:number_of_subsets_1} is equivalent to 
    the statement that
    \begin{align} \label{eq:desired_rearranged}
        \left(u_1a_1, \dots, u_{h_{\text{pos}}}a_{h_{\text{pos}}}, |u_{h_{\text{pos}}+1}|(N-a_{h_{\text{pos}}+1}), \dots, |u_h|(N-a_h)\right)
    \end{align}
    forms a weak composition of $k$. For any $k \in [0, K(N)]$, there are at most 
    \begin{align*}
        K(N)^{h-1} \ll N^{h-1}
    \end{align*}
    weak compositions with $h$ parts of $k$, so \eqref{eq:number_of_subsets_fringe_asymptotics} follows. We now fix $k \in [K(N), mN/2)$, and we prove \eqref{eq:number_of_subsets_asymptotics} by counting the number of $h$-tuples $(a_1, \dots, a_h)$ in $I_N^h$ for which 
    \eqref{eq:desired_rearranged} forms a weak composition of $k$. We then argue that for essentially all such $h$-tuples, all $h$ parts are distinct. Dividing our initial count by $|\Aut(L)|$ to correct for overcounting (resulting from counting $h$-tuples rather than sets) then yields an asymptotic expression for $|\mathscr{E}_k|$. All of this will be done via an argument which applies uniformly over values $k \in [K(N), mN/2)$.
    
    We begin with the initial count. We partition the collection of $h$-tuples $(a_1, \dots, a_h) \in I_N^h$ for which \eqref{eq:desired_rearranged} is a weak composition of $k$ into classes based on the realization of the $h$-tuple
    \begin{align} \label{eq:class_inclusion}
        \left(\lfloor u_1a_1/N \rfloor, \dots, \lfloor u_{h_{\text{pos}}}a_{h_{\text{pos}}}/N \rfloor, \lfloor |u_{h_{\text{pos}}+1}|(N-a_{h_{\text{pos}}+1})/N \rfloor, \dots, \lfloor |u_h|(N-a_h)/N \rfloor \right).
    \end{align}
    We set $t_i \in I_{|u_i|-1}$ for $i \in [h]$, so that those $h$-tuples $(a_1, \dots, a_h)$ in the class where the $h$-tuple \eqref{eq:class_inclusion} is equal to $(t_1, \dots, t_h)$ are exactly those which satisfy the inequalities
    \begin{align} \label{eq:class_ineqs}
        & 0 \leq u_ia_i - t_iN \leq N-1 \text{ for } i \in [h_{\text{pos}}],
        & 0 \leq |u_i|(N-a_i) - t_iN \leq N-1 \text{ for } i > h_{\text{pos}}.
    \end{align}
    This suggests, for all such $h$-tuples $(a_1, \dots, a_h)$, the equivalent reformulation of \eqref{eq:number_of_subsets_1} via
    \begin{align} \label{eq:reformulation}
         \sum_{i=1}^{h_{\text{pos}}} \left( u_ia_i - t_iN \right) + \sum_{i=h_{\text{pos}}+1}^h \left( |u_i| \left(N-a_i \right) - t_iN \right) = k - \left(\sum_{i=1}^h t_i\right)N.
    \end{align}
    From \eqref{eq:class_ineqs} and \eqref{eq:reformulation}, we deduce that the number of $h$-tuples $(a_1, \dots, a_h)$ in the class where the $h$-tuple \eqref{eq:class_inclusion} is equal to $(t_1, \dots, t_h)$ is exactly the number of weak compositions of $k-\left(\sum_{i=1}^h t_i\right)N$ with $h$ parts, each at most $N-1$, and satisfying \eqref{eq:ordered_tuples_conditions} for the vectors
    \begin{align*}
        & u = ( u_1, \dots, u_{h_{\text{pos}}}, |u_{h_{\text{pos}}+1}|, \dots, |u_h| );
        & b = \left( -t_1N, \dots, -t_hN \right).
    \end{align*}
    Invoking \cref{lem:ordered_tuples_asymptotics}, whenever 
    \begin{align} \label{eq:weak_composition_class}
        k/N - \sum_{i=1}^h t_i \in [K(N), hN-K(N)],
    \end{align}
    it follows that the number of such weak compositions is
    \begin{align*}
        \sim \frac{\IH_h\left(k/N - \sum_{i=1}^h t_i\right)}{\prod_{i=1}^h |u_i|} N^{h-1}.
    \end{align*}
    Summing over all choices of $t_i \in I_{|u_i|-1}$ for $i \in [h]$ yields that the number of $h$-tuples $(a_1, \dots, a_h) \in I_N^h$ for which \eqref{eq:desired_rearranged} is a weak composition of $k$ is
    \begin{align} \label{eq:ordered_seq_count}
        \sim \frac{\sum_{t_1=0}^{|u_1|-1} \cdots \sum_{t_h=0}^{|u_h|-1} \IH_h\left(k/N - \sum_{i=1}^h t_i\right)}{\prod_{i=1}^h |u_i|} N^{h-1} \stackrel{\text{(\cref{rmk:partition_asymptotics_small_k})}}{=} \Omega( k^{h-1} ).
    \end{align}
    In particular, the number of $h$-tuples $(a_1, \dots, a_h)$ for which $a_i \in \{0, N\}$ for some $i \in [h]$ is $O(k^{h-2})$, and we have used \eqref{eq:ordered_tuples_fringes} to control those summands in \eqref{eq:ordered_seq_count} for which \eqref{eq:weak_composition_class} fails. Furthermore, the number of these $h$-tuples $(a_1, \dots, a_h)$ for which the $h$ parts are not pairwise distinct is also observed from \eqref{eq:number_of_subsets_left_right} to be $O(k^{h-2})$. Indeed, each of $O(1)$ ``collision patterns" collapses $L$ to a linear form $L'$ in $r \leq h-1$ variables, contributing $O(k^{r-1}) = O(k^{h-2})$ solutions uniformly in $k$.\footnote{While $L'$ may have some zero coefficients, the only genuinely problematic collapse is when $L' \equiv 0$, which can yield $\Omega(N)$ solutions at the midpoint level $k=mN/2$ when $h=2$. This corresponds precisely to the exceptional setting of \cref{lem:L_expression_tuples_lower_bounds_i}. This is the reason for excluding $k=mN/2$ here.} Combining this observation with \eqref{eq:ordered_seq_count} and dividing the asymptotic expression in \eqref{eq:ordered_seq_count} by $|\Aut(L)|$, we conclude that \eqref{eq:number_of_subsets_asymptotics} holds uniformly over $k \in [K(N), mN/2)$.
\end{proof}

We now prove \cref{lem:L_expression_tuples_lower_bounds,lem:L_expressions_tuples_upper_bounds}.

\begin{proof}[Proof of \cref{lem:L_expression_tuples_lower_bounds}]
    Since $\gcd(u_1, \dots, u_h) = 1$, there exists a solution to the Diophantine equation 
    \begin{align} \label{eq:L_expression_num_elmts}
        L(a_1, \dots, a_h) = -dN+k.
    \end{align}
    For $k$ large enough, \eqref{eq:L_expression_num_elmts} also has solutions with $(a_1, \dots, a_h) \in I_N^h$, observable by adding appropriate multiples of $\lcm(u_1, \dots, u_h)$ to the entries of a particular solution $(a_1, \dots, a_h)$ of \eqref{eq:L_expression_num_elmts} to form new solutions to this Diophantine equation. It follows this way that uniformly over all large enough $k \leq mN/2$, we may form $\Omega(k^{h-1})$ solutions to \eqref{eq:L_expression_num_elmts} such that $a_1, \dots, a_h \in I_N$. (Here, $k$ must be large enough so that there indeed exists a solution to \eqref{eq:L_expression_num_elmts} which uses just elements in $I_N$.) Unless we are working in the exceptional setting in which $h = 2$, $L(x_1, x_2) = x_1 - x_2$, and $k = N$, there are $O(k^{h-2})$ solutions with a repeated coordinate. Indeed, as in the proof of \cref{lem:number_of_subsets_asymptotics}, each collision pattern collapses to a form in fewer than $h$ variables. Altogether, we deduce that there remain $\Omega(k^{h-1})$ injective solutions $(a_1,\dots,a_h)\in I_N^{\underline h}$. Since there are $O(1)$ such solutions in each orbit of $\mathscr{E}_k$, this proves \cref{lem:L_expression_tuples_lower_bounds_i}.

    We now turn to proving \cref{lem:L_expression_tuples_lower_bounds_ii}. We derive this lower bound for the number of ordered pairs
    \begin{align} \label{eq:ordered_tuples}
        \left( (c, a_2, \dots, a_h), (c, b_2, \dots, b_h) \right) \in (I_N^{\underline h})^2
    \end{align}
    that satisfy the conditions 
    \begin{align} \label{eq:lower_bound_tuple_constraints}
        & |\{c, a_2, \dots, a_h, b_2, \dots, b_h\}| = 2h-1;
        & L(c, a_2, \dots, a_h) = L(c, b_2, \dots, b_h) = -dN + k.
    \end{align}
    This suffices since passing these ordered pairs of $h$-tuples to ordered pairs of $\mathscr R$-orbits loses only a constant factor. We first consider all $h$-tuples \eqref{eq:ordered_tuples} satisfying the latter condition of \eqref{eq:lower_bound_tuple_constraints}. Arguing constructively as in the proof of \cref{lem:L_expression_tuples_lower_bounds_i}, we have for large $k$ that there are $\Omega(k)$ choices for the shared value $c$ such that $|u_1|c \leq k/2$ and such that the resulting Diophantine equations
    \begin{align*}
        \sum_{i=2}^h u_ia_i = \sum_{j=2}^h u_jb_j = -dN + k - u_1c
    \end{align*}
    have solutions and do not reduce to the exceptional setting of \cref{lem:L_expression_tuples_lower_bounds_i}. (Indeed, in the $h = 3$ setting, at most one such choice of the shared value $c$ may reduce these equations to the exceptional setting of \cref{lem:L_expression_tuples_lower_bounds_i}.) Again arguing as in the proof of \cref{lem:L_expression_tuples_lower_bounds_i} (here applied to the induced $(h-1)$-variable linear form $\sum_{i=2}^h u_i x_i$), we may generate 
    \begin{align*}
        \underbrace{\Omega(k)}_{\text{choice of shared value } c} \cdot \underbrace{\Omega(k^{h-2})^2}_{\text{choice of } a_2, \dots, a_h, b_2, \dots, b_h} = \Omega(k^{2h-3})
    \end{align*} 
    such ordered pairs uniformly over large $k \leq mN/2$. We now enforce the former condition of \eqref{eq:lower_bound_tuple_constraints}. Choosing one of $O(1)$ further constraints of the form $a_i = b_j$ for $i,j \neq 1$ and bounding the number of pairs \eqref{eq:ordered_tuples} satisfying this further constraint, we observe that the number of these ordered pairs which do not satisfy the former condition of \eqref{eq:lower_bound_tuple_constraints} is 
    \begin{align*}
        \underbrace{O(k)^2}_{\text{choice of shared values }c, \ a_i = b_j} \cdot \underbrace{O(k^{h-3})^2}_{\text{choice of remaining values}} = O(k^{2h-4}).
    \end{align*}
    The desired lower bound thus follows.
\end{proof}

\begin{proof}[Proof of \cref{lem:L_expressions_tuples_upper_bounds}]
    It suffices to derive these asymptotic upper bounds for the number of matrices
    \begin{align}
        \mathcal A = \left(a_{i,j} \right)_{i \in [t], j \in [h]} \in I_N^{t \times h},
    \end{align}
    where the matrices $\mathcal A$ that we consider are those such that 
    \begin{enumerate}
        \item each row is an injective $h$-tuple;

        \item there are exactly $\ell$ distinct values amongst the entries of $\mathcal A$;

        \item there do not exist distinct $i, i' \in [t]$ and $\sigma \in \mathscr{R}_k$ such that
        \begin{align} \label{eq:L_expressions_tuples_asymptotics_condition_1}
            (a_{i, 1}, \dots, a_{i, h}) = (a_{i', \sigma(1)}, \dots, a_{i', \sigma(h)});
        \end{align}

        \item if $t \geq 2$, then it holds for all $i \in [t]$ that
        \begin{align} \label{eq:L_expressions_tuples_asymptotics_condition_2}
            \{a_{i,1}, \dots, a_{i,h}\} \cap \left( \bigcup_{i' \neq i} \{a_{i',1}, \dots, a_{i',h}\} \right) \neq \emptyset \text{ for all } i \in [t];
        \end{align}

        \item it holds for all $i \in [t]$ that
        \begin{align} \label{eq:L_expressions_tuples_asymptotics_equations_1}
            L\left(a_{i,1}, \dots, a_{i,h}\right) = -dN+k.
        \end{align}
    \end{enumerate}
    Indeed, given a $t$-tuple $\left(\Lambda_1, \dots, \Lambda_t \right)$ satisfying the conditions of \cref{lem:L_expressions_tuples_upper_bounds}, choosing representatives of each orbit and making them the respective rows of a matrix $\mathcal A$ gives an injective map from the $t$-tuples of interest to the collection of all such matrices $\mathcal A$. Now, we let $m_1$ denote the number of distinct values appearing exactly once in $\mathcal A$. Then the remaining $\ell-m_1$ values appear at least twice, so
    \begin{align*}
        th \geq m_1+2(\ell-m_1) = 2\ell-m_1 \implies m_1 \geq 2\ell-th.
    \end{align*}
    Since a row contains at most $h$ singleton values, the number of rows of $\mathcal A$ with a value that is included exactly once in $\mathcal A$ is at least
    \begin{align} \label{eq:lower_bound_multiplicity_1}
        \lceil (2\ell - th)/h \rceil = \lceil 2\ell/ h \rceil-t.
    \end{align}
    We partition the collection of all such matrices $\mathcal A$ based on the $\ell$ subsets of the entries of $\left(a_{i,j} \right)_{i \in [t], j \in [h]}$ that correspond to the $\ell$ distinct values in the matrix $\mathcal A$. We break into cases based on the values of $\ell$ and $t$. The following discussion is to be understood as having fixed one of the blocks of this partition. We will occasionally assume that we are working with a specific fixed block of this partition.

    \medskip
    
    \paragraph{\textbf{Case 1: $t=1$ and $\ell=h$.}} We let $r=1$ and we let $v_1 = a_{1,1}$.

    \medskip

    \paragraph{\textbf{Case 2: $t \geq 2$ and $h \leq \ell \leq 2h-1$.}} We break into two subcases.
    \begin{enumerate}
        \item If there exist values $v_1, v_2$ such that $v_2$ is in a row that $v_1$ is not in, then we let $r=2$.
        
        \item If no such values $v_1, v_2$ exist, then we necessarily have that $\ell = h$, and all $\ell$ values lie in all $t$ rows of $\mathcal A$. The entries of any such matrix $\mathcal A$ must satisfy, for some $\sigma \notin \mathscr{R}_k$ (due to \eqref{eq:L_expressions_tuples_asymptotics_condition_1}), that
        \begin{align} \label{eq:L_expressions_tuples_asymptotics_system}
            L\left(a_{1,1}, \dots, a_{1,h}\right) = L(a_{1,\sigma(1)}, \dots, a_{1,\sigma(h)}) = -dN+k.
        \end{align}
        We assume (towards a contradiction) that the two equations
        \begin{align} \label{eq:equation_pair}
            & u_1a_{1,1} + \cdots + u_ha_{1,h} = -dN+k;
            & u_{\sigma(1)}a_{1,1} + \cdots + u_{\sigma(h)}a_{1,h} = -dN+k
        \end{align}
        are multiples of each other. By the definition of the redundancy subgroup $\mathscr{R}_k$ and the fact that $\sigma \notin \mathscr{R}_k$, there must exist $j \in [h]$ for which $u_j \neq u_{\sigma(j)}$. Since the equations \eqref{eq:equation_pair} are multiples of each other, it follows that $-dN+k = 0$ and that $u_i \neq u_{\sigma(i)}$ for all $i \in [h]$. Furthermore, it must be that $|u_i| = |u_{\sigma(i)}|$ for all $i \in [h]$ since it would otherwise hold that
        \begin{align*}
            [u_1, \dots, u_h] \neq [u_{\sigma(1)}, \dots, u_{\sigma(h)}]
        \end{align*}
        as multisets. So we have that $u_i = -u_{\sigma(i)}$ for all $i \in [h]$, from which we deduce that $L$ is a balanced linear form, $d = m/2$, and thus $k = dN = mN/2$. In particular, \eqref{eq:equation_pair} holds with $\sigma = \sigma_\rev$. This raises the desired contradiction since $\sigma_\rev \in \mathscr{R}_{mN/2} = \mathscr{R}_k$. 
        
        We deduce that \eqref{eq:equation_pair} gives a full-rank system of linear equations in $(a_{1,1}, \dots, a_{1,h})$, for which there are $O(k^{h-2})$ solutions, agreeing with the guarantee of \cref{lem:L_expressions_tuples_upper_bounds}.
    \end{enumerate}

    \medskip
    
    \paragraph{\textbf{Case 3: $2h \leq \ell \leq 3h-1$.}} We are safe to assume $t \geq 3$. We break into two subcases.
    \begin{enumerate}
        \item If there exist values $v_1, v_2, v_3$ such that $v_2$ is in a row that $v_1$ is not in and $v_3$ is in a row that neither $v_1$ nor $v_2$ are in, then we let $r=3$. It is readily shown that if there exists a value occurring in exactly one row of $\mathcal A$, then such values $v_1, v_2, v_3$ exist. It follows from \eqref{eq:lower_bound_multiplicity_1} that this holds for all $\ell$ such that $2h+1 \leq \ell \leq 3h-1$, and for $\ell = 2h$ if $t = 3$.

        \item Otherwise, it must hold that $\ell = 2h$ and $t=4$. Here, every value occurs exactly twice in $\mathcal A$. Furthermore, these $2h$ values are arranged in such a way that $h$ values lie once in each of two rows of $\mathcal A$, and the other $h$ values lie once in each of the other two rows of $\mathcal A$. Arguing as in Case $2(2)$ and considering both of these pairs of rows, it follows that there are $O(k^{2h-4})$ solutions, which is dominated by the desired guarantee of \cref{lem:L_expressions_tuples_upper_bounds}.
    \end{enumerate}

    \medskip
    
    \paragraph{\textbf{Case 4: $\ell \geq 3h$.}} We are safe to assume $t = 4$. We break into two subcases.
    \begin{enumerate}
        \item If there exist values $v_1, \dots, v_4$ such that for every $j \in [4]$, it holds that $v_j$ is in a row of $\mathcal A$ that no $v_i$ for $i < j$ is in, then we let $r=4$. By \eqref{eq:lower_bound_multiplicity_1}, this can be done for all $\ell \geq 3h+1$, since this yields the existence of at least three rows with a value that is included exactly once in $\mathcal A$.
        
        \item Otherwise, we have that $\ell = 3h$. By \eqref{eq:lower_bound_multiplicity_1}, at least two rows of $\mathcal A$ have values occurring exactly once in $\mathcal A$. The remaining two rows of $\mathcal A$ must therefore contain the same $h$ values. Arguing as in Case $2(2)$ on these remaining two rows, it follows that there are $O(k^{h-2})$ solutions for these remaining two rows. By then choosing the $2h-2$ values not corresponding to those values that occur exactly once in $\mathcal A$, we have a total of $O(k^{3h-4})$ solutions, agreeing with the guarantee of \cref{lem:L_expressions_tuples_upper_bounds}.
    \end{enumerate}
    In all cases in which $r$ was defined, we count the number of such matrices $\mathcal A$ by choosing the $\ell-r$ values that are not the reserved values $v_1, \dots, v_r$. Due to the constraints \eqref{eq:L_expressions_tuples_asymptotics_equations_1}, we have $O(k)$ choices for each of these $\ell-r$ values. After choosing these $\ell-r$ values, the equations \eqref{eq:L_expressions_tuples_asymptotics_equations_1} reduce to a system of $t$ equations in the matrix entries $a_{i,j}$ corresponding to the reserved values $v_1, \dots, v_r$, where the $a_{i,j}$ have nonzero coefficients. These reduced equations then uniquely determine (if a solution exists) the values of these remaining $r$ values $v_1, \dots, v_r$. It is readily checked that in all of the above cases, we recover the bounds given in the statement of \cref{lem:L_expressions_tuples_upper_bounds}. The fact that there are $O(1)$ blocks in this partition completes the proof.
\end{proof}

\section*{References}

\printbibliography[heading=none]

@book{alon2016probabilistic,
  title={The probabilistic method},
  author={Alon, Noga and Spencer, Joel H},
  year={2016},
  publisher={John Wiley \& Sons}
}

@article{hegarty2009almost,
  title={When almost all sets are difference dominated},
  author={Hegarty, Peter and Miller, Steven J.},
  journal={Random Structures \& Algorithms},
  volume={35},
  number={1},
  pages={118--136},
  year={2009},
  publisher={Wiley Online Library}
}

@article{arratia1989two,
  author = {R. Arratia and L. Goldstein and L. Gordon},
    title = {{Two moments suffice for Poisson approximations: the Chen-Stein method}},
    volume = {17},
    journal = {The Annals of Probability},
    number = {1},
    publisher = {Institute of Mathematical Statistics},
    pages = {9 -- 25},
    year = {1989},
}

@article{vu2002concentration,
  title={Concentration of non-Lipschitz functions and applications},
  author={Vu, Van H},
  journal={Random Structures \& Algorithms},
  volume={20},
  number={3},
  pages={262--316},
  year={2002},
  publisher={Wiley Online Library}
}

@article{martin2006many,
  title={Many sets have more sums than differences},
  author={Martin, Greg and O'Bryant, Kevin},
  journal={Additive Combinatorics, CRM Proc. Lecture Notes},
  volume={43},
  publisher={Amer. Math. Soc.},
  year={2007}
}

@book{andrews1998theory,
  title={The theory of partitions},
  author={Andrews, George E},
  year={1998},
  publisher={Cambridge University Press}
}

@article{erdos1941distribution,
author = {Paul Erd\H{o}s and Joseph Lehner},
title = {{The distribution of the number of summands in the partitions of a positive integer}},
volume = {8},
journal = {Duke Mathematical Journal},
number = {2},
publisher = {Duke University Press},
pages = {335 -- 345},
year = {1941}
}

@article{arratia1990poisson,
  author = {Richard Arratia and Larry Goldstein and Louis Gordon},
title = {{Poisson approximation and the Chen-Stein method}},
volume = {5},
journal = {Statistical Science},
number = {4},
publisher = {Institute of Mathematical Statistics},
pages = {403 -- 424},
year = {1990}
}

@article{nathanson2006problems,
  title={Problems in additive number theory, I},
  author={Nathanson, Melvyn B},
  journal={Additive combinatorics, CRM Proc. Lecture Notes},
  volume={43},
  publisher={Amer. Math. Soc.},
  pages={263--270},
  year={2007}
}

@article{iyer2012generalized,
  title={Generalized more sums than differences sets},
  author={Iyer, Geoffrey and Lazarev, Oleg and Miller, Steven J. and Zhang, Liyang},
  journal={Journal of Number Theory},
  volume={132},
  number={5},
  pages={1054--1073},
  year={2012},
  publisher={Elsevier}
}

@article{nathanson2007binary,
  title={Binary linear forms over finite sets of integers},
  author={Nathanson, Melvyn B and O'Bryant, Kevin and Orosz, Brooke and Ruzsa, Imre and Silva, Manuel},
  journal={Acta Arithmetica},
  volume={129},
  number={4},
  pages={341-361},
  year={2007}
}

@book{janson2011random,
  title={Random graphs},
  author={Janson, Svante and Luczak, Tomasz and Rucinski, Andrzej},
  year={2011},
  publisher={John Wiley \& Sons}
}

@article{stanley2016some,
  title={Some asymptotic results on q-binomial coefficients},
  author={Stanley, Richard P and Zanello, Fabrizio},
  journal={Annals of Combinatorics},
  volume={20},
  pages={623--634},
  year={2016},
  publisher={Springer}
}

@article{janson2013euler,
  title={Euler-Frobenius numbers and rounding},
  author={Janson, Svante},
  journal={Online J. Analytic Comb.},
  volume={8},
  year={2013}
}

@book{feller1991introduction,
  title={An introduction to probability theory and its applications, Volume 2},
  author={Feller, William},
  volume={81},
  year={1991},
  publisher={John Wiley \& Sons}
}

@book{laplace1814theorie,
  title={Th{\'e}orie analytique des probabilit{\'e}s},
  author={Laplace, Pierre Simon},
  year={1814},
  publisher={Courcier}
}

@article{sylvester1878xxv,
  title={XXV. Proof of the hitherto undemonstrated fundamental theorem of invariants},
  author={Sylvester, James Joseph},
  journal={The London, Edinburgh, and Dublin Philosophical Magazine and Journal of Science},
  volume={5},
  number={30},
  pages={178--188},
  year={1878},
  publisher={Taylor \& Francis}
}

@article{melczer2020counting,
  title={Counting partitions inside a rectangle},
  author={Melczer, Stephen and Panova, Greta and Pemantle, Robin},
  journal={SIAM Journal on Discrete Mathematics},
  volume={34},
  number={4},
  pages={2388--2410},
  year={2020},
  publisher={SIAM}
}

@article{zhong2022combinatorial,
  title={A combinatorial proof of the unimodality and symmetry of weak composition rank sequences},
  author={Zhong, Yueming},
  journal={Annals of Combinatorics},
  pages={1--15},
  year={2022},
  publisher={Springer}
}

@book{barbour1992poisson,
  title={Poisson approximation},
  author={Barbour, Andrew D and Holst, Lars and Janson, Svante},
  year={1992},
  publisher={The Clarendon Press Oxford University Press}
}

@article{kim2000concentration,
  title={Concentration of multivariate polynomials and its applications},
  author={Kim, Jeong Han and Vu, Van H},
  journal={Combinatorica},
  volume={20},
  number={3},
  pages={417--434},
  year={2000},
  publisher={Budapest: Akademiai Kiado,[1981-}
}

@article{vu2000new,
  title={New bounds on nearly perfect matchings in hypergraphs: higher codegrees do help},
  author={Vu, Van H},
  journal={Random Structures \& Algorithms},
  volume={17},
  number={1},
  pages={29--63},
  year={2000},
  publisher={Wiley Online Library}
}

@book{aldous2013probability,
  title={Probability approximations via the Poisson clumping heuristic},
  author={Aldous, David},
  volume={77},
  year={2013},
  publisher={Springer Science \& Business Media}
}

@article{friedgut1999sharp,
  title={Sharp Thresholds of Graph Properties, and the k-Sat Problem},
  author={Friedgut, Ehud and Bourgain, Jean},
  journal={Journal of the American Mathematical Society},
  pages={1017--1054},
  year={1999},
  publisher={JSTOR}
}

\end{document}